\documentclass[11pt]{amsart}

\usepackage{amssymb, amsfonts, amsmath, amsthm, stmaryrd}
\usepackage{xcolor} 
 
\newtheorem{theorem}{Theorem}[section]
\newtheorem{lemma}[theorem]{Lemma}
\newtheorem{proposition}[theorem]{Proposition}
\newtheorem{corollary}[theorem]{Corollary}
\newtheorem{prop-and-def}[theorem]{Proposition and Definition}

\theoremstyle{definition}
\newtheorem{definition}[theorem]{Definition}
\newtheorem{notation}[theorem]{Notation}
\newtheorem{remark}[theorem]{Remark}
\newtheorem{definition-and-remark}[theorem]{Definition and Remark}
\newtheorem{remark-and-notation}[theorem]{Remark and Notation}
\newtheorem{notation-and-remark}[theorem]{Notation and Remark}
\newtheorem{example}[theorem]{Example}
\newtheorem{ad-hoc}[theorem]{ }
\numberwithin{equation}{section}

\newcommand{\cA}{ {\mathcal A} }

\newcommand{\cC}{ {\mathcal C} }
\newcommand{\cD}{ {\mathcal D} }
\newcommand{\Dalg}{ \cD_{\mathrm{alg}} }
\newcommand{\cF}{ {\mathcal F} }
\newcommand{\cG}{ {\mathcal G} }
\newcommand{\cGtild}{ \widetilde{\cG} }
\newcommand{\cGtildctoc}{ \cGtild_{\mathrm{c-c}} }
\newcommand{\cGtildctom}{ \cGtild_{\mathrm{c-m}} }
\newcommand{\cJ}{ {\mathcal J} }

\newcommand{\cL}{ {\mathcal L} }
\newcommand{\cM}{ {\mathcal M} }
\newcommand{\cN}{ {\mathcal N} }
\newcommand{\cP}{ {\mathcal P} }

\newcommand{\cT}{ {\mathcal T} }

\newcommand{\cZ}{ {\mathcal Z} }

\newcommand{\bB}{ {\mathbb B} }
\newcommand{\bC}{ {\mathbb C} }
\newcommand{\bG}{ {\mathbb G} }
\newcommand{\fM}{ {\mathfrak M} }
\newcommand{\fMcA}{ \fM_{ { }_{\cA} } }
\newcommand{\fMcM}{ \fM_{ { }_{\cM} } }
\newcommand{\bN}{ {\mathbb N} }
\newcommand{\bR}{ {\mathbb R} }
\newcommand{\bX}{ {\mathbb X} }
\newcommand{\bZ}{ {\mathbb Z} }

\newcommand{\uBeta}{ \underline{\beta} }
\newcommand{\uBetat}{ \uBeta^{(t)} }

\newcommand{\uEta}{ \underline{\eta} }
\newcommand{\uGamma}{ \underline{\gamma} }
\newcommand{\uKappa}{ \underline{\kappa} }

\newcommand{\uLambda}{ \underline{\lambda} }
\newcommand{\uPhi}{ \underline{\varphi} }
\newcommand{\uPsi}{ \underline{\psi} }
\newcommand{\uRho}{ \underline{\rho} }
\newcommand{\uTheta}{ \underline{\theta} }
\newcommand{\uz}{ \underline{z} }

\newcommand{\Fbcm}{ g_{\mathrm{bc-m}} }
\newcommand{\Ffcm}{ g_{\mathrm{fc-m}} }
\newcommand{\Fmcm}{ g_{\mathrm{mc-m}} }
\newcommand{\Ffcbc}{ g_{\mathrm{fc-bc}} }

\newcommand{\Fmcbc}{ g_{\mathrm{mc-bc}} }

\newcommand{\Blocks}{ \mathrm{Blocks} }
\newcommand{\Blocksplus}{ \Blocks^{+} }

\newcommand{\EC}{\mathcal{EC} }
\newcommand{\Int}{ \mbox{Int} }
\newcommand{\Kr}{ \mathrm{Kr} }
\newcommand{\NCirr}{ NC_{\mathrm{irr}} }
\newcommand{\Sym}{ \mbox{Sym} }

\newcommand{\Id}{ \mbox{Id} }

\newcommand{\alphans}{ (\alpha_n)_{n=1}^{\infty} }

\newcommand{\lambdans}{ ( \lambda_n )_{n=1}^{\infty} }
\newcommand{\thetans}{ ( \theta_n )_{n=1}^{\infty} }
\newcommand{\ee}{\varepsilon}
\newcommand{\eeps}{\epsilon}
\newcommand{\NCint}{ NC^{(2)} }
\newcommand{\pirrc}{ \overline{\pi}^{\mathrm{irr}} }
\newcommand{\sigmairrc}{ \overline{\sigma}^{\mathrm{irr}} }
\newcommand{\Vleft}{ V_{\mathrm{left}} }
\newcommand{\Vright}{  V_{\mathrm{right}} }

\newcommand{\betat}{ \beta^{(t)} }
\newcommand{\innblocks}{ \mathrm{inner} }
\newcommand{\outblocks}{ \mathrm{outer} }

\newcommand{\chicheck}{ \check{\chi} }
\newcommand{\epsiloncheck}{ \check{\epsilon} }

\newcommand{\oneA}{ 1_{{ }_{\cA}} }
\newcommand{\oneT}{ 1_{{ }_{\cT}} }
\newcommand{\Prim}{ \mathrm{Prim} }
\newcommand{\oneSym}{ 1_{ { }_{\mathrm{Sym}} } }
\newcommand{\oneZ}{ 1_{{ }_{\cZ}} }

\newcommand{\freestar}{\boxast}

\newcommand{\dotinf}{ \stackrel{\mathrm{inf}}{\odot} }

\newcommand{\epsilonSym}{\varepsilon_{\mathrm{Sym}}}
\newcommand{\epsilonT}{\varepsilon_{\mathcal{T}}}
\newcommand{\DeltaSym}{\Delta_{\mathrm{Sym}}}
\newcommand{\DeltaT}{\Delta_{\mathcal{T}}}

\title[Multiplicative and semi-multiplicative functions on 
NC($n$)]{Multiplicative and semi-multiplicative functions on 
non-crossing partitions, and relations to cumulants}

\author[A. Celestino]{Adrian Celestino}
\address{Adrian Celestino: Department of Mathematical Sciences,
Norwegian University of Science and Technology (NTNU), 
7491 Trondheim, Norway.}
\email{adrian.celestino@ntnu.no}

\author[K. Ebrahimi-Fard]{Kurusch Ebrahimi-Fard}
\thanks{KEF: research supported by the Research Council of Norway
through project 302831 
{\it{Computational Dynamics and Stochastics on Manifolds}} (CODYSMA)
as well as by the project Pure Mathematics in Norway, funded by 
Trond Mohn Foundation and Troms{\o} Research Foundation.}
\address{Kurusch Ebrahimi-Fard: Department of Mathematical Sciences, 
Norwegian University of Science and Technology (NTNU),
7491 Trondheim, Norway.}
\email{kurusch.ebrahimi-fard@ntnu.no}

\author[A. Nica]{Alexandru Nica}
\thanks{AN: research supported by a Discovery Grant from 
	NSERC, Canada.}
\address{Alexandru Nica: Department of Pure Mathematics, 
	University of Waterloo, Ontario, Canada.}
\email{anica@uwaterloo.ca}

\author[D. Perales]{Daniel Perales}
\thanks{DP: Supported by CONACyT (Mexico) via the scholarship 714236.}
\address{Daniel Perales: Department of Pure Mathematics,
    University of Waterloo, Ontario, Canada.}
\email{dperales@uwaterloo.ca}

\author[L. Witzman]{Leon Witzman}
\address{Leon Witzman: Department of Pure Mathematics,
    University of Waterloo, Ontario, Canada.}
\email{lwitzman@uwaterloo.ca}

\begin{document}

\begin{abstract}
We consider the group $( \cG, *)$ of unitized multiplicative 
functions in the incidence algebra of non-crossing partitions,
where ``$*$'' denotes the convolution operation.  We introduce
a larger group $( \cGtild, * )$ of unitized functions from the 
same incidence algebra, which satisfy a weaker 
{\em semi-multiplicativity} condition.  The natural 
action of $\cGtild$ on sequences of multilinear functionals of 
a non-commutative probability space captures the combinatorics 
of transitions between moments and some brands of cumulants 
that are studied in the non-commutative probability literature. 
We use the framework of $\cGtild$ in order to explain why the 
multiplication of free random variables can be very nicely 
described in terms of Boolean cumulants and more generally in 
terms of $t$-Boolean cumulants, a one-parameter interpolation 
between free and Boolean cumulants arising from work of 
Bo{\.z}ejko and Wysoczanski.

It is known that the group $\cG$ can be naturally identified 
as the group of characters of the Hopf algebra Sym of symmetric
functions.  We show that $\cGtild$ can also be identified as group
of characters of a Hopf algebra $\cT$, which is an incidence Hopf
algebra in the sense of Schmitt.  Moreover, the inclusion 
$\cG \subseteq \cGtild$ turns out to be the dual of a natural 
bialgebra homomorphism from $\cT$ onto Sym.
\end{abstract}

\maketitle

\section{Introduction}

\subsection{The group \boldmath{$\cG$} of unitized 
multiplicative functions on \boldmath{$NC(n)$}'s.}

$\ $

\noindent
The idea of studying the convolution of multiplicative functions
defined on the set of all intervals of a ``coherent'' collection 
of lattices $( \cL_n )_{n=1}^{\infty}$ goes back to the 1960's 
work of Rota and collaborators, e.g.~in \cite{DoRoSt1972}.  The 
phenomenon which prompts this study is that, in a number of 
important examples: for every $\pi \leq \sigma$ in an $\cL_n$, 
the sublattice $[ \pi, \sigma ] := 
\{ \rho \in \cL_n \mid \pi \leq \rho \leq \sigma \}$ of $\cL_n$
is canonically isomorphic to a direct product,
\begin{equation}   \label{eqn:intro1}
[ \pi, \sigma ] \approx
\cL_1^{p_1} \times \cdots \times \cL_n^{p_n},
\ \mbox{ with $p_1, \ldots , p_n \geq 0$}. 
\end{equation}
A function 
$f : \sqcup_{n=1}^{\infty}  \{ ( \pi , \sigma ) \mid
\pi, \sigma \in \cL_n, \, \pi \leq \sigma \} \to \bC$
is declared to be multiplicative when there exists a sequence 
of complex numbers $\alphans$ such that, for $\pi, \sigma$ and 
non-negative integers $p_1, \ldots , p_n$ 
as in (\ref{eqn:intro1}), one has
$f( \pi , \sigma ) := \alpha_1^{p_1} \cdots \alpha_n^{p_n}$.

In the present paper we are interested in the case when 
$\cL_n$ is the lattice $NC(n)$ of non-crossing partitions 
of $\{ 1, \ldots , n \}$, endowed with the partial order by
reverse refinement.  
In the 1990's it was found by Speicher \cite{Sp1994} that, 
when considered in connection to the $NC(n)$'s, the convolution 
of multiplicative functions plays an essential role in the 
combinatorial development of free probability.
For the purposes of the present paper it is convenient to focus 
on the set $\cG$ consisting of multiplicative functions 
on the $NC(n)$'s where the sequence $\alphans$ defining the function
has $\alpha_1 = 1$.  Then $\cG$ is a group under convolution.
While this is a self-standing structure, which can be considered 
without any knowledge of what is a non-commutative probability space, 
it nevertheless turns out that the group operation of $\cG$ encapsulates 
the combinatorics of the multiplication of free random variables; for a 
detailed presentation of how this goes, we refer to Lectures 14 
and 18 of the monograph \cite{NiSp2006}.

$\ $

\subsection{The group \boldmath{$\cGtild$} of unitized 
semi-multiplicative functions on \boldmath{$NC(n)$}'s.}

$\ $

\noindent
In the case of the lattices $NC(n)$, the canonical isomorphism 
indicated in (\ref{eqn:intro1}) is obtained by combining
two kinds of lattice isomorphisms, as follows.

\vspace{6pt}

{\em First kind of isomorphism:} one observes that for
every $\pi \leq \sigma$ in some $NC(n)$, the interval 
$[ \pi, \sigma ] \subseteq NC(n)$ is canonically
isomorphic to a direct product of intervals of the 
form $[ \theta , 1_k]$, with $\theta \in NC(k)$ for 
some $1 \leq k \leq n$, and where $1_k$ is the maximal 
element of $NC(k)$, i.e.~it is the 
partition of $\{ 1, \ldots , k \}$ into a single block.

\vspace{6pt}

{\em Second kind of isomorphism:} for every $k \geq 1$ and 
$\theta \in NC(k)$ one finds $[ \theta, 1_k ]$ to be  
canonically isomorphic to a direct product
$NC(1)^{q_1} \times \cdots \times NC(k)^{q_k}$, with
$q_1, \ldots , q_k \geq 0$. 

\vspace{6pt}

These two kinds of isomorphisms will be reviewed
precisely as soon as the notation is set for them, 
cf.~Remark \ref{rem:12} below.  But we signal right now
that our main point is this:
\begin{equation}    \label{eqn:intro2}
\begin{array}{ll}
\vline &  
\mbox{It is worth studying convolution for functions}  \\
\vline &
g: \sqcup_{n=1}^{\infty} \{ ( \pi , \sigma ) \mid
\pi, \sigma \in NC(n), \, \pi \leq \sigma \} \to \bC      
\mbox{ which }                         \\
\vline & 
\mbox{are only required to be multiplicative with respect}     \\
\vline & 
\mbox{to the {\em first kind} of isomorphism mentioned above.}
\end{array}
\end{equation}
We will use the term {\em semi-multiplicative} for a function 
$g$ as in (\ref{eqn:intro2}), and we will denote
\begin{equation}   \label{eqn:intro3}
\cGtild = \Bigl\{ 
g: \sqcup_{n=1}^{\infty} \{ ( \pi , \sigma ) \mid
\pi \leq \sigma \mbox{ in } NC(n) \} \to \bC    
\begin{array}{ll}
\vline & \mbox{$g$ is semi-multiplicative and }  \\
\vline & g( \pi , \pi ) = 1, 
\ \forall \, \pi \in \sqcup_{n=1}^{\infty} NC(n)
\end{array}   \Bigr\} .
\end{equation}
It turns out that $\cGtild$ is a group under convolution.  
This group and some of its subgroups (in particular the 
subgroup $\cG$ from Section 1.1) are the main players 
in the considerations of the present paper.  The benefits that 
come from studying $\cGtild$ are presented in the next subsections
of this Introduction.

$\ $

\subsection{Relations of \boldmath{$\cGtild$} with moments 
and with some brands of cumulants.}

$\ $

\noindent
Consider now the framework of a non-commutative probability 
space $( \cA , \varphi )$, where $\cA$ is a unital associative 
algebra over $\bC$ and $\varphi \colon \cA \to \bC$
is a linear functional such that the algebra unit is mapped
to one ($\varphi(1_\cA) =1$), and look at
\[
\fMcA := \{ \uPsi \mid 
\uPsi = ( \psi_n : \cA^n \to \bC )_{n=1}^{\infty},
\mbox{ where $\psi_n$ is an $n$-linear functional} \}.
\]
In $\fMcA$ we have a special element 
$\uPhi = ( \varphi_n )_{n=1}^{\infty}$ called 
{\em family of moment functionals} of $( \cA , \varphi )$,
where $\varphi_n : \cA^n \to \bC$ is defined by putting
$\varphi_n (a_1, a_2, \ldots , a_n) :=
\varphi (a_1 a_2 \cdots a_n)$ for all
$n \geq 1 \mbox{ and } a_1, \ldots , a_n \in \cA$.
Then in $\fMcA$ there also are several families of 
{\em cumulant functionals} which relate to $\uPhi$ via 
summation formulas over non-crossing partitions, and 
receive constant attention in the research literature 
on non-commutative probability: free cumulants, Boolean 
cumulants, monotone cumulants (see e.g.~\cite{ArHaLeVa2015}).
In this paper we also devote some attention to a continuous 
interpolation between Boolean and free cumulants, which 
we refer to as $t$-Boolean cumulants, and are arising from 
the work of Bo\.zejko and Wysoczanski \cite{BoWy2001} -- the 
case $t=0$ gives Boolean cumulants and the case $t=1$ gives 
free cumulants.

The group $\cGtild$ has a natural action on $\fMcA$,
which is discussed in detail in Section 6 of the paper.
This action captures the transitions between moment
functionals and the brands of cumulants mentioned above,
and as a consequence it also captures the formulas for 
direct transitions between two such brands of cumulants.  
We mention that the study of direct transitions between
different brands of cumulants goes back to the work of
Lehner \cite{Le2002}, and was thoroughly pursued in 
\cite{ArHaLeVa2015}.  The benefit of using the group 
$\cGtild$ is that it offers an efficient framework for
streamlining calculations related to various moment-cumulant
and inter-cumulant formulas.

For full disclosure, we reiterate here a fact implicitly
present in the above discussion, namely that this paper only
addresses brands of cumulants which live within the world of
non-crossing partitions.  It is an interesting direction of
future research to clarify how some of the considerations of 
the paper can be adjusted to the setting of full lattices of 
partitions of sets $\{ 1, \ldots , n \}$ (where crossings are
allowed).  The examination of this direction has been started
in Chapters 5-7 of the thesis \cite{Pe2022}, and promises to
extend the results of the present paper to a setting
which will also include the ``classical'' cumulants 
commonly used in the probability literature.

Returning to the group $\cGtild$, our next point is that
it is possible to identify precisely some notions 
of what it means for a function $h \in \cGtild$ to be of 
{\em cumulant-to-moment type}, and what it means for a
$g \in \cGtild$ to be of {\em cumulant-to-cumulant} type.
This is done in Section 7 of the paper.  Denoting
\[
\cGtildctom = \{ h \in \cGtild \mid h
\mbox{ is of cumulant-to-moment type} \} \ \mbox{ and}
\]
\[
\cGtildctoc = \{ g \in \cGtild \mid g
\mbox{ is of cumulant-to-cumulant type} \},
\]
we prove in Section 8 that $\cGtildctoc$ is a subgroup of 
$( \cGtild , * )$, while $\cGtildctom$ is a right coset 
of $\cGtildctoc$.  The latter statement means
that we have
\begin{equation}   \label{eqn:intro4}
\cGtildctom = \cGtildctoc * h
:= \{ g * h \mid g \in \cGtildctoc \},
\end{equation}
for no matter what $h \in \cGtildctom$ we choose to 
fix.  An easy choice is to fix the $h$ which is 
identically equal to $1$; this is indeed a function in
$\cGtildctom$, and encodes the transition from free 
cumulants to moment functionals.  
However, as pointed out in Section 8.2 of the paper, 
it seems to be more advantageous (both for writing 
proofs and for finding applications) if in (\ref{eqn:intro4})
we use a different choice for $h$, and pick the function
which encodes the transition from Boolean cumulants to moments.

$\ $

\subsection{The 1-parameter subgroup 
\boldmath{$\{ u_q \mid q \in \bR \}$} of 
\boldmath{$\cGtildctoc$}.}

$\ $

\noindent
The method we use for proving (\ref{eqn:intro4}) draws attention 
to the subgroup of $\cGtild$ generated by the function which
encodes transition between free cumulants and Boolean cumulants.
In the notation system used throughout the paper, the latter
function is denoted as $\Ffcbc$.  The subgroup
$\{ \Ffcbc^p \mid p \in  \bZ \} \subseteq \cGtildctoc$ can be
naturally incorporated into a continuous 1-parameter subgroup 
of $\cGtildctoc$, which we denote as $\{ u_q \mid q \in \bR \}$ 
(thus $u_q = \Ffcbc^q$ for $q \in \bZ$).
Working with the $u_q$'s nicely streamlines the various
formulas involving $t$-Boolean cumulants, and in particular 
gives an easy way (cf.~Corollary \ref{cor:85} below) to 
write the transition formula between $s$-Boolean cumulants 
and $t$-Boolean cumulants for distinct values $s,t \in \bR$.

In Section 10 we prove that every $u_q$
belongs to the normalizer of the subgroup $\cG \subseteq \cGtild$
from Section 1.1:
\begin{equation}   \label{eqn:intro5}
\Bigl( q \in \bR, \ f \in \cG \Bigr) \ \Rightarrow
\ u_q^{-1} * f * u_q \in \cG.
\end{equation}
This is a non-trivial fact, as the $u_q$'s are coming
from $\cGtildctoc$, and there is no obvious direct
connection between $\cGtildctoc$ and $\cG$ -- it is, in
any case, easy to check that the intersection 
$\cG \cap \cGtildctoc$ only contains the unit $e$ of 
$\cGtild$, while the intersection of $\cG$ with the coset 
$\cGtildctom$ only contains the function $h$ which is 
constantly equal to $1$.
 
$\ $

\subsection{Multiplication of free random variables,
in terms of \boldmath{$t$}-Boolean cumulants.}

$\ $

\noindent
The result obtained in (\ref{eqn:intro5}) can be used in order
to give a neat explanation of the intriguing fact that 
the multiplication of freely independent random variables 
is nicely described in terms of Boolean cumulants
(who aren't a priori meant to be related to free probability).

We find it convenient to place the discussion in the more general
framework of $t$-Boolean cumulants.  So let us consider a 
non-commutative probability space $( \cA , \varphi )$, let $x,y$ 
be two freely independent elements of $\cA$, and let $t$ be a 
parameter with values in $\bR$.  What happens is that the formula 
describing the $t$-Boolean cumulants of the product $xy$ in terms 
of the separate $t$-Boolean cumulants of $x$ and of $y$ is 
{\em one and the same}, no matter what value of $t$ we are using.
More precisely: denoting the family of $t$-Boolean cumulants as 
$\uBeta^{(t)} = ( \betat_n : \cA^n \to \bC )_{n=1}^{\infty} \in \fMcA$,
the formula for $t$-Boolean cumulants of $xy$ says that:
\begin{equation}  \label{eqn:intro6}
\betat_n (xy, \ldots , xy) = \sum_{\pi \in NC(n)} 
\betat_{\pi} (x, \ldots , x) \cdot
\betat_{\Kr ( \pi ) } (y, \ldots , y),
\ \ \forall \, n \geq 1.
\end{equation}
Equation (\ref{eqn:intro6}) contains some notation that has
to be clarified (such as what is the multilinear functional 
$\betat_{\pi} : \cA^n \to \bC$ associated to a partition 
$\pi \in NC(n)$, and the fact that every $\pi \in NC(n)$ has a 
complement $\Kr ( \pi ) \in NC(n)$).  All the necessary notation 
will be reviewed in the body of the paper; the reason for giving 
the formula (\ref{eqn:intro6}) at this point is so that we can 
explain our way of proving it.

Our approach can be summarized as follows.  For every $t \in \bR$,
consider the statement:
\[
\mbox{(Statement $t$)} 
\hspace{0.8cm}  \left\{ \begin{array}{c}
\mbox{ The formula (\ref{eqn:intro6}) holds true for this $t$}  \\
\mbox{ and for any freely independent elements $x,y$ in }   \\
\mbox{ some non-commutative probability space $( \cA , \varphi )$}
\end{array}  \right\} .
\hspace{0.8cm}  \mbox{$\ $}
\]
The action by conjugation of the $u_q$'s on multiplicative 
functions allows us to prove the following fact:

\vspace{6pt}

{\em Fact.}  If there exists a $t_o \in \bR$ for which 
(Statement $t_o$) is true, then it follows 

\hspace{1cm} that 
(Statement $t$) is true for all $t \in \bR$.

\vspace{6pt}

But it has been known since the 1990's that 
(Statement $t_o$) is true for $t_o = 1$ -- this is the very basic
description of multiplication of free random variables in terms
of free cumulants, cf.~\cite[Theorem 14.4]{NiSp2006}.  The above 
``Fact'' then assures us that (Statement $t$) is indeed true for 
all $t$; in particular, at $t=0$ we retrieve the result (first 
found in \cite{BeNi2008} via a direct combinatorial analysis) 
about how multiplication of free random variables is described
in terms of Boolean cumulants.

$\ $

\subsection{Hopf algebra aspects.}

$\ $

\noindent
A significant fact about the group $\cG$ from Section 1.1, 
observed in \cite{MaNi2010},
is that it can be naturally identified as the group of characters of 
the Hopf algebra Sym of symmetric functions.  When combined with the 
log map for characters of Sym, this identification retrieves the 
celebrated $S$-transform of Voiculescu \cite{Vo1987}, which is the 
most efficient tool for computing distributions of products of free 
random variables.

In analogy to that, we present the construction of a Hopf 
algebra $\cT$, done in such a way that the character group 
$\bX ( \cT )$ is naturally isomorphic to $\cGtild$.  $\cT$ can 
be identified as an incidence Hopf algebra, 
cf.~\cite{Schm1994, Ein2010},
and is also closely related to one of the Hopf 
algebras studied in the recent paper \cite{EFFoKoPa2019}. 
Moreover, we find that the inclusion of groups 
$\cG \subseteq \cGtild$ is precisely (in view of the 
canonical isomorphisms $\cG \approx \bX ( \mathrm{Sym} )$ and 
$\cGtild \approx \bX ( \cT )$) the dual 
$\Psi^{*} : \bX ( \mbox{Sym} ) \to \bX ( \cT )$ of a natural 
bialgebra homomorphism $\Psi : \cT \to \mathrm{Sym}$ provided 
by the Kreweras complementation map.

A promising feature of $\cT$ is that its antipode map
can, in principle, serve as a universal tool for inversion in 
formulas that relate moments to cumulants, or relate different 
brands of cumulants living in the $NC(n)$ framework.  In 
Section 13 of the paper we examine the antipode of $\cT$ and in
particular we identify (Theorem \ref{thm:1210}) a cancellation-free 
formula for how the antipode works, described in terms of a suitable 
notion of ``efficient chains'' in the lattices $NC(n)$.

$\ $

\subsection{Organising of the paper.}

$\ $

\noindent
Following to the present Introduction, the sections of the paper 
can be roughly divided into four parts. 

\vspace{6pt}

$\bullet$
In the first part, Sections 2-5, we establish some basic relevant facts 
concerning the group $( \cGtild , * )$.  More precisely: after setting 
some background and notation in Section 2, we introduce $\cGtild$ in
Section 3. Then in Section 4 (Theorem \ref{thm:33}) we prove that 
$\cGtild$ is indeed a group under convolution.  The review of the 
smaller group $\cG$ and some discussion around the inclusion 
$\cG \subseteq \cGtild$ appears in Section 5.

\vspace{6pt}

$\bullet$
In the second part, Sections 6-8, we demonstrate the relevance of 
$\cGtild$ to the study of non-commutative cumulants.  This comes into 
the picture via a natural action which a function $g \in \cGtild$ 
has on sequences of multilinear functionals on a non-commutative 
probability space.  This action is presented in Section 6.  Then in
Section 7 we look at specific examples of cumulants and, based on them, 
we identify what it means for $g \in \cGtild$ to encode transitions 
of ``moment-to-cumulant'' type or of ``cumulant-to-cumulant'' type.  

In Section 8 we prove (Propositions \ref{prop:72} and \ref{prop:74})
that, as announced in the above subsection 1.3: the set $\cGtildctoc$ 
of functions of cumulant-to-cumulant type is a subgroup 
of $\cGtild$, and the set $\cGtildctom$ of functions of
cumulant-to-moment type is a right coset of $\cGtildctoc$.
The method of proof of Proposition \ref{prop:74} points to the 
importance of the function $\Fbcm \in \cGtild$ which encodes the 
transition from Boolean cumulants to moments; as an application, 
we show (Section 8.3) how this leads to the known ``rule of thumb''
that Boolean cumulants are the easiest cumulants to relate
to, when we start from a moment-cumulant formula given for some 
other brand of cumulants.

\vspace{6pt}

$\bullet$
In the third part, Sections 9-11, we present the results 
announced in the above subsections 1.4 and 1.5.  Section 9 
discusses the 1-parameter subgroup 
$\{ u_q \mid q \in \bR \} \subseteq \cGtildctoc$ and its
applications to $t$-Boolean cumulants. 
In Section 10 we prove (Theorem \ref{thm:91}) that the 
$u_q$'s normalize $\cG$, and in Section 11 we flesh out 
the plan outlined in Section 1.5 for how to derive the 
description of the multiplication of free random variables
in terms of $t$-Boolean cumulants.

\vspace{6pt} 

$\bullet$
Finally the fourth part, Sections 12-14, discusses 
Hopf algebra aspects of the study of $\cGtild$.  In 
Section 12 we provide a detailed explicit description of
the Hopf algebra $\cT$ and we put into evidence the canonical
isomorphism $\bX ( \cT ) \approx \cGtild$.  Section 13 is 
devoted to studying the antipode of $\cT$, and in Section 14
we present (Theorem \ref{thm:135}) the bialgebra homomorphism 
$\Psi : \cT \to \mathrm{Sym}$ which dualizes the inclusion 
of $\cG$ into $\cGtild$.

$\ $

\section{Background and notation}

\subsection{Some \boldmath{$NC(n)$} terminology.} 

$\ $

\noindent
We will assume the reader is familiar with the lattices of 
non-crossing partitions $NC(n)$, and we will follow standard 
notation commonly used in connection to them, as presented for 
instance in Lectures 9 and 10 of \cite{NiSp2006}.  Here is a 
quick review of some notational highlights.

\vspace{6pt}

\begin{notation}  \label{def:11}
Let $n$ be a positive integer.

\noindent
$1^o$  The number of blocks of a partition $\pi \in NC(n)$ is 
denoted as $| \pi |$.  One can meaningfully define what it means
for two blocks of $\pi$ to be nested inside each other; consequently,
one gets a notion of {\em outer block} (a block $V \in \pi$ which 
is not nested into anything else) versus {\em inner block} (a block 
which is not outer).  We will use the notation
$\innblocks ( \pi )$ and $\outblocks ( \pi )$ for the numbers 
of inner respectively outer blocks of $\pi$.  We thus 
have $| \pi | = \innblocks ( \pi ) + \outblocks ( \pi )$, with 
$\innblocks ( \pi ) \geq 0$ and $\outblocks ( \pi ) \geq 1$.
Note that
\begin{equation}    \label{eqn:11a}
\bigl\{ \pi \in NC(n) \mid \innblocks ( \pi ) = 0 \bigr\} 
\ =: \ \Int (n)
\end{equation}
is the set of all interval partitions
of $\{ 1, \ldots , n \}$; these are the partitions $\pi \in NC(n)$ where 
every block $V \in \pi$ is an interval of $\{ 1, \ldots , n \}$. 

\vspace{6pt}

\noindent
$2^o$  The main partial order we consider on $NC(n)$ is the one  
given by reverse refinement: for $\pi , \sigma \in NC(n)$ we write 
``$\pi \leq \sigma$'' to mean that every block of $\sigma$ is a union 
of blocks of $\pi$.  We will also make occasional use of two other 
partial orders on $NC(n)$, denoted $\ll$ and $\sqsubseteq$, which are reviewed 
in the next subsection.

\vspace{6pt}

\noindent
$3^o$  We denote by $0_n \in NC(n)$ the partition with $n$ blocks of
cardinality $1$, and by $1_n \in NC(n)$ the partition consisting of a single 
block.  These are the minimal respectively maximal element of the partially ordered set (poset) $(NC(n), \leq )$.

\vspace{6pt}

\noindent
$4^o$  Every $\pi \in NC(n)$ has a {\em Kreweras complement} 
$\Kr ( \pi ) \in NC(n)$, and the map 

\noindent
$\Kr : NC(n) \to NC(n)$ 
so defined is an anti-automorphism of the poset $( NC(n) , \leq )$. 
For the description of how $\Kr ( \pi )$ is constructed and for some 
of its basic properties, see e.g.~pages 147-148 in Lecture 9
of \cite{NiSp2006}.  Occasionally it is useful to consider the 
more general notion of 
{\em relative Kreweras complement of $\pi$ in $\sigma$}, defined 
for any $\pi \leq \sigma$ in $NC(n)$, and where the ``usual'' Kreweras 
complement corresponds to the special case $\sigma = 1_n$; see the 
discussion on pages 288-291 in Lecture 18 of \cite{NiSp2006}.
\end{notation}

\vspace{6pt}

Since throughout the paper we will work extensively with 
restrictions of non-crossing partitions, we take a moment 
to state clearly what is our notation for how this works.

\vspace{6pt}

\begin{notation}   \label{def:12}
{\em (Relabeled-restrictions of partitions.)}  
Let $n \geq 1$, let $\pi$ be an element in $NC(n)$, and 
consider a set 
$W = \{ p_1, \ldots , p_m \} \subseteq \{ 1, \ldots , n \}$
where $1 \leq m \leq n$ and $p_1 < \cdots < p_m$.  We use the 
notation ``$\pi_{ { }_W}$'' for the partition of 
$\{ 1, \ldots , m \}$ described as follows: 
for $i, j \in \{ 1, \ldots , m \}$ we have
\[ 
\left(  \begin{array}{c}
\mbox{ $i$ and $j$ belong to } \\
\mbox{ the same block of $\pi_{ { }_W }$ }
\end{array}  \right) \ \Leftrightarrow \
\left(  \begin{array}{c}
\mbox{ $p_i$ and $p_j$ belong to } \\
\mbox{ the same block of $\pi$ }
\end{array}  \right) .
\]
It is immediate that the hypothesis of $\pi$ being 
non-crossing implies
\footnote{It would be possible to introduce and use here 
the notion of ``non-crossing partition of $W$''.  After 
weighing the pros and cons of doing so, we decided to 
rather let $\pi_{ { }_W }$ be a partition
in $NC(m)$, for $m = |W|$.}
that $\pi_{ { }_W } \in NC(m)$.
\end{notation}

\vspace{6pt}

\begin{remark}   \label{rem:12}
We now have the notation set to state precisely what are
the two kinds of lattice isomorphisms indicated in Section 1.2
of the Introduction.

\vspace{6pt}

{\em First kind of isomorphism:} for every $n \geq 1$ and 
$\pi \leq \sigma$ in $NC(n)$ one has
\begin{equation}   \label{eqn:12b}
[ \pi , \sigma ] \approx \prod_{W \in \sigma} 
[ \pi_{ { }_W }, 1_{ { }_{|W|} } ],
\end{equation}
where the relabeled-restriction $\pi_{ { }_W } \in NC( |W| )$
is as above, and $1_{ { }_{|W|} } \in NC( |W| )$ is the partition
with a single block.

\vspace{6pt}

{\em Second kind of isomorphism:} for every $k \geq 1$ and 
$\theta \in NC(k)$ one has
\begin{equation}   \label{eqn:12c}
[ \theta , 1_k ] \approx [ 0_k , \Kr ( \theta ) ]
\approx \prod_{U \in \Kr ( \theta )} 
[ 0_{ { }_{|U|} }, 1_{ { }_{|U|} } ]
= \prod_{U \in \Kr ( \theta )} NC ( |U| ).  
\end{equation}

\vspace{6pt}

It is immediate how these two kinds of isomorphisms work 
together to yield the fact that for any $n \geq 1$ and 
$\pi \leq \sigma$ in $NC(n)$ one has a canonical isomorphism
\begin{equation}  \label{eqn:12d}
[ \pi , \sigma ] \approx NC(1)^{p_1} \times NC(2)^{p_2} 
\times \cdots \times NC(n)^{p_n}
\mbox{ for some $p_1, \ldots , p_n \geq 0$.}
\end{equation}
For a detailed discussion of all this we refer to 
\cite[pages 148-153 in Lecture 9]{NiSp2006}. It may be 
re-assuring to know that, more than being canonical, the 
exponents $p_2, \ldots , p_n$ in (\ref{eqn:12d}) are in fact uniquely
determined -- cf.~\cite[Proposition 9.38]{NiSp2006}.  (The exponent
$p_1$ in (\ref{eqn:12d}) is not uniquely determined, since 
$| NC(1) | = 1$.)
\end{remark}

\vspace{6pt}

\begin{notation-and-remark}    \label{def:13}
{\em (Concatenation and irreducibility).}
% for non-crossing partitions.)}

\noindent
$1^o$ Given $n_1, n_2 \geq 1$ and 
$\pi_1 \in NC(n_1)$, $\pi_2 \in NC(n_2)$, we denote by 
$\pi_1 \diamond \pi_2$ the non-crossing partition in $NC( n_1 + n_2 )$ 
which is obtained by placing $\pi_1$ on the points
$1, \ldots , n_1$ and $\pi_2$ on the points 
$n_1 + 1, \ldots , n_1 + n_2$.  

\vspace{6pt}

\noindent
$2^o$ A non-crossing partition $\pi \in NC(n)$ is said to be 
{\em irreducible} when it cannot be written in the form 
$\pi = \pi_1 \diamond \pi_2$ with $\pi_1 \in NC(n_1)$ and 
$\pi_2 \in NC(n_2)$ for some $n_1, n_2 \geq 1$ with 
$n_1 + n_2 = n$.  This condition is easily seen to be 
equivalent to the fact that the numbers $1$ and $n$ belong 
to the same block of $\pi$, i.e., $\outblocks ( \pi )=1$.  

\vspace{6pt}

\noindent
$3^o$ Every $\pi \in \sqcup_{n=1}^{\infty} NC(n)$ can be
written as a concatenation of irreducible partitions.  This 
is best understood by referring to the outer blocks of $\pi$.  
Indeed, it is straightforward to check that these outer 
blocks can be listed as $W_1, \ldots , W_k$, with 
\[
\min (W_1) = 1, \min (W_2) = 1 + \max (W_1), \ldots , 
\min (W_k) = 1 + \max (W_{k-1}), \max (W_k) = n.
\]
For every $1 \leq i \leq k$, consider the interval 
$J_i := \{ m \in \bN \mid 
\min (W_i) \leq m \leq \max (W_i) \}$, 
which is a union of blocks of $\pi$,
and consider the restricted partition 
$\pi_{ { }_{J_i} } \in NC( |J_i| )$.
The concatenation 
$\pi_{{ }_{J_1}} \diamond \cdots \diamond \pi_{{ }_{J_k}}$ 
will then give back the $\pi$ we started with, and every 
$\pi_{{ }_{J_i}}$ is an irreducible partition in 
$NC( |J_i| )$.

\vspace{6pt}

$4^o$ We mention that, in the setting of part $3^o$, the interval 
partition $\theta := \{ J_1, \ldots , J_k \}$ is called the 
{\em interval cover} of $\pi$.  It is easily checked that this 
$\theta$ is the smallest upper bound for $\pi$ in $\Int (n)$, in 
the sense that one has
\[
\Bigl( \tau \in \Int (n) \mbox{ and } \tau \geq \pi \Bigr)
\ \Rightarrow \ \tau \geq \theta .
\]
\end{notation-and-remark}

$\ $

\subsection{The partial orders 
\boldmath{$\ll$} and \boldmath{$\sqsubseteq$} on 
\boldmath{$NC(n)$}.}

$\ $

\noindent
We will make use of two partial order relations on $NC(n)$ 
which are coarser than reverse refinement, and are defined 
as follows.

\vspace{6pt}

\begin{notation}   \label{def:14}
Let $n \in \bN$ and $\pi , \sigma \in NC(n)$.

\vspace{6pt}

$1^o$ We will write ``$\pi \ll \sigma$'' to mean that 
$\pi \leq \sigma$ in the reverse refinement order and
that, in addition, the following happens:
\begin{equation}   \label{eqn:14a}
\left\{   \begin{array}{l}
\mbox{ For every block $W \in \sigma$ there exists a block} \\
\mbox{ $V \in \pi$ such that $\min (W), \max (W) \in V$.}
\end{array}  \right.
\end{equation}
 
\vspace{6pt}

$2^o$ We will write ``$\pi \sqsubseteq \sigma$'' to mean 
that $\pi \leq \sigma$ in the reverse refinement order 
and that, in addition, the following happens:
\begin{equation}   \label{eqn:14b}
\left\{   \begin{array}{l}
\mbox{ Suppose $W \in \sigma$ and $i_1 < i_2 < i_3$ 
       are elements of $W$.}                          \\
\mbox{ Suppose moreover that $i_1$ and $i_3$ belong to
       the same block $V \in \pi$.}                    \\
\mbox{ Then it follows that $i_2 \in V$ as well.}
\end{array}  \right.
\end{equation}
\end{notation}

\vspace{6pt}

\begin{remark}   \label{rem:15}
Both partial orders $\ll$ and $\sqsubseteq$ have been
considered before: $\ll$ was introduced in \cite{BeNi2008}, in 
connection to the study of the so-called Boolean
Bercovici-Pata bijection, while $\sqsubseteq$ was introduced 
and studied in \cite{JV2015}. 

We mention that the recent paper \cite{BiJV2019} generalizes
$\ll$ and $\sqsubseteq$ to the setting of Coxeter groups 
and puts into evidence the fact that these two partial orders
are, in a certain sense, dual to each other.  A special case of 
this duality, which we use in Section 10 of the paper, 
is reviewed in Remark \ref{rem:17} below. 
\end{remark}

\vspace{6pt}

\begin{remark}   \label{rem:16}
$1^o$ Let $n \in \bN$ and $\pi \in NC(n)$, and let us 
record what happens when in Notation \ref{def:14} we put
$\sigma = 1_n$.  We note that:

\noindent
-- ``$\pi \ll 1_n$'' means that 
$1$ and $n$ are in the same block of $\pi$, i.e. that 
$\pi$ is irreducible.

\noindent
-- ``$\pi \sqsubseteq 1_n$'' means precisely that 
$\pi$ is an interval partition.

\vspace{6pt}

$2^o$ More generally, let $n \in \bN$ and let 
$\pi, \sigma \in NC(n)$ be such that $\pi \leq \sigma$.  
One can construe the latter inequality as saying that 
$\pi$ is obtained out of $\sigma$ by taking, one by one, 
the blocks of $\sigma$, and by performing a non-crossing 
partition of each of these blocks.  From this perspective:
the relation $\pi \ll \sigma$ amounts to the fact that 
$\pi$ is obtained  by performing an {\em irreducible} 
partition of every block of $\sigma$, while the relation 
$\pi \sqsubseteq \sigma$ amounts to the fact that $\pi$ 
is obtained by performing an {\em interval} partition of 
every block of $\sigma$.
\end{remark}

\vspace{6pt}

\begin{remark}   \label{rem:17}
We record here two facts about the order relations
$\ll$ and $\sqsubseteq$ that will be used later on in 
the paper.

\vspace{6pt}

$1^o$ The Kreweras complementation map 
$\Kr$ on $NC(n)$ provides us with a bijection
\begin{equation}  \label{eqn:17a}
\bigl\{ \pi \in NC(n) \mid \pi 
\mbox{ is irreducible} \bigr\}
\ni \tau \mapsto \Kr ( \tau ) \in
\bigl\{ \sigma \in NC(n) \mid
\{ n \} \mbox{ is a block of } \sigma \bigr\},
\end{equation}
which is a {\em poset anti-isomorphism} when the set on 
the left-hand side of (\ref{eqn:17a}) is endowed with the partial 
order $\ll$, while the set on the right-hand side is 
endowed with $\sqsubseteq$.
For reference, see e.g. \cite[Lemma 2.10]{FeMaNiSz2019}. 

\vspace{6pt}

$2^o$ For an irreducible partition $\pi \in NC(n)$, the 
upper ideal $\{ \sigma \in NC(n) \mid \sigma \gg \pi \}$ 
has cardinality $2^{| \pi | - 1}$.  And more precisely:
for $\pi \in NC(n)$ and every $1 \leq k \leq | \pi |$, one has that
\begin{equation}   \label{eqn:17b}
\vline \  \{ \sigma \in NC(n) \mid \sigma \gg \pi, \, 
| \sigma | = k \} \ \vline \ =  \left(  
\begin{array}{c} | \pi | - 1 \\ k-1 \end{array} \right) .
\end{equation}
For reference, see e.g. \cite[Proposition 2.13]{BeNi2008}. 
\end{remark}

$\ $

\section{Definition of $\cGtild$}

\subsection{Framework of the incidence algebra 
on non-crossing partitions.}

\begin{definition}  \label{def:22}
 We denote 
\begin{equation}   \label{eqn:22a}
\NCint := \sqcup_{n=1}^{\infty} 
\{ ( \pi , \sigma ) \mid \pi , \sigma \in NC(n), \ \pi \leq \sigma \}.
\end{equation}
The set of functions from $\NCint$ to $\bC$ goes under the name of
{\em incidence algebra of non-crossing partitions}.
This set of functions carries a natural associative operation
of {\em convolution}, denoted as ``$*$'', where for any
$f, g : \NCint \to \bC$ and any $\pi \leq \sigma$ in an $NC(n)$ 
one puts
\begin{equation}   \label{eqn:22b}
f * g \, ( \pi, \sigma ) = \sum_{ \begin{array}{c} 
{\scriptstyle \rho \in NC(n),} \\
{\scriptstyle \pi \leq \rho \leq \sigma}
\end{array} } \ f ( \pi, \rho ) \cdot g ( \rho , \sigma ).
\end{equation}
\end{definition}

\vspace{6pt}

In the next remark we collect a few relevant facts concerning the 
above mentioned convolution operation.  The reader is referred to 
\cite[Lecture 10]{NiSp2006} (cf.~pages 155-158 there) 
for a more detailed presentation.  The general framework 
of incidence algebras comes from work of Rota and collaborators, 
e.g.~in \cite{DoRoSt1972}; a detailed presentation of this
appears in Chapter 3 of \cite{St1997}.

\vspace{6pt}

\begin{remark}    \label{rem:23}
It is easy to verify that the convolution operation 
``$*$'' defined by (\ref{eqn:22b}) is associative 
and unital, where the unit is the function 
$e : \NCint \to \bC$ given by 
\begin{equation}   \label{eqn:23a}
e ( \pi, \sigma) = \left\{   \begin{array}{ll}
1, &  \mbox{ if $\pi = \sigma$,}   \\
0, &  \mbox{ otherwise.}
\end{array}   \right.
\end{equation}
For a function $f: \NCint \to \bC$ one has 
(see e.g.~\cite[Proposition 10.4]{NiSp2006}) that
\begin{equation}   \label{eqn:23b}
\left(  \begin{array}{c}  
\mbox{$f$ is invertible} \\
\mbox{with respect to ``$*$''} 
\end{array}  \right) \ \Leftrightarrow \ 
\Bigl( f( \pi, \pi ) \neq 0, 
\ \ \forall \, \pi \in \sqcup_{n=1}^{\infty} NC(n) \Bigr).
\end{equation}
Moreover, if $f$ is invertible with respect to ``$*$'', then
upon writing explicitly what it means to have 
$f * f^{-1} ( \pi , \pi ) = e( \pi , \pi ) = 1$, one
immediately sees that the inverse $f^{-1}$ satisfies
\begin{equation}   \label{eqn:23c}
f^{-1} ( \pi, \pi ) = \frac{1}{f( \pi , \pi )},
\ \ \forall \, n \geq 1 \mbox{ and } \pi \in NC(n).
\end{equation}

A reader who is matrix-inclined may choose to take the point 
of view that a function $f : \NCint \to \bC$ is just an 
upper triangular matrix with rows and columns indexed by 
$\sqcup_{n=1}^{\infty} NC(n)$, and where the values 
$f( \pi, \sigma )$ appear as certain entries of the matrix.
Then the operation ``$*$'' amounts to matrix multiplication, 
and the formulas (\ref{eqn:23a}), (\ref{eqn:23b}), 
(\ref{eqn:23c}) have obvious meanings in that language as well.
\end{remark}

\vspace{6pt}

\begin{notation-and-remark}   \label{def:24}
{\em (Unitized functions on $\NCint$.)}
We denote
\begin{equation}  \label{eqn:24a}
\cF := \{ f : \NCint \to \bC \mid f( \pi , \pi ) = 1, 
\ \ \forall \, \pi \in \sqcup_{n=1}^{\infty} NC(n) \} .
\end{equation}
The observations made in (\ref{eqn:23b}), (\ref{eqn:23c}) 
show that every $f \in \cF$ is invertible under convolution, 
where the inverse $f^{-1}$ still belongs to $\cF$.
It is also immediate that if $f,g \in \cF$ then 
$f*g \in \cF$, since for every 
$\pi \in \sqcup_{n=1}^{\infty} NC(n)$ the formula defining
$f * g \, ( \pi , \pi )$ boils down to just
$f*g \, ( \pi, \pi) 
= f( \pi, \pi ) \cdot g ( \pi, \pi ) = 1$.
Thus $( \cF , * )$ is a group.
\end{notation-and-remark}

$\ $

\subsection{Unitized semi-multiplicative functions 
on \boldmath{$\NCint$}.}

$\ $

\noindent
We now proceed, as promised, to looking at functions 
on $\NCint$ which are (only) required to be multiplicative 
with respect to the first kind of isomorphism indicated in 
Section 1.2. 

\vspace{6pt}

\begin{definition}   \label{def:26}
We will denote by $\cGtild$ the set of functions
$g : \NCint \to \bC$ which have $g( \pi , \pi ) = 1$ 
for all $\pi \in \sqcup_{n=1}^{\infty} NC(n)$ and 
satisfy the following condition: 
\begin{equation}   \label{eqn:26a}
\left\{   \begin{array}{l}
\mbox{For every $n \geq 1$ and $\pi \leq \sigma$ in $NC(n)$ 
      one has the factorization}                            \\
                                                           \\
\mbox{ $\ $ \hspace{2cm}  $\ $} 
       g( \pi , \sigma) = \prod_{W \in \sigma} 
                  g( \pi_{ { }_W } , 1_{ { }_{|W|} } )
\end{array}  \right.
\end{equation}
(where $\pi_{ { }_W }$ and $1_{ { }_{|W|} }$ are the same as in 
Equation (\ref{eqn:12b}) of Remark \ref{rem:12}).
We will refer to the condition (\ref{eqn:26a}) by calling it 
{\em semi-multiplicativity}, in contrast with the stronger 
{\em multiplicativity} condition from the work of Speicher
\cite{Sp1994}, which also considers the second kind of 
isomorphism reviewed in Remark \ref{rem:12}.
\end{definition}

\vspace{6pt}

From (\ref{eqn:26a}) it is obvious that a function
$g \in \cGtild$ is completely determined when we know 
the values $g( \pi, 1_n )$ for all $n \geq 1$ and 
$\pi \in NC(n)$.  It is hence clear that the map indicated 
in (\ref{eqn:27a}) below is injective.  This map turns 
out to also be surjective; it thus identifies $\cGtild$, 
as a set, with the countable direct product of copies of 
$\bC$ denoted as ``$\cZ$'' in the next proposition.

\vspace{6pt}

\begin{proposition}   \label{prop:27}
Let us denote 
$\cZ := \bigl\{ \uz \mid \uz : \sqcup_{n=1}^{\infty} 
NC(n) \setminus \{ 1_n \}  \to \bC \bigr\}$.

\vspace{6pt}

$1^o$ One has a bijection
$\cGtild \ni g \mapsto \uz \in \cZ$,
with $\uz$ obtained out of $g$ by putting 
\begin{equation}   \label{eqn:27a}
\uz ( \pi ) =  g( \pi , 1_n ) \mbox{ for every 
$n \geq 1$ and }
\pi \in NC(n) \setminus \{ 1_n \}.
\end{equation}

$2^o$ The inverse of the bijection from (\ref{eqn:27a}) is 
described as follows.  Given a $\uz \in \cZ$, 
we ``fill in'' values $\uz (1_n) = 1$ for all $n \geq 1$,
and then define $g : \NCint \to \bC$ by
\begin{equation}   \label{eqn:27c}
g( \pi , \sigma ) := \prod_{W \in \sigma} 
\uz ( \pi_{ { }_W } ), \ \ \forall \, 
( \pi, \sigma ) \in \NCint.
\end{equation}
Then $g \in \cGtild$, and is sent by the map from 
(\ref{eqn:27a}) onto the $\uz$ we started with.
\end{proposition}

\begin{proof}  
Let $\uz \in \cZ$ be given and
let $g : \NCint \to \bC$ be defined as in (\ref{eqn:27c}).
Then (\ref{eqn:27a}) is satisfied, because it is the 
special case ``$\sigma = 1_n$'' of (\ref{eqn:27c}).
Upon combining (\ref{eqn:27c}) and (\ref{eqn:27a}), we thus 
see that $g$ satisfies the factorization
condition indicated in (\ref{eqn:26a}).  We have moreover that
$g( \pi, \pi ) = \prod_{W \in \pi} z( 1_{|W|} ) = 1,
\ \ \forall \, \pi \in \sqcup_{n=1}^{\infty} NC(n)$,
so we conclude that $g \in \cGtild$.  Clearly, this $g$ 
is sent by the map (\ref{eqn:27a}) into the $\uz \in \cZ$ 
that we started with.

The argument in the preceding paragraph covers at the 
same time the surjectivity which was left to check
in part $1^o$ of the proposition, and the inverse 
description in part $2^o$.
\end{proof}

\begin{remark}   \label{rem:28}
Recall that, in parallel with $1_n \in NC(n)$, 
one uses the notation ``$0_n$'' for the partition in 
$NC(n)$ which has $n$ blocks of cardinality $1$.
We warn the reader that $0_n$ and $1_n$ do not play
symmetric roles in the study of $\cGtild$.  Indeed, 
it is immediate that if $g \in \cGtild$ corresponds 
to a $\uz \in \cZ$ in the way described in
Proposition \ref{prop:27}, then we have
\begin{equation}   \label{eqn:28a}
g( 0_n , \sigma ) 
= \prod_{W \in \sigma} g( 0_{|W|} , 1_{|W|} )  
= \prod_{W \in \sigma} \uz ( 0_{|W|} ), 
\ \ \forall \, n \geq 1 \mbox{ and } \sigma \in NC(n),
\end{equation}
quite different from the Equation (\ref{eqn:27a}) 
giving the values $g( \pi , 1_n )$.
\end{remark}

$\ $

\section{$\cGtild$ is a group under convolution}

\noindent
In this section we prove that
$\cGtild$ is a subgroup of the convolution group
$(\cF, *)$ considered in Remark \ref{def:24}.
We start by observing that the semi-multiplicativity condition 
(\ref{eqn:26a}) has an automatic upgrade to a ``local'' version,
shown in the next lemma (where the special case 
$U = \{ 1, \ldots , n \}$ retrieves the original definition of 
semi-multiplicativity).

\vspace{6pt}

\begin{lemma}   \label{lemma:31}
{\em (Local semi-multiplicativity.)}

\noindent
Let $n \geq 1$ and $\pi, \sigma \in NC(n)$ be 
such that $\pi \leq \sigma$.  Let $U$ be a non-empty subset 
of $\{ 1, \ldots , n \}$ which is a union of blocks of 
$\sigma$.  For every $g \in \cGtild$ one has:
\begin{equation}   \label{eqn:31a}
g( \pi_{{ }_U}, \sigma_{{ }_U} ) =
\prod_{ \substack{W \in \sigma , \\ W \subseteq U} } 
\ g( \pi_{{ }_W}, 1_{|W|} ),
\end{equation}
where the relabeled-restrictions ($\pi_{{ }_U}$ and such)
are in the sense of Notation \ref{def:12}.
\end{lemma}

\begin{proof} We write explicitly
$U = W_1 \cup \cdots \cup W_k$ with blocks $W_1, \ldots , W_k \in \sigma$,
and for every $1 \leq i \leq k$ we write 
$W_i = W_{i,1}\cup \cdots \cup W_{i,p_i}$ with blocks
$W_{i,1}, \ldots , W_{i,p_i} \in \pi$.  It follows that the partitions
$\pi_{ { }_U }, \sigma_{ { }_U } \in NC( |U| )$ are of the form
\[
\sigma_{ { }_U } = \{ T_1, \ldots , T_k \} \ \mbox{ and } 
\ \pi_{ { }_U } = \{ T_{1,1}, \ldots , T_{1,p_1}, \ldots ,
T_{k,1}, \ldots , T_{k,p_k} \},
\]
with $T_i = T_{i,1} \cup \cdots \cup T_{i,p_i} \subseteq
\{ 1, \ldots , |U| \}$ for every $1 \leq i \leq k$. 
We leave it as an exercise to the reader to follow the 
necessary relabeled-restrictions and verify that for every 
$1 \leq i \leq k$ we have
\begin{equation}  \label{eqn:31b}
( \pi_{ { }_U } )_{ { }_{T_i} } = \pi_{ { }_{W_i} }
\mbox{ (equality of partitions in $NC(n_i)$, with
        $n_i = |W_i| = |T_i|$).}
\end{equation}

When applied to the partitions $\pi_{{ }_U} \leq \sigma_{{ }_U}$
in $NC( |U| )$, the original semi-multiplicativity condition
from (\ref{eqn:26a}) says that
\[
g( \pi_{{ }_U}, \sigma_{{ }_U} ) =
\prod_{i=1}^k g \bigl( \, 
( \pi_{ { }_U } )_{{ }_{T_i}} , 1_{|T_i|} \, \bigr).
\]
On the right-hand side of the latter equation we replace
$( \pi_{ { }_U } )_{ { }_{T_i} }$ by $\pi_{ { }_{W_i} }$ and 
$1_{|T_i|}$ by $1_{|W_i|}$, and the required formula 
(\ref{eqn:31a}) follows.
\end{proof}

\vspace{6pt}

As an application of local semi-multiplicativity, we 
get the following fact.

\vspace{6pt}

\begin{lemma}   \label{lemma:32}
Let $n \geq 1$, let $\pi , \rho, \sigma \in NC(n)$ with
$\pi \leq \rho \leq \sigma$, and let $g \in \cGtild$. One has
\begin{equation}   \label{eqn:32a}
g( \pi, \rho ) = \prod_{U \in \sigma} 
\ g( \pi_{{ }_U}, \rho_{{ }_U} ).
\end{equation}
\end{lemma}

\begin{proof}  Every block $U \in \sigma$ is a union 
of blocks of $\rho$, hence Lemma \ref{lemma:31} can be 
invoked in connection to this $U$ and the partitions
$\pi \leq \rho$ to infer that
\begin{equation}  \label{eqn:32b}
g( \pi_{{ }_U}, \rho_{{ }_U} ) =
\prod_{ \substack{W \in \rho , \\ W \subseteq U} } 
\ g( \pi_{{ }_W}, 1_{|W|} ).
\end{equation}
We thus find that 
\begin{align*}
\prod_{U \in \sigma} \ g( \pi_{{ }_U}, \rho_{{ }_U} )
& = \prod_{U \in \sigma}
\Bigl(  \prod_{ \substack{W \in \rho , \\ W \subseteq U} } 
\ g( \pi_{{ }_W}, 1_{|W|} ) \Bigr)
\mbox{ (by (\ref{eqn:32b}))}                 \\
& = \prod_{W \in \rho}
\ g( \pi_{{ }_W}, 1_{|W|} ) \Bigr)       \\
& = g( \pi , \rho )
\ \ \mbox{ (by definition of semi-multiplicativity), }
\end{align*}
and the required formula (\ref{eqn:32a}) is obtained. 
\end{proof}

\vspace{6pt}

\begin{theorem}    \label{thm:33}
$\cGtild$ is a subgroup of $( \cF , * )$.
\end{theorem}

\begin{proof}
$1^o$ {\em We pick two functions $g_1, g_2 \in \cGtild$, 
and we prove that $g_1 * g_2$ is in $\cGtild$ as well.}

\noindent
The function $g_1 * g_2 : \NCint \to \bC$ can be 
in any case considered as an element of the larger group 
$\cF$.  Proposition \ref{prop:27} gives us a function 
$g \in \cGtild$ such that 
\begin{equation}   \label{eqn:33a}
g( \pi, 1_n) = g_1 * g_2 \, ( \pi , 1_n )
\ \mbox{ for every 
$n \geq 1$ and $\pi \in NC(n)$.}
\end{equation}
We will prove that, for this $g$, we actually have
\begin{equation}   \label{eqn:33b}
g( \pi, \sigma ) = g_1 * g_2 \, ( \pi , \sigma )
\ \mbox{ for every 
$n \geq 1$ and $\pi \leq \sigma$ in $NC(n)$.}
\end{equation}
This will imply in particular that 
$g_1 * g_2 = g \in \cGtild$, as required.

So let us fix an $n \geq 1$ and some $\pi \leq \sigma$
in $NC(n)$, for which we will verify that (\ref{eqn:33b})
holds.  We write explicitly 
$\sigma = \{ W_1, \ldots , W_k \}$, and we calculate as
follows:
\begin{align*}
g( \pi , \sigma )
& = \prod_{i=1}^k g( \pi_{{ }_{W_i}}, 1_{|W_i|} )
\ \ \mbox{ (by semi-multiplicativity)}                \\
& = \prod_{i=1}^k g_1 * g_2 \, ( \pi_{{ }_{W_i}}, 1_{|W_i|} )
\ \ \mbox{ (by (\ref{eqn:33a}))}                        
\end{align*}
\[
= \prod_{i=1}^k  \Bigl( \sum_{ \begin{array}{c} 
{\scriptstyle \rho_i \in NC( |W_i| ), }   \\
{\scriptstyle \rho_i \geq \pi_{{ }_{W_i}} }
\end{array} } \, g_1 ( \pi_{{ }_{W_i}}, \rho_i ) \cdot
                 g_2 ( \rho_i, 1_{|W_i|} ) \Bigr)
\mbox{ (by the def. of ``$*$'')}
\]
\begin{equation}   \label{eqn:33c}
= \sum_{ \begin{array}{c}
{\scriptstyle \rho_1 \geq \pi_{{ }_{W_1}} \in NC(|W_1|), \ldots} \\ 
{\scriptstyle \ldots , \rho_k \geq \pi_{{ }_{W_k}} \in NC(|W_k|)} 
\end{array} } \Bigl( \prod_{i=1}^k g_1 ( \pi_{{ }_{W_i}}, \rho_i ) \Bigr) 
\cdot \Bigl( \prod_{i=1}^k g_2 ( \rho_i, 1_{|W_i|} ) \Bigr),
\end{equation}
where the latter equality is obtained by expanding the product from 
the preceding line.

But now, one has a natural order-preserving bijection 
\begin{equation}    \label{eqn:33d}
\left\{   \begin{array}{rcl}
\{ \rho \in NC(n) \mid \rho \leq \sigma \} & \longrightarrow
       & NC( |W_1| ) \times \cdots \times NC( |W_k| ),         \\
\rho & \mapsto & ( \rho_{{ }_{W_1}}, \ldots , \rho_{{ }_{W_k}} ),
\end{array}   \right.
\end{equation}
where the relabeled-restrictions 
$\rho_{{ }_{W_1}}, \ldots \rho_{{ }_{W_k}}$ are 
as described in Notation \ref{def:12}.  We observe that 
the bijection from (\ref{eqn:33d}) sends the set
$\{ \rho \in NC(n) \mid \pi \leq \rho \leq \sigma \}$ onto
\[
\Bigl\{ ( \rho_1, \ldots , \rho_k ) \in NC( |W_1| ) 
\times \cdots \times NC( |W_k| ) \mid
\rho_1 \geq \pi_{{ }_{W_1}}, \ldots , \rho_k \geq \pi_{{ }_{W_k}} 
\Bigr\} .
\]
Consequently, the latter bijection can be used in order to perform a 
``change of variable'' in the summation from (\ref{eqn:33c}), and 
turn it into a summation over the set
$\{ \rho \in NC(n) \mid \pi \leq \rho \leq \sigma \}$.  When we perform
this change of variable we arrive to the formula
\begin{equation}   \label{eqn:33e}
g( \pi , \sigma ) = \sum_{ \begin{array}{c}
{\scriptstyle \rho \in NC(n),}   \\
{\scriptstyle \pi \leq \rho \leq \sigma}  
\end{array} } \ \Bigl( 
\prod_{i=1}^k g_1 ( \pi_{{ }_{W_i}}, \rho_{{ }_{W_i}} ) \Bigr) 
\cdot \Bigl( 
\prod_{i=1}^k g_2 ( \rho_{{ }_{W_i}}, 1_{|W_i|} ) \Bigr).
\end{equation}

At this point we recognize the products on the right-hand side
of (\ref{eqn:33e}) as
\begin{equation}   \label{eqn:33f}
\left\{  \begin{array}{l}
\prod_{i=1}^k g_1 ( \pi_{{ }_{W_i}}, \rho_{{ }_{W_i}} )
= g_1 ( \pi , \rho ) \ \ \mbox{ (by Lemma \ref{lemma:32} 
                                    for $g_1$), and}       \\
                                                           \\
\prod_{i=1}^k g_2 ( \rho_{{ }_{W_i}}, 1_{|W_i|} ) 
= g_2 ( \rho , \sigma ) \ \ \mbox{ (by plain 
       semi-multiplicativity for $g_2$).}
\end{array}  \right.
\end{equation}
Upon substituting (\ref{eqn:33f}) into (\ref{eqn:33e}), 
we arrive to
\[
g( \pi , \sigma ) = \sum_{ \begin{array}{c}
{\scriptstyle \rho \in NC(n),}   \\
{\scriptstyle \pi \leq \rho \leq \sigma}  
\end{array} }  g_1 ( \pi , \rho ) \cdot g_2 ( \rho , \sigma )
= g_1 * g_2 \, ( \pi , \sigma ),
\mbox{ as required in (\ref{eqn:33b}).}
\]

\vspace{6pt}

$2^o$ {\em We pick a $g \in \cGtild$ and we prove
that $g^{-1}$ (inverse under convolution) is in 
$\cGtild$ as well.}

\noindent
The inverse $g^{-1}$ of $g$ can be in any case 
considered in the larger group $\cF$.  Our task 
here is to prove that $g^{-1}$ belongs in fact to $\cGtild$.

For every $n \geq 1$ we define a family of complex numbers
$\{ \uz( \pi ) \mid \pi \in NC(n) \}$ in the way described 
as follows.  We first put $\uz (1_n) =1$, then for 
$\pi \in NC(n) \setminus \{ 1_n \}$ we proceed by 
induction on the number $| \pi |$ of blocks of $\pi$ and put
\begin{equation}   \label{eqn:33g}
\uz( \pi ) := - \sum_{  \begin{array}{c}
{\scriptstyle \sigma \in NC(n),}  \\
{\scriptstyle \sigma \geq \pi, \ \sigma \neq \pi}
\end{array}  } \ g( \pi, \sigma ) \uz ( \sigma ).
\end{equation}
Note that all the values $\uz( \sigma )$ invoked on the 
right-hand side of (\ref{eqn:33g}) can indeed be used in this
inductive definition, since the conditions 
$\sigma \geq \pi, \ \sigma \neq \pi$ imply
that $| \sigma | < | \pi |$.

Proposition \ref{prop:27} gives us a function $h \in \cGtild$ 
such that $h( \pi, 1_n) = \uz( \pi )$ for every $n \geq 1$ 
and $\pi \in NC(n)$.  It is immediate that, with $h$ so defined, 
Equation (\ref{eqn:33g}) can be read as saying that
\begin{equation}   \label{eqn:33h}
g * h \, ( \pi , 1_n ) = 0,
\ \mbox{ for every 
$n \geq 1$ and $\pi \in NC(n) \setminus \{ 1_n \}$. }
\end{equation}
Now, in view of part $1^o$ of the present {proof}, 
we have that $g * h \in \cGtild$.
Equation (\ref{eqn:33h}) states that $g * h$ agrees with 
the unit $e$ of $\cGtild$ on all couples $( \pi , 1_n )$ 
with $n \geq 1$ and $\pi \in NC(n) \setminus \{ 1_n \}$.
Since an element of 
$\cGtild$ is uniquely determined by its values on such 
couples $( \pi , 1_n )$, we conclude that $g * h = e$.

Upon reading the equality $g * h = e$ in the larger group 
$\cF$, we see that $h = g^{-1}$.  Hence 
$g^{-1} = h \in \cGtild$, as we had to prove.
\end{proof}

$\ $

\section{Multiplicative vs semi-multiplicative:
the inclusion $\cG \subseteq \cGtild$}

\noindent
In this section we briefly review the situation when a 
function $g \in \cGtild$ is also required to respect the 
second kind of isomorphism reviewed in
Remark \ref{rem:12}, and is thus a 
{\em multiplicative function on non-crossing partitions}
in the sense considered by Speicher \cite{Sp1994}.
It is easily seen that in order to upgrade to this 
situation, it suffices to require $g$ to be well-behaved 
with respect to the isomorphism 
$[ \theta , 1_k ] \approx [ 0_k , \Kr ( \theta ) ]$ 
mentioned at the beginning of the line in (\ref{eqn:12c}).
We can therefore go with the following concise definition.

\vspace{6pt}

\begin{definition}    \label{def:41}
Consider the group of semi-multiplicative functions $\cGtild$ 
discussed in Sections 3 and 4.
A function $g \in \cGtild$ will be said to be {\em multiplicative}
when it has the property that
\begin{equation}   \label{eqn:41a}
g( \pi , 1_n ) = g ( 0_n, \Kr ( \pi ) ),
\ \ \forall \, n \geq 1 \mbox{ and } 
\pi \in NC(n),
\end{equation}
where $\Kr$ is the Kreweras complementation map on $NC(n)$.
We will denote 
\begin{equation}   \label{eqn:41b}
\cG := \{ g \in \cGtild \mid g 
\mbox{ satisfies the condition (\ref{eqn:41a})} \} .
\end{equation}
\end{definition}

\vspace{6pt}

\begin{remark}   \label{rem:42}
Let $g$ be a function in $\cG$ and let us denote 
\begin{equation}   \label{eqn:42a}
g( 0_n, 1_n) =: \lambda_n, \ \ n \geq 1.
\end{equation}
Upon combining (\ref{eqn:41a}) with the formula for 
$g( 0_n , \sigma )$ that had been recorded in Remark 
\ref{rem:28}, we find that for every $n \geq 1$ and 
$\pi \in NC(n)$ we have
\begin{equation}   \label{eqn:42b}
g( \pi, 1_n ) = 
\prod_{W \in \Kr ( \pi )} \lambda_{ |W| }.
\end{equation}
This generalizes to  
\begin{equation}   \label{eqn:42c}
g( \pi, \sigma ) = 
\prod_{W \in \Kr_{\sigma} ( \pi )} \lambda_{ |W| },
\ \ \forall \, n \geq 1 \mbox{ and } 
\pi \leq \sigma \mbox{ in } NC(n),
\end{equation}
where $\Kr_{\sigma} ( \pi )$ is the relative Kreweras 
complement of $\pi$ in $\sigma$.  Indeed, 
Equation (\ref{eqn:42c}) follows easily from 
(\ref{eqn:42b}) when we invoke the semi-multiplicativity 
factorization (\ref{eqn:26a}) and then take into account 
that $\Kr_{\sigma} ( \pi )$ is obtained by performing in 
parallel Kreweras complementation on all the restricted 
partitions $\pi_{ { }_W }$, with $W$ running among the 
blocks of $\sigma$.

In connection to the above, we have the following statement.
\end{remark}

\begin{proposition}   \label{prop:43}
Let $\lambdans$ be a sequence in $\bC$, with $\lambda_1 = 1$.
There exists a multiplicative function $g \in \cG$, uniquely
determined, such that $g( 0_n, 1_n) = \lambda_n$ for all 
$n \geq 1$.
\end{proposition}

The uniqueness part of Proposition \ref{prop:43} is 
clearly implied by the formula (\ref{eqn:42c}).  For the 
existence part, one {\em defines} $g$ by using the 
formula (\ref{eqn:42c}), and then proves (via a 
discussion very similar to the one on pages 164-167 of 
\cite[Lecture 10]{NiSp2006}) that $g \in \cG$. 

\vspace{6pt}

\begin{remark}   \label{rem:44}
It turns out that $\cG$ is in fact a {\em subgroup} 
of $\cGtild$.  For the proof of this fact we refer to 
\cite[Theorem 18.11]{NiSp2006}.  Due to some basic symmetry  
properties enjoyed by the Kreweras complementation map it 
turns out, moreover, that $\cG$ (unlike $\cGtild$) is a 
commutative group -- see \cite[Corollary 17.10]{NiSp2006}.
\end{remark}

$\ $

\section{The action of $\cGtild$ on sequences
of multilinear functionals}

\noindent
The relevance of the group $\cGtild$ for non-commutative
probability considerations stems from a natural action that
this group has on certain sequences of multilinear functionals. 
In order to describe this action, it is convenient to introduce
the following notation.

\vspace{6pt}

\begin{notation}   \label{def:51}
Let $\cA$ be a vector space over $\bC$.  We denote
\begin{equation}   \label{eqn:51a}
\fM_{ { }_{\cA} } := \{ \uPsi \mid 
\uPsi = ( \psi_n : \cA^n \to \bC )_{n=1}^{\infty},
\mbox{ where $\psi_n$ is an $n$-linear functional} \}.
\end{equation}
\end{notation}

\vspace{6pt}

\begin{remark-and-notation}   \label{rem:52}
$1^o$ In Notation \ref{def:51} we did not need to assume 
that $\cA$ is an algebra, or that it comes endowed with an 
expectation functional $\varphi : \cA \to \bC$.  If that would 
be the case, and we would thus be dealing with a non-commutative 
probability space $( \cA , \varphi )$, then the set 
$\fM_{ { }_{\cA} }$ would get to have a special element 
$\uPhi = ( \varphi_n : \cA^n \to \bC)_{n=1}^{\infty}$ where
\begin{equation}   \label{eqn:52a}
\varphi_n (x_1, \ldots , x_n ) := \varphi (x_1 \cdots x_n), 
\ \ \forall \, x_1, \ldots , x_n \in \cA.
\end{equation}
Such a $\uPhi$ is called ``family of moment functionals'' of 
$( \cA , \varphi )$.  

\vspace{6pt}

$2^o$  Given a $\uPsi = ( \psi_n )_{n=1}^{\infty}$ as in 
(\ref{eqn:51a}), there is a standard way of enlarging $\uPsi$ by 
adding to it some multilinear functionals indexed by non-crossing 
partitions.  More precisely: for any $n \geq 1$ and $\pi \in NC(n)$, 
it is customary to denote as $\psi_{\pi} : \cA^n \to \bC$ the 
multilinear functional which acts by 
\begin{equation}   \label{eqn:52b}
\psi_{\pi} (x_1, \ldots , x_n)
= \prod_{V \in \pi} \psi_{|V|} ( \, ( x_1, \ldots , x_n )
\mid V \, ), \ \ x_1, \ldots , x_n \in \cA .
\end{equation} 
[A concrete example: if we have, say, $n=5$ and 
$\pi = \{ \, \{ 1,2,5 \}, \, \{ 3,4 \} \, \} \in NC(5)$,
then the formula defining $\psi_{\pi}$ becomes 
$\psi_{\pi} (x_1, \ldots , x_5) :=
\psi_3 (x_1, x_2, x_5) \cdot \psi_2 (x_3, x_4)$.]

\noindent
The convention for how to enlarge $\uPsi$ is useful when 
we introduce the following notation.
\end{remark-and-notation}

\vspace{6pt}

\begin{notation}   \label{def:53}  
Let $\cA$ be a vector space over $\bC$.  For every 
$\uPsi = ( \psi_n )_{n=1}^{\infty} \in \fM_{ { }_{\cA} }$ and 
$g \in \cGtild$, we denote by ``$\uPsi \cdot g$'' the element 
$\uTheta = ( \theta_n )_{n=1}^{\infty} \in \fM_{ { }_{\cA} }$ 
defined by putting
\begin{equation}   \label{eqn:53a}
\theta_n = \sum_{\pi \in NC(n)} 
g( \pi , 1_n ) \psi_{\pi}, \ \ \forall \, n \geq 1.
\end{equation}
The right-hand side of (\ref{eqn:53a}) has a linear combination
done in the vector space of $n$-linear functionals from 
$\cA^n$ to $\bC$, where the $\psi_{\pi}$ are as defined in 
Notation \ref{rem:52}.2.
\end{notation}

\vspace{6pt}

We will prove that the map introduced in 
Notation \ref{def:53} is a group action.  It is 
convenient to first record an extension 
of the formula used to define $\uPsi \cdot g$.

\vspace{6pt}

\begin{lemma}  \label{lemma:54} 
Let $\uPsi , \uTheta \in \fMcA$ and $g \in \cGtild$ be such 
that $\uTheta = \uPsi \cdot g$. Consider the extended families 
of multilinear functionals 
$\{ \psi_{\pi} \mid \pi \in \sqcup_{n=1}^{\infty} NC(n) \}$   
and $\{ \theta_{\pi} \mid \pi \in \sqcup_{n=1}^{\infty} NC(n) \}$
that are obtained out of $\uPsi$ and $\uTheta$, respectively,
in the way indicated in Notation \ref{rem:52}.2.  Then for every 
$n \geq 1$ and $\sigma \in NC(n)$ one has
\begin{equation}    \label{eqn:54a}
\theta_{\sigma} = 
\sum_{ \substack{\pi \in NC(n), \\  \pi \leq \sigma} } 
\,  g( \pi , \sigma ) \psi_{\pi} .
\end{equation}
\end{lemma}

\begin{proof} Let us write explicitly 
$\sigma = \{ W_1, \ldots , W_k \}$.  Then for every 
$x_1, \ldots , x_n \in \cA$ we have
\begin{align*}
\theta_{\sigma} (x_1, \ldots , x_n)
& = \prod_{j=1}^k \theta_{|W_j|} 
    \bigl( \, (x_1, \ldots , x_n) \mid W_j \, \bigr)
    \mbox{ (by the definition of $\theta_{\sigma}$)}   \\
& = \prod_{j=1}^k \Bigl( \sum_{\pi_j \in NC( |W_j| )}
    g( \pi_j, 1_{|W_j|}) \psi_{\pi_j} 
    \bigl( \, (x_1, \ldots , x_n) \mid W_j \, \bigr) \Bigr)
    \mbox{ (by Eqn.~(\ref{eqn:53a})). }
    \end{align*}
Upon expanding the latter product of $k$ factors,
we find $\theta_{\sigma} (x_1, \ldots , x_n)$ to be equal to
\begin{equation}   \label{eqn:54b}
\sum_{ \begin{array}{c}
    {\scriptstyle \pi_1 \in NC( |W_1| ), \ldots } \\
    {\scriptstyle \ldots ,  \pi_k \in NC( |W_k| )}
    \end{array} } 
\ \Bigl( \prod_{j=1}^k g( \pi_j, 1_{ |W_j| } ) \Bigr)
\cdot \Bigl( \prod_{j=1}^k \psi_{\pi_j}
\bigl( \, (x_1, \ldots , x_n) \mid W_j \, \bigr) \Bigr).
\end{equation}
But now, one has a natural bijection
\begin{equation}    \label{eqn:54c}
\left\{   \begin{array}{rcl}
\{ \pi \in NC(n) \mid \pi \leq \sigma \} & \longrightarrow
       & NC( |W_1| ) \times \cdots \times NC( |W_k| ),         \\
\pi & \mapsto & ( \pi_{W_1}, \ldots , \pi_{W_k} ),
\end{array}   \right.
\end{equation}
where the partitions $\pi_{W_j} \in NC( |W_j| )$ are 
relabeled-restrictions of $\pi$ (cf.~Notation \ref{def:12}).
When we use this bijection in order to perform a change of 
variables in the summation from (\ref{eqn:54b}), the 
semi-multiplicativity property of $g$ assures us that the
product $\prod_{j=1}^k g( \pi_j, 1_{ |W_j| } )$ is converted
into just ``$g( \pi , \sigma )$''.  On the other hand, it is 
easily checked that the said change of variable transforms 
$\prod_{j=1}^k \psi_{\pi_j}
\bigl( \, (x_1, \ldots , x_n) \mid W_j \, \bigr)$ into 
``$\psi_{\pi} (x_1, \ldots , x_n)$''.   Hence our computation of
what is $\theta_{\sigma} (x_1, \ldots , x_n )$ has lead to 
$\sum_{\pi \leq \sigma}
g( \pi , \sigma ) \cdot \psi_{\pi} (x_1, \ldots , x_n)$,
as required.
\end{proof} 

\vspace{6pt}

\begin{proposition}   \label{prop:55}
Let $\cA$ be a vector space over $\bC$.  The formula 
(\ref{eqn:53a}) from Notation \ref{def:53} defines an 
action of the group $\cGtild$ on the set $\fM_{ { }_{\cA} }$.
That is, one has
\begin{equation}   \label{eqn:55a}
( \uPsi \cdot g_1 ) \cdot g_2 =  \uPsi \cdot (g_1 * g_2),
\ \ \forall \uPsi \in \fM_{ { }_{\cA} } \mbox{ and }
g_1, g_2 \in \cGtild.
\end{equation}
\end{proposition}

\begin{proof} We denote 
$\uPsi \cdot g_1 =: \uTheta = ( \theta_n )_{n=1}^{\infty}$
and $( \uPsi \cdot g_1) \cdot g_2 
=: \uEta = ( \eta_n )_{n=1}^{\infty}$.
Our goal for the proof is to verify that 
$\uEta = \uPsi \cdot (g_1 * g_2)$, i.e.~that we have 
\begin{equation}   \label{eqn:55b}
\eta_n = \sum_{\pi \in NC(n)} 
g_1 * g_2 \, ( \pi, 1_n) \, \psi_{\pi},
\ \ \forall \, n \geq 1,
\end{equation}
where 
$\{ \psi_{\pi} \mid \pi \in \sqcup_{n=1}^{\infty} NC(n) \}$
is the extension of $\uPsi$.
We thus fix an $n \geq 1$ for which we will verify that
(\ref{eqn:55b}) holds.  We write the formula given for 
$\eta_n$ by the relation $\uEta = \uTheta \cdot g_2$ and 
then we invoke Lemma \ref{lemma:54} in connection to the 
relation $\uTheta = \uPsi \cdot g_1$, to find that:
\[
\eta_n 
= \sum_{\sigma \in NC(n)}  g_2 ( \sigma, 1_n)  \theta_{\sigma}
= \sum_{\sigma \in NC(n)}  g_2 ( \sigma, 1_n)
\Bigl( 
\sum_{ \substack{\pi \in NC(n), \\  \pi \leq \sigma} } 
\ g_1 ( \pi, \sigma ) \, \psi_{\pi} \Bigr).
\]  
Changing the order of summation 
in the latter double sum then leads to:
\begin{equation}  \label{eqn:55c}
\eta_n = \sum_{\pi \in NC(n)}
\Bigl( 
\sum_{ \substack{\sigma \in NC(n), \\ \pi \leq \sigma} } 
\ g_1 ( \pi, \sigma ) g_2 ( \sigma, 1_n ) \Bigr)
\, \psi_{\pi}.
\end{equation}
The interior sum in (\ref{eqn:55c}) is equal to 
$g_1 * g_2 \, ( \pi , 1_n)$, and we have thus obtained
the required Equation (\ref{eqn:55b}).
\end{proof}

\vspace{6pt}

\begin{remark}   \label{rem:56}
Throughout this section we have considered, for the sake of 
simplicity, only multilinear functionals with values in $\bC$.
We invite the reader to take a moment to observe that the 
whole discussion could have been pursued, without any change,
in the framework where we consider multilinear functionals with
values in a unital commutative algebra $\cC$ over $\bC$.
Indeed, suppose we have fixed such a $\cC$.  Then 
Notation \ref{def:51} is adjusted by putting
\begin{equation}   \label{eqn:56a}
\fMcA^{( \cC )} := \Bigl\{  
\uPsi  \begin{array}{ll}
\vline  & \uPsi = ( \psi_n : \cA^n \to \cC )_{n=1}^{\infty},
          \mbox{ where every $\psi_n$}                       \\
\vline  & \mbox{ is a $\bC$-multilinear functional}
\end{array} \Bigr\}.  
\end{equation}

\noindent
Given $\uPsi \in \fMcA^{( \cC ) }$ and $g \in \cGtild$, we define
what is $\uPsi \cdot g \in \fMcA^{( \cC ) }$ by the very same formula
as in (\ref{eqn:53a}) of Notation \ref{def:53}.  The proof of 
Proposition \ref{prop:55} goes through without any changes, to show that
in this way we obtain a right group action of $\cGtild$ on 
$\fMcA^{ ( \cC ) }$.

In the rest of the paper we will stick everywhere to the basic case 
when $\cC = \bC$, with only one exception: Section 7.4 will have an 
occurrence of the case where $\cC$ is the Grassmann algebra 
$\bG := \{ \alpha + \ee \beta \mid \alpha , \beta \in \bC \}$,
with multiplication defined by
\[
( \alpha_1 + \ee \beta_1) \cdot ( \alpha_2 + \ee \beta_2 )
= \alpha_1 \alpha_2 + \ee ( \alpha_1 \beta_2 + \alpha_2 \beta_1),
\ \ \mbox{ for }
\alpha_1, \beta_1, \alpha_2, \beta_2 \in \bC.
\]
\end{remark}

$\ $

\section{Cumulant-to-moment type, and cumulant-to-cumulant type}

\noindent
There are several brands of cumulants which live naturally 
in the universe of non-crossing partitions, and are
commonly used in the non-commutative probability literature.  
Each such brand of cumulants has its own ``moment-cumulant'' 
summation formula, and there also exist useful summation 
formulas that connect different brands of cumulants.  The action 
of $\cGtild$ on sequences of multilinear functionals that was 
observed in Section 6 offers an efficient way to do calculations 
related to these moment-cumulant and inter-cumulant formulas.  
In connection to that, we next put into evidence: a 
{\em factorization property} which seems to always be fulfilled 
when one considers functions $g \in \cGtild$ involved in 
moment-cumulant formulas; and a {\em vanishing property} 
fulfilled by functions $g \in \cGtild$ which are involved in 
inter-cumulant formulas.  Both these 
properties are phrased in connection to the operation 
``$\diamond$'' of concatenation of non-crossing partitions,
and to the notion of irreducibility with respect to 
concatenation, as reviewed in Notation \ref{def:13}.

\vspace{6pt}

\begin{definition}   \label{def:62}
$1^o$  A function $g \in \cGtild$ will be said to be of
{\em cumulant-to-moment} type when it has the property that
\begin{equation}   \label{eqn:62a}
g( \pi_1 \diamond \pi_2, 1_{n_1 + n_2} ) 
= g( \pi_1, 1_{n_1} ) \cdot g( \pi_2 , 1_{n_2} ), 
\end{equation}
holding for all $n_1, n_2 \geq 1$ and $\pi_1 \in NC(n_1)$,
$\pi_2 \in NC(n_2)$.  We denote
\[
\cGtildctom := \{ g \in \cGtild \mid g \mbox{ is of 
cumulant-to-moment type} \}.
\]

\vspace{6pt}

\noindent
$2^o$  A function $g \in \cGtild$ will be said to be of 
{\em cumulant-to-cumulant} type 
when it satisfies
\begin{equation}   \label{eqn:62b}
g( \pi_1 \diamond \pi_2, 1_{n_1 + n_2} ) = 0, 
\ \ \forall \, n_1, n_2 \geq 1 \mbox{ and } \pi_1 \in NC(n_1),
\ \pi_2 \in NC(n_2).
\end{equation}
We denote
\[
\cGtildctoc := \{ g \in \cGtild \mid g \mbox{ is of 
cumulant-to-cumulant type} \}.
\]
\end{definition}

\vspace{6pt}

\begin{remark}   \label{rem:63}
$1^o$ A function $g \in \cGtildctoc$ is completely 
determined when we know its values $g( \pi , 1_n )$ with 
$\pi \in NC(n)$ irreducible.  Indeed, the condition on 
$g$ stated in (\ref{eqn:62b}) just says that if 
$\pi \in NC(n)$ is not irreducible, then $g ( \pi, 1_n ) = 0$. 
So we know the values $g( \pi , 1_n )$ for {\em all} $n \geq 1$
and $\pi \in NC(n)$, which determines $g$ (cf.~Proposition 
\ref{prop:27}).

\vspace{6pt}

$2^o$ Consider now a function $g \in \cGtildctom$.  An easy 
induction shows that for every $k \geq 1$, 
$n_1, \ldots , n_k \geq 1$ and 
$\pi_1 \in NC(n_1), \ldots , \pi_k \in NC(n_k)$,
one has:
\begin{equation}   \label{eqn:64a} 
g( \pi_1 \diamond \cdots \diamond \pi_k , 1_{n_1 + \cdots + n_k} ) 
= \prod_{j=1}^k g( \pi_j , 1_{n_j} ).
\end{equation}
Since every non-crossing partition can be written as a 
concatenation of irreducible ones, we conclude that our 
$g \in \cGtildctom$ can be completely reconstructed if we
know its values $g( \pi , 1_n)$ with $\pi \in NC(n)$ 
irreducible -- indeed, Equation (\ref{eqn:64a}) then tells 
us what is $g( \pi , 1_n)$ for {\em all} $n \geq 1$ and 
$\pi \in NC(n)$, and Proposition \ref{prop:27} can be applied.
\end{remark}

\vspace{6pt}

In the next section we will examine $\cGtildctom$ and 
$\cGtildctoc$ from the group structure point of view, within the 
group $( \cGtild , * )$.  For now we only want to show, by example, 
what is the rationale for the terms ``cumulant-to-moment'' and 
``cumulant-to-cumulant'' used in Definition \ref{def:62}.  This is 
an opportunity to review a few salient examples of cumulants, and to
display some of the functions in $\cGtild$ which encode transition formulas 
from these cumulants to moments, or encode transition formulas
between two different brands of cumulants.

$\ $

\subsection{Free and Boolean cumulants.}

$\ $

\noindent
Throughout this subsection we fix a non-commutative 
probability space $( \cA, \varphi )$, we look at
\[
\fM_{ { }_{\cA} } := \{ \uPsi \mid 
\uPsi = ( \psi_n : \cA^n \to \bC )_{n=1}^{\infty},
\mbox{ where $\psi_n$ is an $n$-linear functional} \},
\]
and we consider the family of moment functionals 
$\uPhi = ( \varphi_n )_{n=1}^{\infty} \in \fMcA$ which was
introduced in Notation \ref{rem:52}.1.

\vspace{6pt}

\begin{definition-and-remark}   \label{def:65}
The family of {\em free cumulant functionals} of 
$( \cA , \varphi )$ is the family
$\uKappa = ( \kappa_n )_{n=1}^{\infty} \in \fMcA$
defined via the requirement that for every
$n \geq 1$ and $x_1, \ldots , x_n \in \cA$ one has:
\begin{equation}   \label{eqn:65a}
\varphi (x_1 \cdots x_n) = 
\sum_{\pi \in NC(n)} \prod_{V \in \pi} 
\kappa_{|V|} \bigl( \, (x_1, \ldots , x_n) \mid V \, \bigr).
\end{equation}

\noindent
This requirement can be re-phrased as follows: let 
$\Ffcm : \NCint \to \bC$ be 
\footnote{In $\Ffcm$, the subscript ``fc-m'' is a 
reminder that we are doing a transition from free 
cumulants to moments.  Similar conventions will be used
for other such special functions, e.g. ``$\Fbcm$'' for
the function in $\cGtildctom$ which encodes the transition
from Boolean cumulants to moments, or ``$\Ffcbc$'' for the
function in $\cGtildctoc$ which encodes the transition from
free cumulants to Boolean cumulants.}
defined by
\begin{equation}   \label{eqn:65b}
\Ffcm ( \pi , \sigma ) = 1, \ \ \forall 
\, n \geq 1 \mbox{ and } \pi \leq \sigma 
\mbox{ in } NC(n).
\end{equation}
It is immediate that $\Ffcm \in \cGtild$ and that it
fulfills the factorization condition (\ref{eqn:62a}), hence 
$\Ffcm \in \cGtildctom$.  The ``moment-cumulant''
formula (\ref{eqn:65a}) can be read as an instance of 
the group action from Section 6, it just says that
\begin{equation}   \label{eqn:65c}
\uPhi = \uKappa \cdot \Ffcm .
\end{equation}
Indeed, (\ref{eqn:65a}) asks for the equality
of $n$-linear functionals
$\varphi_n = \sum_{\pi \in NC(n)} \kappa_{\pi}$,
holding for every $n \geq 1$,
and with $\kappa_{\pi}$'s defined as in 
Notation \ref{rem:52}.2; but the latter equality is the
same as (\ref{eqn:65c}).
\end{definition-and-remark}

\vspace{6pt}

We next repeat the same moment-cumulant formulation 
in connection to {\em Boolean cumulants}, where 
we now refer to interval partitions.

\vspace{6pt}

\begin{definition-and-remark}   \label{def:67}
The family of {\em Boolean cumulant functionals} of 
$( \cA , \varphi )$ is the family
$\uBeta = ( \beta_n )_{n=1}^{\infty} \in \fMcA$
defined via the requirement that for every
$n \geq 1$ and $x_1, \ldots , x_n \in \cA$ one has:
\begin{equation}   \label{eqn:67a}
\varphi (x_1 \cdots x_n) = 
\sum_{\pi \in \Int (n)} \prod_{J \in \pi} 
\beta_{|J|} \bigl( \, (x_1, \ldots , x_n) \mid J \, \bigr).
\end{equation}

Now, consider the function $\Fbcm \in \cGtild$ 
defined via the requirement that for every $n \geq 1$
and $\pi \in NC(n)$ we have
\begin{equation}   \label{eqn:67b}
\Fbcm ( \pi , 1_n ) = \left\{  \begin{array}{ll}
1, & \mbox{if $\pi \in \Int (n)$,}  \\
0, & \mbox{otherwise.}
\end{array}  \right.
\end{equation}
Such a function does indeed exist and is unique, as
guaranteed by Proposition \ref{prop:27}.  We see moreover 
that $\Fbcm$ is a function of cumulant-to-moment type: 
indeed, given any $n_1, n_2 \geq 1$ and 
$\pi \in NC(n_1), \, \pi_2 \in NC(n_2)$, it is immediate 
that
\[
\Fbcm ( \pi_1 \diamond \pi_2 , 1_{n_1 + n_2} ) 
= \Fbcm ( \pi_1 , 1_{n_1} ) \cdot 
\Fbcm ( \pi_2, 1_{n_2} )
= \left\{  \begin{array}{ll}
1, & \mbox{if both $\pi_1$ and $\pi_2$ are} \\
   & \mbox{ $\ $ interval partitions,}  \\
0, & \mbox{otherwise.}
\end{array}  \right.
\]

Hence $\Fbcm \in \cGtildctom$ and (exactly as we did for free
cumulants in Remark \ref{def:65}) we see that the 
moment-cumulant formula (\ref{eqn:67a}) amounts to just:
\begin{equation}  \label{eqn:67c}  
\uPhi = \uBeta \cdot \Fbcm .
\end{equation}
\end{definition-and-remark}

\vspace{6pt}

\begin{remark}    \label{rem:68}
It was convenient to introduce the function $\Fbcm$ by just 
postulating its values $\Fbcm ( \pi , 1_n)$, and then by 
invoking Proposition \ref{prop:27}.  It is not hard to  
actually write down the formula for the values taken by 
$\Fbcm$ on general couples in $\NCint$; this is found by 
using the semi-multiplicativity property, and comes out 
(immediate verification) as 
\begin{equation}   \label{eqn:68a}
\Fbcm ( \pi , \sigma ) = 
\left\{   \begin{array}{ll}
1, & \mbox{ if $\pi \sqsubseteq \sigma$,}  \\
0, & \mbox{ otherwise.}
\end{array}  \right\} , \ \ \forall \, n \geq 1
\end{equation}
where $\sqsubseteq$ is one of the partial order relations
reviewed in Section 2.2.
\end{remark}

$\ $

\subsection{An interpolation between free and Boolean:
\boldmath{$t$}-Boolean cumulants.}

$\ $

\noindent
In this subsection we continue to use the notation 
from Section 7.1, where $\uPhi \in \fMcA$ is
the family of moment functionals of the non-commutative 
probability space $( \cA, \varphi )$, and 
$\uKappa, \uBeta \in \fMcA$ are the families of free
and respectively Boolean cumulants of the same space.
Our goal for the subsection is to review a $1$-parameter 
interpolation between $\uBeta$ and $\uKappa$, arising 
from the work of Bo{\.z}ejko and Wysoczanski 
\cite{BoWy2001}, and defined in the way described as 
follows.

\vspace{6pt}

\begin{definition}   \label{def:69}
Let $t \in \bR$ be a parameter.  We will use the name 
{\em $t$-Boolean cumulant functionals} of $( \cA , \varphi )$
to refer to the sequence of multilinear functionals
$\uBetat = ( \betat_n )_{n=1}^{\infty} \in \fMcA$
defined via the requirement that for every $n \geq 1$ and 
$x_1, \ldots , x_n \in \cA$ one has:
\begin{equation}   \label{eqn:69a}
\varphi (x_1 \cdots x_n) = 
\sum_{\pi \in NC(n)}  t^{\innblocks ( \pi )}
\prod_{V \in \pi}
\betat_{|V|} \bigl( \, (x_1, \ldots , x_n) \mid V \, \bigr).
\end{equation}
Recall that $\innblocks ( \pi )$ is our notation for the 
number of inner blocks of $\pi \in NC(n)$.
\end{definition}

\vspace{6pt}

\begin{remark}    \label{rem:610}
It is clear that for $t=1$ one gets $\uBeta^{(1)} = \uKappa$.
On the other hand, for $t = 0$ one gets that
$\uBeta^{(0)} = \uBeta$, because in this case the right-hand 
side of Equation (\ref{eqn:69a}) reduces to a sum over 
$\Int (n)$ (cf.~(\ref{eqn:11a}) in the review of background).
\end{remark}

\vspace{6pt}

\begin{notation-and-remark}    \label{def:611}
For every $t \in \bR$, let $\Fbcm^{(t)}$ be the 
function in $\cGtild$ defined via the requirement that 
\begin{equation}   \label{eqn:611a}
\Fbcm^{(t)} ( \pi , 1_n ) := t^{\innblocks ( \pi )},
\ \mbox{ for all $n \geq 1$ and $\pi \in NC(n)$.}
\end{equation}
As an immediate consequence of the obvious fact that
\[
\innblocks ( \pi_1 \diamond \pi_2 ) =
\innblocks ( \pi_1 ) + \innblocks ( \pi_2 ),
\ \ \forall \, \pi_1 , \pi_2 \in \sqcup_{n=1}^{\infty} NC(n),
\]
one has that $\Fbcm^{(t)}$ is a function of 
cumulant-to-moment type.  The formula (\ref{eqn:69a}) 
used to define the $t$-Boolean cumulant functionals 
can be concisely re-written in the form
\begin{equation}    \label{eqn:611b}
\uPhi = \uBeta^{(t)} \cdot \Fbcm^{(t)}.
\end{equation}
This is a common generalization of the 
formulas (\ref{eqn:65a}) and (\ref{eqn:67a}) 
observed for free and for Boolean cumulants -- the latter 
formulas are obtained by setting the parameter to $t=1$
and to $t=0$, respectively.
\end{notation-and-remark}

\vspace{6pt}

\begin{remark}   \label{rem:612}
When using the action of the group $\cGtild$, one sees
very clearly how to combine moment-cumulant formulas for 
two different brands of cumulants in order to get a direct 
connection between the cumulants themselves. 
We illustrate how this works when we want to go from 
$s$-Boolean cumulants to $t$-Boolean cumulants for two 
distinct parameters $s,t \in \bR$.  We have 
$\uBeta^{(t)} \cdot \Fbcm^{(t)} = \uPhi
= \uBeta^{(s)} \cdot \Fbcm^{(s)}$, hence:
\begin{equation}   \label{eqn:612a}
\uBeta^{(t)} = \uPhi \cdot ( \Fbcm^{(t)} )^{-1}
= ( \uBeta^{(s)} \cdot \Fbcm^{(s)} ) \cdot ( \Fbcm^{(t)} )^{-1}
= \uBeta^{(s)} \cdot 
  \bigl( \Fbcm^{(s)} * ( \Fbcm^{(t)} )^{-1}  \bigr).
\end{equation}
In short: the transition from $s$-Boolean cumulants to 
$t$-Boolean cumulants is encoded by the function 
$\Fbcm^{(s)} * ( \Fbcm^{(t)} )^{-1} \in \cGtildctoc$.
The values of this function turn out to have a nice explicit 
description (cf.~Remark \ref{rem:82}.1 and Corollary \ref{cor:85}
below), where in particular we find that for $n \geq 1$ and 
$\pi \in NC(n)$ we have:
\begin{equation}    \label{eqn:612b}   
\Fbcm^{(s)} * ( \Fbcm^{(t)} )^{-1} \,
( \pi , 1_n ) =
\left\{   \begin{array}{ll}
(s-t)^{| \pi | - 1}, & \mbox{ if $\pi$ is irreducible,} \\
0, & \mbox{ otherwise.}
\end{array}   \right.
\end{equation}
Hence, when spelled out explicitly, the transition formula 
(\ref{eqn:612a}) says this: for every 
$n \geq 1$ and $x_1, \ldots , x_n \in \cA$ we have
\begin{equation}   \label{eqn:612c}
\betat_n ( x_1, \ldots , x_n ) 
= \sum_{  \begin{array}{c}
\pi \in NC(n),  \\  
{\scriptstyle \mathrm{irreducible} }
\end{array}  } \
(s-t)^{\innblocks ( \pi )}
\prod_{V \in \pi}
\, \beta^{(s)}_{|V|} 
\bigl( \, (x_1, \ldots , x_n) \mid V \, \bigr).
\end{equation}

In the special case when $s=1$ and $t=0$, Equation (\ref{eqn:612c})
becomes the transition formula from free cumulants to Boolean 
cumulants, which is well-known since the work of 
Lehner \cite{Le2002}.  When swapping the role of the parameters
and putting $s=0$ and $t=1$, one finds the inverse transition formula 
which writes free cumulants in terms of Boolean cumulants, and is also
well-known (cf.~\cite[Proposition 3.9]{BeNi2008}, 
\cite[Section 4]{ArHaLeVa2015}).
\end{remark}

$\ $

\subsection{Monotone cumulants.}

$\ $

\noindent
We continue to use the framework and notation of the 
subsections 7.1 and 7.2.  Another family of cumulant functionals
associated to $( \cA , \varphi )$ that gets constant attention
in the research literature on non-commutative probability is the
family of {\em monotone cumulant functionals} which were introduced 
in \cite{HaSa2011}, based on the notion of {\em monotone ordering} 
of a partition $\pi \in NC(n)$.   The latter notion is defined as 
a bijection $\ell : \pi \to \{ 1, \ldots , | \pi | \}$ (or in 
other words: a total ordering of the blocks of $\pi$) 
which has the property that 
\begin{equation}   \label{eqn:613x}
\left\{   \begin{array}{l}
\mbox{ If $V,W \in \pi$ are such that $V$ is nested inside $W$} \\
\mbox{ then it follows that $\ell (V) \geq \ell (W)$.}
\end{array}  \right.
\end{equation}
With this notion in hand, one then proceeds as follows.

\vspace{6pt}

\begin{definition}    \label{def:613}
The family 
$\uRho = ( \rho_n : \cA^n \to \bC )_{n=1}^{\infty}$
of {\em monotone cumulants} of $( \cA , \varphi )$
is defined via the requirement that for every $n \geq 1$ and 
$x_1, \ldots , x_n \in \cA$ one has
\begin{equation}   \label{eqn:613a}  
\varphi (x_1 \cdots x_n) = 
\sum_{\pi \in NC(n)} 
\frac{\mbox{\# of monotone orderings of $\pi$}}{ | \pi | !}
\cdot \prod_{V \in \pi}
\rho_{ { }_{|V|} } \bigl( \, (x_1, \ldots , x_n) \mid V \, \bigr).
\end{equation}
\end{definition}

\vspace{6pt}

\begin{notation-and-remark}    \label{def:614}
In order to re-phrase the preceding definition in terms
of the action of $\cGtild$ on $\fMcA$, we let $\Fmcm$ be 
the function in $\cGtild$ defined via the requirement that
\begin{equation}   \label{eqn:614a}
\Fmcm ( \pi, 1_n ) =
\frac{\mbox{\# of monotone orderings of $\pi$}}{ | \pi | !},
\ \ \forall \, n \geq 1 \mbox{ and } \pi \in NC(n).
\end{equation}
An elementary counting argument (presented for instance in
\cite[Proposition 3.3]{ArHaLeVa2015}) shows that $\Fmcm$ 
satisfies the factorization condition stated in
(\ref{eqn:62a}), and is therefore of cumulant-to-moment type.
The formula (\ref{eqn:613a}) from the preceding definition 
gets to be re-phrased as
\begin{equation}   \label{eqn:614b}
\uPhi = \uRho \cdot \Fmcm,
\end{equation}
in close analogy to how the definitions of 
$\uKappa, \uBeta, \uBetat$ were re-phrased in the 
preceding subsections.
\end{notation-and-remark}

$\ $

\subsection{Infinitesimal cumulants.}

$\ $

\noindent
There exists an ``infinitesimal'' extension of the notion of 
non-commutative probability space, which has been considered 
primarily for the purpose of pinning down an infinitesimal 
version of the notion of free independence for non-commutative 
random variables (cf.~\cite{BeSh2012}, and the follow-up in 
\cite{Sh2018} relating this topic to random matrix theory). 
An {\em infinitesimal non-commutative probability space} is a 
triple $( \cA , \varphi, \varphi ' )$ where $( \cA, \varphi )$
is a non-commutative probability space in the usual sense 
and one also has a second linear functional
$\varphi ' : \cA \to \bC$ such that 
$\varphi ' ( \oneA ) = 0$.  To such a space one associates:

\vspace{4pt}

$\bullet$ a sequence of free infinitesimal cumulant 
functionals
$\uKappa ' = ( \kappa_n ' : \cA^n \to \bC )_{n=1}^{\infty}$;

\vspace{4pt}

$\bullet$ a sequence of Boolean infinitesimal 
cumulant functionals
$\uBeta ' = ( \beta_n ' : \cA^n \to \bC )_{n=1}^{\infty}$;

\vspace{4pt}

$\bullet$ a sequence of monotone infinitesimal 
cumulant functionals
$\uRho ' = ( \rho_n ' : \cA^n \to \bC )_{n=1}^{\infty}$.

\vspace{4pt}

\noindent
The infinitesimal free cumulants $\uKappa '$ were introduced in
\cite{FeNi2010}, while ${{\uBeta'}}, \uRho '$ were introduced in 
\cite{Ha2011} (see also the detailed study of all these notions
appearing in the recent paper \cite{CeEFPe2019}).  

We note that $\uKappa ', \uBeta ', \uRho '$ belong to 
$\fMcA$, the set bearing the action of $\cGtild$ from 
Section 6.  The definitions of these infinitesimal cumulants can 
be described in terms of a variation of this action of $\cGtild$, 
going now on $\fMcA \times \fMcA$.
The occurrence of $\fMcA \times \fMcA$ comes from the fact that
the summations over lattices $NC(n)$ used to describe 
infinitesimal cumulants have terms which depend on 
{\em both} the linear functionals $\varphi, \varphi '$ 
considered on $\cA$.  The details are as follows.

\vspace{6pt}

\begin{notation}   \label{def:615}
Let $\cA$ be a vector space over $\bC$ and let $\fMcA$ be the 
set of sequences of multilinear functionals introduced in 
Notation \ref{def:51}.  Suppose we are given a couple
$( \uPsi^{(1)}, \uPsi^{(2)} ) \in \fMcA \times \fMcA$,
where $\uPsi^{(1)} = ( \psi^{(1)}_n )_{n=1}^{\infty}$ and
$\uPsi^{(2)}  = ( \psi^{(2)}_n )_{n=1}^{\infty}$, and suppose 
we are also given a function $g \in \cGtild$.  We then denote
\begin{equation}   \label{eqn:615a}
( \uPsi^{(1)} , \uPsi^{(2)} ) \dotinf g :=
( \uTheta^{(1)} , \uTheta^{(2)} ) \in \fMcA \times \fMcA,
\end{equation}
where $\uTheta^{(1)} = \uPsi^{(1)} \cdot g$ (exactly as in 
Notation \ref{def:53}) and 
$\uTheta^{(2)} = ( \theta_n^{(2)} )_{n=1}^{\infty}$ is defined 
by the requirement that for every $n \geq 1$ and 
$x_1, \ldots , x_n \in \cA$ we have
\begin{equation}   \label{eqn:615b}
\theta_n^{(2)} (x_1, \ldots , x_n) = 
\end{equation}
\[
\sum_{\pi \in NC(n)} g( \pi, 1_n ) \cdot  \sum_{W \in \pi} \Bigl(
\psi^{(2)}_{|W|} ( \, (x_1, \ldots , x_n) \mid W \, ) 
\cdot \prod_{ \substack{V \in \pi, \\ V \neq W} }
\, \psi^{(1)}_{|V|} ( \, (x_1, \ldots , x_n) \mid V \, ) \Bigr) .
\]
\end{notation}

\vspace{6pt}

\begin{remark}   \label{rem:616}
Let $\cA$ be as in Notation \ref{def:615}.
What hides behind (\ref{eqn:615a}) and 
(\ref{eqn:615b}) is the fact 
that we have a canonical identification:
\begin{equation}   \label{eqn:616a}
\fMcA \times \fMcA \ni ( \uPsi^{(1)}, \uPsi^{(2)} ) 
\mapsto \widetilde{\uPsi} \in \fMcA^{( \bG ) },
\end{equation}
where $\bG$ is the Grassmann algebra and $\fMcA^{ ( \bG ) }$
is as considered in Remark \ref{rem:56} at the end of Section 6.
That is: given 
$\uPsi^{(1)} = ( \psi^{(1)}_n )_{n=1}^{\infty}$ and
$\uPsi^{(2)}  = ( \psi^{(2)}_n )_{n=1}^{\infty}$ in $\fMcA$,  
we create a sequence of $\bC$-multilinear functionals 
$\widetilde{\uPsi} = ( \widetilde{\psi}_n : \cA^n \to \bG )_{n=1}^{\infty}$
by simply putting
\[
\widetilde{\psi}_n (x_1, \ldots , x_n) =
\psi^{(1)}_n (x_1, \ldots , x_n) + \ee
\psi^{(2)}_n (x_1, \ldots , x_n),
\ \ \forall \, x_1, \ldots , x_n \in \cA.
\]

As explained in Remark \ref{rem:56}, the group $\cGtild$ acts 
on the right on $\fMcA^{( \bG ) }$.  The explicit formula for 
$( \uPsi^{(1)} , \uPsi^{(2)} ) \dotinf g$ shown in the 
preceding notation is just the conversion of the 
formula for $\widetilde{\uPsi} \cdot g \in \fMcA^{( \bG )}$,
via the identification (\ref{eqn:616a}).

As a byproduct of the connection with the Grassmann algebra, 
one also gets an immediate proof of the following fact.
\end{remark}

\begin{proposition}     \label{prop:617}
Let $\cA$ be a vector space over $\bC$.
Then ``$\dotinf$'' from Notation \ref{def:615} is an
action on the right of the group $\cGtild$ on 
$\fMcA \times \fMcA$.  That is, one has
\begin{equation}
\bigl( \, ( \uPsi^{(1)}, \uPsi^{(2)} ) \dotinf g \bigr)
\dotinf h = ( \uPsi^{(1)}, \uPsi^{(2)} ) \dotinf ( g *h ),
\ \ \forall \, \uPsi^{(1)}, \uPsi^{(2)} \in \fMcA
\mbox{ and } g,h \in \cGtild,
\end{equation}
where on the right-hand side we use the convolution 
operation of $\cGtild$.  
\hfill  $\square$
\end{proposition}

\begin{remark}    \label{rem:618}
Consider now an infinitesimal non-commutative probability
space $( \cA , \varphi, \varphi ' )$, and let us spell out 
how the infinitesimal cumulants 
$\uKappa ', \uBeta ', \uRho ' \in \fMcA$
mentioned at the beginning of this subsection are described 
in terms of the action $\dotinf$ of the group $\cGtild$.
To that end, let
$\uPhi = ( \varphi_n )_{n=1}^{\infty}$ and
$\uPhi ' = ( \varphi_n ' )_{n=1}^{\infty}$ be the sequences of 
moment functionals associated to $\varphi$ and to $\varphi '$;
that is, for every $n \geq 1$ the $n$-linear 
functionals $\varphi_n, \varphi_n ' : \cA^n \to \bC$ act by
\[
\varphi_n (x_1, \ldots , x_n) = \varphi (x_1 \cdots x_n)
\mbox{ and } 
\varphi_n ' (x_1, \ldots , x_n) = \varphi ' (x_1 \cdots x_n),
\ \ \forall \, x_1, \ldots , x_n \in \cA.
\]
The infinitesimal cumulants we are interested in 
are determined by the ``moment-cumulant'' equations
\begin{equation}   \label{eqn:618a}
( \uPhi , \uPhi ' ) = ( \uKappa , \uKappa ' )  \dotinf \Ffcm 
= ( \uBeta , \uBeta ' )  \dotinf \Fbcm 
= ( \uRho , \uRho ' )  \dotinf \Fmcm ,
\end{equation}
where $\Ffcm, \Fbcm, \Fmcm \in \cGtild$ are the functions of 
cumulant-to-moment type that appeared in the preceding subsections 
(cf.~Equations (\ref{eqn:65b}), (\ref{eqn:68a}) and (\ref{eqn:614a}),
respectively).  So for instace the sequence of free infinitesimal
cumulants $\uKappa '$ is found by picking the second component in
the formula
\begin{equation}    \label{eqn:618b}
( \uKappa , \uKappa ' ) = ( \uPhi , \uPhi ' )  \dotinf \Ffcm^{-1}.
\end{equation}
We note that the $\uKappa$ appearing in 
(\ref{eqn:618a}), (\ref{eqn:618b}) is precisely the sequence of 
free cumulant functionals of $( \cA , \varphi )$, as one sees by 
picking the first component in (\ref{eqn:618b}) and by taking 
into account that on the first component of $\dotinf$ we have 
the ``usual'' action of $\cGtild$ of $\fMcA$.
\end{remark}

$\ $

\section{$\cGtildctoc$ is a subgroup, and 
$\cGtildctom$ is a right coset}

\noindent
In this section we follow up on the subsets 
$\cGtildctoc, \cGtildctom \subseteq \cGtild$
introduced in Definition \ref{def:62}.  We will 
prove that:
\begin{tabular}[t]{ll}
(i)  & $\cGtildctoc$ is a {\em subgroup} of 
$( \cGtild , * )$, and   \\ 
(ii)  & $\cGtildctom$ is a {\em right coset} of the 
subgroup $\cGtildctoc$.
\end{tabular}

\vspace{6pt} 

\noindent 
The statement (ii) means that we can write
\begin{equation}   \label{eqn:7a}
\cGtildctom = \cGtildctoc * h
= \{ g * h \mid g \in \cGtildctoc \},
\end{equation}
for no matter what $h \in \cGtildctom$ we choose to 
fix.  The easiest choice for $h$ is to pick
$h( \pi , \sigma ) = 1$ for all 
$( \pi, \sigma ) \in \NCint$, that is, let $h$ be the
special function $\Ffcm$ from Definition \ref{def:65}.
However, as we will see in Section 8.2 below, it may be 
more advantageous for proofs and applications if we
go instead with $h = \Fbcm$, picked from Definition 
\ref{def:67}.

$\ $

\subsection{Proof that \boldmath{$\cGtildctoc$} is a 
subgroup of \boldmath{$( \cGtild , * )$}.}

$\ $

\noindent
We first record a straightforward extension of
the vanishing condition postulated in Equation 
(\ref{eqn:62b}), in the definition of $\cGtildctoc$.

\vspace{6pt}

\begin{lemma}   \label{lemma:71}
Let $n \geq 1$ and let $\pi \leq \sigma$ be two partitions 
in $NC(n)$, where:
\begin{equation}   \label{eqn:71a}
\left\{   \begin{array}{ll}
\mbox{there exists a block $W_o$ of $\sigma$ such that} \\
\mbox{$\min (W_o)$ and $\max (W_o)$ belong to different 
        blocks of $\pi$.}
\end{array}  \right.
\end{equation}
Then $g( \pi , \sigma ) = 0$ for all $g \in \cGtildctoc$.
\end{lemma}

\begin{proof} For any $g \in \cGtildctoc$ we write the 
factorization $g( \pi , \sigma ) 
= \prod_{W \in \sigma} g( \pi_{{ }_W}, 1_{|W|} )$
provided by semi-multiplicativity, and we observe that the 
factor $g( \pi_{{ }_{W_o}}, 1_{|W_o|} )$ of this factorization 
is sure to be equal to $0$, since the partition 
$\pi_{{ }_{W_o}} \in NC( |W_o| )$ is not irreducible.
\end{proof}

\vspace{6pt}

\begin{proposition}   \label{prop:72}
$1^o$ Let $g_1, g_2$ be in $\cGtildctoc$.  Then 
$g_1 * g_2$ is in $\cGtildctoc$ as well.

\vspace{6pt}

\noindent
$2^o$ Let $g$ be in $\cGtildctoc$.  Then $g^{-1}$ 
(inverse under convolution) is in $\cGtildctoc$ as well.
\end{proposition}

\begin{proof}  $1^o$ Let $n \geq 1$ and let $\pi \in NC(n)$ 
which is not irreducible, that is, $1$ and $n$ belong to 
distinct blocks of $\pi$.  We want to prove that 
$g_1 * g_2 \, ( \pi , 1_n) = 0$.  We have  
\[
g_1 * g_2 \, ( \pi , 1_n) 
= \sum_{\sigma \in NC(n), \, \sigma \geq \pi}
\ g_1 ( \pi , \sigma ) \cdot g_2 ( \sigma , 1_n),
\]
and we will argue that every term of the latter sum is 
equal to $0$.  Indeed, for a $\sigma \in NC(n)$ such that 
$\sigma \geq \pi$ there are two possible cases.

\vspace{6pt}

{\em Case 1: $\sigma$ is not irreducible.}  In this case 
$g_2 ( \sigma , 1_n ) = 0$, and thus
$g_1 ( \pi , \sigma ) \cdot g_2 ( \sigma , 1_n) = 0$.

\vspace{6pt}

{\em Case 2: $\sigma$ is irreducible.}  In this case the numbers $1$ 
and $n$ belong to the same block of $\sigma$, but belong to 
different blocks of $\pi$.  Lemma \ref{lemma:71} applies, and tells 
us  that $g_1 ( \pi , \sigma ) = 0$.  It thus follows that
$g_1 ( \pi , \sigma ) \cdot g_2 ( \sigma , 1_n) = 0$ in this case 
as well.

\vspace{6pt}

$2^o$ We fix an $n \geq 1$, for which we prove that:
\[
\mbox{$g^{-1} ( \pi , 1_n) = 0$ for every $\pi \in NC(n)$ 
which is not irreducible.}
\]
What we will do is to prove by induction on $m$, 
with $1 \leq m \leq n$, that: 
\begin{equation}   \label{eqn:72b}
\left\{   \begin{array}{c}
\pi \in NC(n), \mbox{ not irreducible,}  \\
\mbox{ with $| \pi | = m$ }
\end{array}   \right\}
\ \Rightarrow \ 
g^{-1} ( \pi , 1_n) = 0.
\end{equation}
The base-case
$m =1$ holds trivially, because the set of partitions indicated 
in (\ref{eqn:72b}) is empty in that case (the only partition with 
$| \pi | = 1$ is $\pi = 1_n$, which is irreducible). 
In the remaining part of the proof we discuss the 
induction step: we pick an $m_o \in \{ 2, \ldots , n \}$, we 
assume that (\ref{eqn:72b}) is true for all $m < m_o$,
and we verify that it is also true for $m_o$.

So consider a partition $\pi \in NC(n)$ which is 
not irreducible and has $| \pi | = m_o$.  We have 
$g * g^{-1} \, ( \pi , 1_n ) = e( \pi, 1_n ) = 0$, and upon
writing explicitly what is $g * g^{-1} \, ( \pi , 1_n)$ we 
find, very similar to Equation (\ref{eqn:33g}) in the proof
of Theorem \ref{thm:33}) that
\begin{equation}   \label{eqn:72c}
g^{-1} ( \pi , 1_n ) = - \sum_{\begin{array}{c}
{\scriptstyle  \sigma \in NC(n)}  \\
{\scriptstyle  \sigma \geq \pi, \, \sigma \neq \pi}
\end{array}  } \ g( \pi , \sigma ) \, g^{-1} ( \sigma , 1_n ).
\end{equation}
In order to arrive to the desired conclusion that 
$g^{-1} ( \pi , 1_n ) = 0$, we now verify that every term in 
the sum on the right-hand side of (\ref{eqn:72c}) is equal 
to $0$.  In reference to the partition $\sigma$ which indexes 
the terms of that sum, we distinguish two cases.

\vspace{6pt}

{\em Case 1: $\sigma$ is not irreducible.}  Since
$| \sigma | < | \pi | = m_o$ (as implied by the conditions 

\noindent
$\sigma \geq \pi, \, \sigma \neq \pi$), the induction hypothesis 
applies to $\sigma$, and tells us that $g^{-1} ( \sigma , 1_n ) = 0$.
Hence $g( \pi , \sigma ) \, g^{-1} ( \sigma , 1_n ) = 0$, as we 
wanted.

\vspace{6pt}

{\em Case 2: $\sigma$ is irreducible.}  In this case the 
numbers $1$ and $n$ belong to the same block of $\sigma$, 
but belong to different blocks of $\pi$.  Lemma \ref{lemma:71} 
tells us that $g ( \pi , \sigma ) = 0$, and we thus get that
$g( \pi , \sigma ) \, g^{-1} ( \sigma , 1_n ) = 0$ in this case
as well.
\end{proof}

$\ $

\subsection{Proof that \boldmath{$\cGtildctom$} is a 
right coset of \boldmath{$\cGtildctoc$}.}

$\ $

\noindent
The claim about $\cGtildctom$ that we want to prove
is as stated in Equation (\ref{eqn:7a}) at the beginning
of the section, where on the right-hand side we have to
choose a suitable ``representative'' $h$ picked from 
$\cGtildctom$.  As mentioned immediately following to 
(\ref{eqn:7a}), the coset representative we will go with 
is the function $\Fbcm$ introduced in 
Equation (\ref{eqn:67b}) of the preceding section.  In
connection to it, we first prove a lemma.

\vspace{6pt}

\begin{lemma}   \label{lemma:73}
Let $g$ be in $\cGtildctoc$. Then: $1^o$
$g * \Fbcm \in \cGtildctom$. 

\vspace{6pt}

\noindent
$2^o$ One has
\begin{equation}  \label{eqn:73a}
g * \Fbcm \, ( \pi , 1_n ) =  g( \pi , 1_n )
\ \mbox{ for every
$n \geq 1$ and irreducible $\pi \in NC(n)$.}
\end{equation}
\end{lemma}

\begin{proof}  
We start by recording the general fact that
\begin{equation}   \label{eqn:73b}
g * \Fbcm \, ( \pi, 1_n) 
= \sum_{\sigma \in \Int (n), 
\, \sigma \geq \pi} \ g( \pi , \sigma ), 
\ \ \forall \, n \geq 1 \mbox{ and } \pi \in NC(n).
\end{equation}
This is obtained directly from the definition of the 
convolution operation, when we take into account the 
specifics of what is $\Fbcm$.

\vspace{6pt}

\noindent
{\em Proof of $1^o$.}  We take a partition 
$\pi = \pi_1 \diamond \pi_2 \in NC(n)$ where $n = n_1 + n_2$
with $n_1, n_2 \geq 1$ and where $\pi_1 \in NC(n_1)$,
$\pi_2 \in NC(n_2)$.  Our goal here is to verify that
\begin{equation}   \label{eqn:73c}
g * \Fbcm \, ( \pi, 1_n ) 
= \bigl( g * \Fbcm \, ( \pi_1, 1_{n_1} ) \bigr) \cdot
\bigl( g * \Fbcm \, ( \pi_2, 1_{n_2} ) \bigr) . 
\end{equation}

Observe that we clearly have
\begin{equation}   \label{eqn:73d}
\{ \sigma \in \Int (n) \mid \sigma \geq \pi \}
\supseteq \Bigl\{ \sigma_1 \diamond \sigma_2 
\begin{array}{ll}
\vline &  \sigma_1 \in \Int (n_1), \, \sigma_1 \geq \pi_1, \\ 
\vline &  \sigma_2 \in \Int (n_2), \, \sigma_2 \geq \pi_2
\end{array} \Bigr\} .
\end{equation}
By starting from (\ref{eqn:73b}), we can thus write
\begin{equation}    \label{eqn:73e}
g * \Fbcm \, ( \pi, 1_n) 
= \Bigl( \sum_{ \begin{array}{c}
{\scriptstyle \sigma_1 \in \Int (n_1), \, \sigma_1 \geq \pi_1}  \\
{\scriptstyle \sigma_2 \in \Int (n_2), \, \sigma_2 \geq \pi_2} 
\end{array} } \ g( \pi , \sigma_1 \diamond \sigma_2 ) \Bigr) \ + 
\ \sum_{\sigma \in \cJ} g( \pi, \sigma ),
\end{equation}
where $\cJ$ denotes the difference of the two sets indicated in 
(\ref{eqn:73d}).  But note that $\cJ$ can be described as
\[
\cJ = \Bigl\{ \sigma \in \Int (n) 
\begin{array}{ll}
\vline &  \sigma \geq \pi 
          \mbox{ and there exists a block $W_o$ of $\sigma$} \\
\vline & \mbox{such that $\min (W_o) \leq n_1, \, \max (W_o) > n_1$}
\end{array}  \Bigr\} ;
\]
a direct application of Lemma \ref{lemma:71} then gives that
$g( \pi , \sigma ) = 0$ for every $\sigma \in \cJ$.  So the 
second sum on the right-hand side of Equation (\ref{eqn:73e}) 
is actually equal to $0$.  Concerning the first sum appearing 
there, we observe that its general term can be written as 
\[
g( \pi , \sigma_1 \diamond \sigma_2 )
= g( \pi_1 \diamond \pi_2 , \sigma_1 \diamond \sigma_2 )
= g( \pi_1 , \sigma_1 ) \cdot g( \pi_2 , \sigma_2 ),
\]
with the factorization at the second equality sign coming from the 
semi-multiplicativity of $g$.  When we put these observations 
together, we find that (\ref{eqn:73e}) leads to the factorization
\begin{equation}   \label{eqn:73f}
g * \Fbcm \, ( \pi, 1_n) 
= \Bigl( \sum_{\sigma_1 \in \Int (n_1), \, \sigma_1 \geq \pi_1}
g( \pi_1 , \sigma_1 )  \Bigr) \cdot 
\Bigl( \sum_{\sigma_2 \in \Int (n_2), \, \sigma_2 \geq \pi_2}
g( \pi_2 , \sigma_2 ) \Bigr) .
\end{equation}
Finally, upon specializing the general Equation (\ref{eqn:73b}) 
to the case of the partitions $\pi_1$ and $\pi_2$, we identify the 
right-hand side of (\ref{eqn:73f}) as being the product that
had been announced on the right-hand side of (\ref{eqn:73c}).
This completes the verification that had to be done in this part
of the proof.

\vspace{6pt}

\noindent
{\em Proof of $2^o$.}  Let $n \geq 1$ and let $\pi \in NC(n)$ be
irreducible.  It is immediate that, since $1$ and $n$
belong to the same block of $\pi$, the only $\sigma \in \Int (n)$ 
such that $\sigma \geq \pi$ is $\sigma = 1_n$.  Hence, in this 
special case, the sum on the right-hand side of (\ref{eqn:73b}) 
has only $1$ term, which is equal to $g( \pi , 1_n )$.  It follows 
that $g * \Fbcm \, ( \pi , 1_n ) = g( \pi , 1_n )$, as stated 
in (\ref{eqn:73a}).
\end{proof}

\vspace{6pt}

\begin{proposition}   \label{prop:74}
$\cGtildctom = \{ g * \Fbcm \mid g \in \cGtildctoc \}$.
\end{proposition}

\begin{proof} ``$\supseteq$'': This inclusion is provided by 
Lemma \ref{lemma:73}.1.

\vspace{6pt}
``$\subseteq$'':  Let a function $h \in \cGtildctom$ be given. 
We have to prove that $h$ can be written in the form 
$h = g * \Fbcm$, with $g \in \cGtildctoc$.

Proposition \ref{prop:27} assures us that there exists 
$g \in \cGtild$, uniquely determined, such that for every 
$n \geq 1$ and $\pi \in NC(n)$ we have
\begin{equation}   \label{eqn:74a}
g( \pi, 1_n) = \left\{  \begin{array}{ll}
h( \pi, 1_n), & \mbox{if $\pi$ is irreducible,}   \\
0,            & \mbox{otherwise.}
\end{array}  \right.
\end{equation}
Since Equation (\ref{eqn:74a}) includes the fact that 
$g( \pi , 1_n ) = 0$ whenever $\pi$ is not irreducible, we 
know that $g \in \cGtildctoc$.  
Let $\widetilde{h} := g * \Fbcm$.  Then 
$\widetilde{h} \in \cGtildctom$, by Lemma \ref{lemma:73}.1.
Moreover, for every $n \geq 1$ and irreducible partition 
$\pi \in NC(n)$ we have
\begin{align*}
\widetilde{h} ( \pi , 1_n )
& = g( \pi , 1_n) \ \ \mbox{ (by Lemma \ref{lemma:73}.2)} \\
& = h( \pi , 1_n) \ \ \mbox{ (by (\ref{eqn:74a})). }
\end{align*}
We thus get to have two functions $h$ and $\widetilde{h}$ in
$\cGtildctom$, such that 
$h( \pi , 1_n) = \widetilde{h} ( \pi , 1_n)$ whenever $\pi \in NC(n)$
is irreducible.  As observed in Remark \ref{rem:63}.2, 
this implies $h = \widetilde{h}$.  In particular we have obtained
$h = g * \Fbcm$ with $g \in \cGtildctoc$, and this concludes the 
proof.
\end{proof}

\vspace{6pt}

Proposition \ref{prop:74} has established the required claim 
that $\cGtildctom$ is a right coset of the subgroup 
$\cGtildctoc \subseteq \cGtild$.  It is useful to also record 
the following fact, which gives a converse to 
Lemma \ref{lemma:73}.2, and was implicitly included in our 
method of deriving Proposition \ref{prop:74}.

\vspace{6pt}

\begin{corollary}   \label{cor:75}
Suppose that $g \in \cGtildctoc$, $h \in \cGtildctom$ and it
holds true that 
\begin{equation}   \label{eqn:75a}
g( \pi , 1_n ) = h( \pi , 1_n ), \ \ \forall \, n \geq 1
\mbox{ and irreducible $\pi \in NC(n)$.}
\end{equation}
Then it follows that $h = g * \Fbcm$.
\end{corollary}

\vspace{6pt}

\noindent
{\em Proof.}  We have $g * \Fbcm \in \cGtildctom$, by 
Proposition \ref{prop:74}.  We are thus required to prove an 
equality between two functions in $\cGtildctom$.  To this end,
we know (cf.~Remark \ref{rem:63}.2) it is sufficient to verify 
that the two functions in question, $h$ and $g * \Fbcm$, agree 
on every couple $( \pi , 1_n)$ where $n \geq 1$ and 
$\pi \in NC(n)$ is irreducible.  
And indeed, for such $( \pi, 1_n)$ we have
\begin{align*}
g * \Fbcm \, ( \pi , 1_n )
& = g (\pi, 1_n) \mbox{ (by Lemma \ref{lemma:73}.2) }  \\
& = h( \pi, 1_n) \mbox{ (by hypothesis). \hspace{5cm} $\square$}
\end{align*}

$\ $

\subsection{An application: why Boolean cumulants are 
the easiest to connect to.}

$\ $

\noindent
In this subsection we show how the method of proof used 
in Proposition \ref{prop:74} can be invoked 
to retrieve a known ``rule of thumb'', which says that 
when given a cumulant-to-moment summation formula, 
it is usually immediate to write down the corresponding
cumulant-to-(Boolean cumulant) summations: one uses
{\em the very same coefficients} as in the description 
of moments, only that the summations are now restricted to 
non-crossing partitions that are irreducible.  The precise 
statement of this fact goes as follows.

\vspace{6pt}

\begin{proposition}   \label{prop:76}
Let $( \cA , \varphi )$ be a non-commutative probability space.
Suppose we are given a sequence of multilinear functionals 
$\uLambda = ( \lambda_n )_{n=1}^{\infty} \in \fM_{ { }_{\cA} }$
and a family of complex coefficients 
$( c( \pi ) )_{\pi \in \sqcup_{n=1}^{\infty} NC(n)}$ 
such that for every $n \geq 1$ and 
$x_1, \ldots , x_n \in \cA$ we have
\begin{equation}   \label{eqn:76a}
\varphi (x_1 \cdots x_n) = 
\sum_{\pi \in NC(n)}  c( \pi  )
\, \prod_{V \in \pi}
\lambda_{|V|} \bigl( \, (x_1, \ldots , x_n) \mid V \, \bigr).
\end{equation}
Suppose moreover that, in relation to the operation ``$\diamond$''
of concatenating non-crossing partitions, the coefficients 
$c( \pi )$ have the property that 
\begin{equation}   \label{eqn:76b}
c( \pi_1 \diamond \pi_2 ) = c( \pi_1 ) \cdot c( \pi_2 ),
\ \ \forall \, \pi_1, \pi_2 \in \sqcup_{n=1}^{\infty} NC(n).
\end{equation}
Then: denoting by 
$\uBeta = ( \beta_n )_{n=1}^{\infty}$ the Boolean 
cumulant functionals of $( \cA , \varphi )$, we have
\begin{equation}   \label{eqn:76c}
\beta_n (x_1, \ldots , x_n) = 
\sum_{ \begin{array}{c}
{\scriptstyle \pi \in NC(n),} \\
{\scriptstyle \mathrm{irreducible}} 
\end{array} } \ c ( \pi )
\, \prod_{V \in \pi}
\lambda_{|V|} \bigl( \, (x_1, \ldots , x_n) \mid V \, \bigr),
\end{equation}
holding for every $n \geq 1$ and every $x_1, \ldots , x_n \in \cA$.
\end{proposition}

\begin{proof} Let $h$ be the function in $\cGtild$ which 
is defined via the requirement that $h( \pi, 1_n ) = c( \pi )$,
for all $n \geq 1$ and $\pi \in NC(n)$.
The factorization hypothesis (\ref{eqn:76b}) satisfied by
the coefficients $c( \pi )$ tells us that $h$ is a function 
of cumulant-to-moment type.  On the other hand: Equation
(\ref{eqn:76a}) can be re-written concisely in terms of 
$h$, in the form of the relation $\uPhi = \uLambda \cdot h$, 
where $\uPhi$ is the family of moment functionals of the 
space $( \cA , \varphi )$.  We can therefore write that 
\begin{align*}
\uBeta
& = \uPhi \cdot \Fbcm^{-1} 
    \mbox{ (Boolean cumulants expressed in terms of moments)}  \\
& = ( \uLambda \cdot h ) \cdot \Fbcm^{-1}  
  = \uLambda \cdot ( h * \Fbcm^{-1} ) = \uLambda * g,  
\end{align*}
where we denoted $g := h * \Fbcm^{-1}$.  

Since $h \in \cGtildctom$, Proposition \ref{prop:74} 
implies that $g \in \cGtildctoc$.  Moreover, since $g$ 
and $h$ are related by the convolution $h = g * \Fbcm$, 
Lemma \ref{lemma:73}.2 tells us that we have
\[
g( \pi , 1_n ) = h( \pi , 1_n ), \ \ \forall 
\, n \geq 1 \mbox{ and irreducible $\pi \in NC(n)$.}
\]
By taking into account that $h( \pi, 1_n ) = c( \pi )$, 
we thus come to the conclusion that for every $n \geq 1$ 
and $\pi \in NC(n)$ we have
\begin{equation}   \label{eqn:76d}
g( \pi , 1_n ) = \left\{  \begin{array}{ll}
c( \pi ), & \mbox{ if $\pi$ is irreducible,}  \\
0,        & \mbox{ otherwise.}
\end{array} \right.
\end{equation}

Finally, for a fixed $n \geq 1$ we obtain that
\[
\beta_n
= \sum_{\pi \in NC(n)} g( \pi, 1_n) \,  \lambda_{\pi}  \\
= \sum_{\begin{array}{c}
{\scriptstyle \pi \in NC(n),} \\
{\scriptstyle \mathrm{irreducible}} 
\end{array}} c( \pi ) \, \lambda_{\pi},
\]
where the first equality is just spelling out the meaning of 
``$\uBeta = \uLambda \cdot g$'', and the second equality makes 
use of (\ref{eqn:76d}).  The formula for $\beta_n$ obtained in
this way is precisely the one stated in Equation (\ref{eqn:76c}).
\end{proof}

\vspace{6pt}

Here is how the preceding proposition applies to some
examples discussed in Section 7.  We mention that a rather
general result of this kind, going in a framework of cumulant
constructions related to trees, appears as Lemma 7.6 of 
\cite{JeLi2019}.

\vspace{6pt}

\begin{example}  \label{example:77}
{\em (Boolean cumulants in terms of $t$-Boolean cumulants.)}

\noindent
Let $t \in \bR$ be a parameter, and in Proposition \ref{prop:76}
let us make $\uLambda$ be the family of $t$-Boolean cumulant 
functionals of $( \cA , \varphi )$:
$\uLambda = \uBetat = ( \betat_n )_{n=1}^{\infty}$,
when the coefficients of interest are 
$c( \pi ) = t^{\innblocks ( \pi )}$ and
Equation (\ref{eqn:76a}) becomes the moment-cumulant formula
recorded in Definition \ref{def:69}.
The factorization condition from (\ref{eqn:76b}) is
holding; this corresponds precisely to the fact that the
function $\Fbcm^{(t)}$ introduced in Notation \ref{def:611} 
is of cumulant-to-moment type.  Thus Proposition \ref{prop:76} 
applies, and yields a formula expressing Boolean cumulants 
in terms of $t$-Boolean cumulants:
\begin{equation}   \label{eqn:77a}
\beta_n = \sum_{\begin{array}{c}
{\scriptstyle \pi \in NC(n),} \\
{\scriptstyle \mathrm{irreducible}} 
\end{array}}
\ t^{| \pi | - 1} \betat_{\pi}, \ \ n \geq 1,
\end{equation}
where on the right-hand side we took into account that 
an irreducible partition $\pi$ has 
$\innblocks ( \pi ) = | \pi | - 1$, and therefore 
has $c( \pi ) = t^{ | \pi | - 1 }$.  A generalization of this 
formula appears in Corollary \ref{cor:85} of the next section.
\end{example}

\vspace{6pt}

\begin{example}  \label{example:78}
{\em (Boolean cumulants in terms of monotone cumulants.)}

\noindent
In Proposition \ref{prop:76} let us make $\uLambda$ be the 
family of monotone cumulant functionals of $( \cA , \varphi )$:
$\uLambda = \uGamma 
= ( \gamma_n : \cA^n \to \bC )_{n=1}^{\infty}$,
when the coefficients of interest are
\[
c( \pi ) =
\frac{\mbox{\# of monotone orderings of $\pi$}}{ | \pi | !},
\ \ \pi \in \sqcup_{n=1}^{\infty} NC(n),
\]
and Equation (\ref{eqn:76a}) becomes the moment-cumulant 
formula recorded in Definition \ref{def:613}.
The factorization condition from (\ref{eqn:76b}) is
holding; this corresponds precisely to the fact that the
function $\Fmcm$ introduced in Notation \ref{def:613} 
is of cumulant-to-moment type.  Thus Proposition 
\ref{prop:76} applies, and yields a formula expressing 
Boolean cumulants in terms of monotone cumulants:
\begin{equation}   \label{eqn:78a}
\beta_n = \sum_{ \begin{array}{c}
{\scriptstyle \pi \in NC(n),} \\
{\scriptstyle \mathrm{irreducible}} 
\end{array}}
\ \frac{\mbox{\# of monotone orderings of $\pi$}}{ | \pi | !}
\cdot \gamma_{\pi}, \ \ n \geq 1.
\end{equation}
This retrieves Equation (1.6) of \cite{ArHaLeVa2015}, 
which is the beginning of the analysis done in that paper on
how to relate monotone cumulants to other brands of cumulants.

Equation (\ref{eqn:78a}) is equivalent to a formula 
that gives an explicit description of the function 
$\Fmcbc := \Fmcm * \Fbcm^{-1} \in \cGtildctoc$, encoding 
the transition from monotone cumulants to Boolean cumulants.
We mention that it is an interesting and non-trivial issue, 
addressed in \cite{ArHaLeVa2015} and in the recent paper 
\cite{CeEFPaPe2020}, to describe explicitly the inverse
\[
\Fmcbc^{-1} = \bigl( \, \Fmcm * \Fbcm^{-1} \, \bigr)^{-1}
= \Fbcm * \Fmcm^{-1} \in \cGtildctoc, 
\]
which encodes the reverse transition from Boolean to
monotone cumulants.
\end{example}

$\ $

\section{The 1-parameter subgroup $\{ u_q \mid q \in \bR \}$
of $\cGtild$, and its action on $\Fbcm^{(t)}$'s}

\noindent
The method used for studying the right coset $\cGtildctom$ in 
Section 8.2 draws attention to a 1-parameter family of functions
in $\cGtildctoc$, defined as follows.

\vspace{6pt}

\begin{notation}   \label{def:81}
For every $q \in \bR$, we denote by $u_q$ the function in 
$\cGtild$ which is determined via the requirement that for all 
$n \geq 1$ and $\pi \in NC(n)$ we have
\begin{equation}   \label{eqn:81a}
u_q ( \pi , 1_n ) = \left\{ \begin{array}{ll} 
q^{| \pi | -1},  & \mbox{ if $\pi$ is irreducible}  \\
0, & \mbox{ otherwise.}
\end{array}  \right.
\end{equation}
Clearly, this is a function of cumulant-to-cumulant type.
\end{notation}

\vspace{6pt}

\begin{remark}   \label{rem:82}
$1^o$ In the case $q=0$, the usual conventions 
apply to yield that $u_0 ( 1_n, 1_n) =1$ and 
$u_0 ( \pi, 1_n ) = 0$ for every $\pi \neq 1_n$ in $NC(n)$.
This implies that $u_0 = e$, the unit of $\cGtild$. 

\vspace{6pt}

$2^o$ The formula for the values taken by $u_q$ on 
general couples in $\NCint$ is determined from 
(\ref{eqn:81a}) by using the semi-multiplicativity property;
we leave it as an easy exercise to the reader to check that 
for every $n \geq 1$ and $\pi \leq \sigma$ in $NC(n)$ one gets:
\begin{equation}   \label{eqn:82a}
u_q ( \pi , \sigma ) =
\left\{   \begin{array}{ll}
q^{| \pi | - | \sigma |}, & \mbox{ if $\pi \ll \sigma$,} \\
0, & \mbox{ otherwise,}
\end{array}   \right.
\end{equation}
where $\ll$ is one of the partial order relations
reviewed in Section 2.2.
\end{remark}

\vspace{6pt}

\begin{proposition}   \label{prop:83}
The $u_q$'s form a 1-parameter subgroup of 
$\cGtild$:
\begin{equation}    \label{eqn:83a}
u_{q_1} * u_{q_2} = u_{q_1 + q_2},
\ \mbox{ for all $q_1, q_2 \in \bR$.}
\end{equation}
\end{proposition}

\begin{proof} We fix $q_1, q_2 \in \bR$ for which we 
will prove that (\ref{eqn:83a}) holds. 
The case when $q_1 =0$ or $q_2 = 0$ is clear, so 
we will assume that $q_1 \neq 0 \neq q_2$.

Since both $u_{q_1} * u_{q_2}$ and $u_{q_1 + q_2}$ 
are in $\cGtildctoc$, in order to prove their equality 
it will suffice (cf. Remark \ref{rem:63}) to check that
\begin{equation}    \label{eqn:83b}
\left\{    \begin{array}{l}
u_{q_1} * u_{q_2} \, ( \pi , 1_n )
= u_{q_1 + q_2} ( \pi, 1_n) \\
\mbox{ for every $n \geq 1$ and every irreducible 
$\pi \in NC(n)$.}
\end{array}   \right.
\end{equation}
For the rest of the proof, we fix an $n \geq 1$ and an 
irreducible $\pi \in NC(n)$ for which we will verify 
that Equation (\ref{eqn:83b}) holds.   

The right-hand side of 
(\ref{eqn:83b}) is, directly from the definitions, equal to 
$(q_1 + q_2)^{ | \pi | - 1 }$.  So our job is to 
verify that the left-hand side of (\ref{eqn:83b}) is 
equal to that same quantity.

We compute:
\[
u_{q_1} * u_{q_2} \, ( \pi , 1_n )
= \sum_{ \substack{\sigma \in NC(n), \\ \sigma \geq \pi} } 
\ u_{q_1} ( \pi , \sigma ) \cdot 
  u_{q_2} ( \sigma , 1_n )           
= \sum_{ \substack{\sigma \in NC(n), \\ \sigma \gg \pi} }
    \ q_1^{| \pi | - | \sigma |} 
    \cdot q_2^{ | \sigma | -1},
\]
where at the second equality sign we used (\ref{eqn:82a})
and also the fact that every $\sigma \geq \pi$ in $NC(n)$ 
is irreducible, thus has 
$u_{q_2} ( \sigma, 1_n) = q_2^{| \sigma | - 1}$.
In the latter summation over $\sigma$ we sort out the 
terms according to what is $| \sigma |$.  As reviewed
in Remark \ref{rem:17}.2, one has
\[
\vline \  \{ \sigma \in NC(n) \mid \sigma \gg \pi, \, 
| \sigma | = k \} \ \vline \ =  \left(  
\begin{array}{c} | \pi | - 1 \\ k-1 \end{array} \right) ,
\ \ \forall \, 
k \in \{ 1, \ldots , | \pi | \}.
\]
Hence our evaluation of the left-hand side of 
Equation (\ref{eqn:83b}) continues as follows:
\[
u_{q_1} * u_{q_2} \, ( \pi , 1_n )
= \sum_{k=1}^{| \pi |}  \binom{| \pi | - 1 }{k-1}
q_1^{ | \pi | - k}  q_2^{k-1} 
= \sum_{\ell = 0}^{| \pi | -1}  \binom{ | \pi | -1 }{\ell}
q_1^{ ( | \pi | -1 ) - \ell } q_2^{\ell} 
= (q_1+q_2)^{| \pi | -1},
\]
which is precisely the value we wanted to obtain. 
\end{proof}

\vspace{6pt}

We now consider the functions of cumulant-to-moment type
$\Fbcm^{(t)}$ introduced in Section 7.2, which encode 
moment-cumulant formulas for $t$-Boolean cumulants, and
we look at how our 1-parameter subgroup of $u_q$'s  
acts on $\Fbcm^{(t)}$'s, by left translations.

\vspace{6pt}

\begin{proposition}  \label{prop:84}
One has
\begin{equation}    \label{eqn:84a}
u_q * g^{(t)}_{\mathrm{bc-m}} 
= g^{(q+t)}_{\mathrm{bc-m}},
\ \mbox{ for all $q, t \in \bR$.}
\end{equation}
\end{proposition}

\begin{proof} We first verify the special case of 
(\ref{eqn:84a}) when $t=0$.  Since $\Fbcm^{(0)}$ is 
just the function $\Fbcm$ from Definition \ref{def:67},
this case amounts to checking that for 
every $q \in \bR$ we have
\begin{equation}   \label{eqn:84b}
u_q * \Fbcm =  \Fbcm^{(q)}.
\end{equation}
And indeed, let us notice that: $u_q$ 
is in $\cGtildctoc$, $\Fbcm^{(t)}$ is in 
$\cGtildctom$, and they are such that 
\[
\Fbcm^{(q)} ( \pi , 1_n) = q^{\innblocks ( \pi )}
= q^{ | \pi | -1 } = u_q ( \pi , 1_n ), \ \ \forall
\, \pi \in NC(n), \mbox{ irreducible.}
\]
Thus Equation (\ref{eqn:84b}) does hold, 
as a special case of Corollary \ref{cor:75}.

Going to general $q, t \in \bR$ we can then write:
\begin{align*}
u_q * \Fbcm^{(t)}
& = u_q * ( u_t * \Fbcm ) \mbox{ (by (\ref{eqn:84b}))}        \\
& = ( u_q * u_t ) * \Fbcm                                     \\
& = u_{q+t} * \Fbcm  \mbox{ (by Proposition \ref{prop:83})}  \\
& = \Fbcm^{(q+t)} \mbox{ (by (\ref{eqn:84b}))}, 
\end{align*}
yielding the required Equation (\ref{eqn:84a}).
\end{proof}

\vspace{6pt}

The preceding proposition yields, in particular, the 
explicit transition formula from $s$-Boolean cumulants 
to $t$-Boolean cumulants which was anticipated in 
Remark \ref{rem:612} of Section 7.2.  Recall that, in the 
said Remark \ref{rem:612}, the point that remained to be 
justified was the validity of Equation (\ref{eqn:612b});
this is now very easy to fill in.

\vspace{6pt}

\begin{corollary}   \label{cor:85}
{\em (A repeat of Equation (\ref{eqn:612b}).)}

\noindent
Let $s$ and $t$ be real parameters, and consider the functions
$\Fbcm^{(s)}, \Fbcm^{(t)} \in \cGtildctoc$.  For every $n \geq 1$
and $\pi \in NC(n)$ one has:
\[
\Fbcm^{(s)} * ( \Fbcm^{(t)} )^{-1} \,
( \pi , 1_n ) =
\left\{   \begin{array}{ll}
(s-t)^{| \pi | - 1}, & \mbox{ if $\pi$ is irreducible,} \\
0, & \mbox{ otherwise.}
\end{array}   \right.
\]
\end{corollary}

\begin{proof} 
Proposition \ref{prop:84} says that
$\Fbcm^{(s)} = u_{s-t} * \Fbcm^{(t)}$, which 
implies that
\begin{equation}   \label{eqn:85d}
\Fbcm^{(s)} * ( \Fbcm^{(t)} )^{-1} = u_{s-t}.
\end{equation}

\noindent
We evaluate both sides of (\ref{eqn:85d}) at 
$( \pi, 1_n )$, then we refer to the formula for 
$u_{s-t} ( \pi , 1_n )$ which comes from 
Equation (\ref{eqn:81a}), and the corollary follows.
\end{proof}

$\ $

\section{The action of $\{ u_q \mid q \in \bR \}$, 
by conjugation, on multiplicative functions}

\noindent 
The goal of the present section is to prove the following
result.

\vspace{6pt}

\begin{theorem}    \label{thm:91}
Let $\{ u_q \mid q \in \bR \}$ be as in the preceding section
and let $\cG$ be the subgroup of $\cGtild$ which consists of
multiplicative functions, as reviewed in Section 5.  One has
that:
\begin{equation}  \label{eqn:91a}
\Bigl( q \in \bR \mbox{ and } f \in \cG \Bigr)
\ \Rightarrow \ u_q^{-1} * f * u_q \in \cG.
\end{equation}
\end{theorem}

\vspace{6pt}

\begin{remark}   \label{rem:92}
Recall from Remark \ref{rem:42} and Proposition \ref{prop:43} 
that a function $f \in \cG$ is completely determined
by the sequence of complex numbers $\lambdans$, where 
$\lambda_n := f( 0_n, 1_n)$, $n \geq 1$.  If we accept
Theorem \ref{thm:91}, then it follows that $u_q^{-1} * f * u_q$ 
must be determined in a similar way by the sequence of 
$\theta_n$'s where $\theta_n := u_q^{-1} * f * u_q \, (0_n, 1_n)$
for $n \geq 1$.  It is easy to write down the explicit 
formula which gives $\theta_n$ in terms of 
$\lambda_1, \ldots , \lambda_n$ and $q$, this is:
\begin{equation}  \label{eqn:92a}
\theta_n = \sum_{\begin{array}{c}  
{\scriptstyle \pi \in NC(n)}  \\
{\scriptstyle \mathrm{irreducible}}
\end{array} } \ q^{| \pi | - 1} \prod_{V \in \pi} \lambda_{|V|}.
\end{equation}
This gives for instance: 
\[
\theta_1 = \lambda_1 = 1, \ \theta_2 = \lambda_2, 
\ \theta_3 = \lambda_3 + q \lambda_2,
\ \theta_4 = \lambda_4 + 2q \lambda_3 + q \lambda_2^2 + q^2 \lambda_2.
\]
Verification of (\ref{eqn:92a}): use the definition
of the convolution operation ``$*$'' to find that
\begin{equation}   \label{eqn:92b}
\theta_n  = u_q^{-1} * f * u_q \, (0_n, 1_n)
= \sum_{  \begin{array}{c} 
{\scriptstyle \sigma, \tau \in \ NC(n),}  \\
{\scriptstyle \sigma \leq \tau}
\end{array} } \ 
u_q^{-1} ( 0_n, \sigma) \, f( \sigma , \tau )
\, u_q ( \tau , 1_n ),
\end{equation}
then notice that $u_q^{-1} ( 0_n, \sigma ) = 0$ for every 
$\sigma \neq 0_n$ in $NC(n)$ (since $u_q^{-1} = u_{-q}$ and 
we can invoke the formula (\ref{eqn:82a})).  Thus $\sigma$ in 
(\ref{eqn:92b}) is forced to be $0_n$, and we continue with
\[
= \sum_{\tau \in NC(n)}  f(0_n , \tau ) u_q ( \tau , 1_n),
\]
which yields (\ref{eqn:92a}) upon replacing 
$f( 0_n , \tau )$ by $\prod_{V \in \tau} \lambda_{|V|}$ and 
$u_q ( \tau , 1_n )$ from (\ref{eqn:81a}).

Deriving the formula (\ref{eqn:92a}) for $\theta_n$ does not, 
however, substitute for a proof of Theorem \ref{thm:91}.
We still need to evaluate $u_q^{-1} * f * u_q \, ( \pi , 1_n)$
for general $\pi \in NC(n)$, and to suitably express the
resulting value as a product of $\theta_m$'s.  In order to achieve 
this we will prove two factorization formulas, presented in 
Lemmas \ref{lemma:95} and \ref{lemma:98} below.   

For Lemma \ref{lemma:95} we will need the following notation.
\end{remark}

\vspace{6pt}

\begin{notation-and-remark}   \label{def:93}
{\em (Irreducible cover of a non-crossing partition.)}

\vspace{6pt}

\noindent
Let $n \geq 1$ and let $\pi$ be in $NC(n)$.  

\vspace{6pt}

\noindent
$1^o$ It is easy to see that there exists a partition
$\pirrc \in NC(n)$, uniquely determined, with the following
properties:
\begin{equation}   \label{eqn:93a}
\left\{   \begin{array}{cl}
\mbox{(i)}  & \mbox{$\pirrc$ is irreducible and 
                     $\pirrc \geq \pi$;}                  \\
\mbox{(ii)}  & \mbox{If $\sigma \in NC(n)$ is irreducible and 
$\sigma \geq \pi$, then } \sigma \geq \pirrc .
\end{array}  \right.
\end{equation}
We will refer to $\pirrc$ as the {\em irreducible cover} 
of $\pi$.  For its explicit description we distinguish two 
cases.

\vspace{6pt}

{\em Case 1. $\pi$ is irreducible.}
Then, clearly, $\pirrc = \pi$.

\vspace{6pt}

{\em Case 2. $\pi$ is not irreducible.}  Then the blocks 
$\Vleft, \Vright \in \pi$ which contain the numbers 
$1$ respectively $n$ are such that $\Vleft \neq \Vright$. 
In this case, $\pirrc$ is obtained out of $\pi$ by merging together
$\Vleft$ and $\Vright$.  (It is easy to see that the said merger 
is sure to give a partition which is still non-crossing
and has all the properties required in (\ref{eqn:93a}).)

\vspace{6pt}

$2^o$ In what follows, we will need at some point to deal with the 
relative Kreweras complement of $\pi$ in $\pirrc$. In the case
when $\pirrc = \pi$, the complement $\Kr_{\pirrc} ( \pi )$ is
just $0_n$.  In the case when $\pirrc \neq \pi$ (i.e.~in the 
Case 2 indicated above), the complement 
$\Kr_{\pirrc} ( \pi )$ has $1$ block with $2$ elements and $n-2$ 
blocks with $1$ element.  Upon drawing a picture which features
the outer blocks of $\pi$, the reader should have no difficulty 
to check that, in Case 2, the unique $2$-element block of 
$\Kr_{\pirrc} ( \pi )$ is of the form $\{ m,n \}$, with $m$ 
described as follows:
\begin{equation}   \label{eqn:93b}
\left\{  \begin{array}{c}
\mbox{$m = \max ( V_{\mathrm{left}} ) 
         = \min ( W_{\mathrm{right}} )$, where }  \\
\mbox{$V_{\mathrm{left}}$ is the block
of $\pi$ which contains the number $1$, and }      \\   
\mbox{$W_{\mathrm{right}}$ is the block
of $\Kr ( \pi )$ which contains the number $n$.}
\end{array}   \right.
\end{equation}

$3^o$ The drawing of the outer blocks of $\pi$ that was recommended
above also reveals that the block 
$W_{\mathrm{right}} \in \Kr ( \pi )$ can be explicitly written in 
the form
\begin{equation}   \label{eqn:93c}
W_{\mathrm{right}} = \{ \max (U_1), \ldots , \max (U_k) \},
\end{equation}
where $U_1, \ldots, U_k$ are the outer blocks of $\pi$ 
(in particular, $U_1 = \Vleft$).
A consequence of (\ref{eqn:93c}) which will be needed in the sequel 
is this: if $\sigma \in NC(n)$ is such that $\sigma \gg \pi$ and if 
$W_{\mathrm{right}} '$ denotes the block of $\Kr ( \sigma )$ which
contains the number $n$, then it follows that
$W_{\mathrm{right}} ' = W_{\mathrm{right}}$.  This is because
the relation $\gg$ forces $\sigma$ to have the
same maximal elements of outer blocks as $\pi$ does, and thus the right-hand 
side of (\ref{eqn:93c}) also serves as an explicit description for what is
$W_{\mathrm{right}} '$. 
\end{notation-and-remark}

\vspace{6pt}

\begin{notation-and-remark}   \label{def:94}
In Lemma \ref{lemma:95} we will use three sequences of numbers
$\alphans$, $( \widehat{\alpha}_n )_{n=1}^{\infty}$
and $( \widetilde{\alpha}_n )_{n=1}^{\infty}$, where the 
$\widehat{\alpha}_n$'s and $\widetilde{\alpha}_n$'s are obtained
out of the $\alpha_n$'s via summation formulas, as follows:
\begin{equation}   \label{eqn:94a}
\widehat{\alpha}_n = \sum_{\pi \in NC(n)} 
\prod_{V \in \pi} \alpha_{|V|} 
\ \mbox{ $\ $ and $\ $ }
\ \widetilde{\alpha}_n = 
\sum_{  \begin{array}{c}
{\scriptstyle \pi \in NC(n),} \\
{\scriptstyle \mathrm{irreducible}}
\end{array} } 
\ \prod_{V \in \pi} \alpha_{|V|} , \ \ n \geq 1.
\end{equation}

One can also write summation formulas which 
give a direct relation between the two derived sequences 
$( \widehat{\alpha}_n )_{n=1}^{\infty}$ and 
$( \widetilde{\alpha}_n )_{n=1}^{\infty}$.  For future 
reference, we record here one such formula (which is not
hard to verify via direct calculation) saying that
\begin{equation}   \label{eqn:94b}
\widetilde{\alpha}_n = \sum_{\rho \in \Int (n)}
(-1)^{ | \rho | + 1 } \prod_{J \in \rho} \widehat{\alpha}_{|J|},
\ \ \forall \,  n \geq 1.
\end{equation}
\end{notation-and-remark}

\vspace{6pt}

\begin{lemma}   \label{lemma:95}
{\em (A factorization formula.)}
Consider sequences of numbers as in Notation \ref{def:94},
and on the other hand let us pick an $n \geq 1$ and a 
partition $\pi \in NC(n)$.  We consider 
the Kreweras complement $\Kr ( \pi )$ and, same as in 
Remark \ref{def:93}, we denote by $W_{\mathrm{right}}$ 
the block of $\Kr ( \pi )$ which contains the number $n$. 
Then:
\begin{equation}   \label{eqn:95a}
\sum_{  \begin{array}{c} 
{\scriptstyle \sigma \in \ NC(n),}  \\
{\scriptstyle \sigma \geq \pirrc}
\end{array} } 
\, \Bigl( \, \prod_{U \in \Kr_{\sigma} ( \pi )} \alpha_{|U|}
\, \Bigr) \ = \ \widetilde{\alpha}_{| W_{\mathrm{right}} | }
\cdot \prod_{ \begin{array}{c}
{\scriptstyle W \in \Kr ( \pi ),}  \\
{\scriptstyle W \not\ni n}
\end{array} } \
\widehat{\alpha}_{|W|}.
\end{equation}
\end{lemma}

\begin{proof}  On the left-hand side of (\ref{eqn:95a}) 
we perform the change of variable 
``$\tau = \Kr_{\sigma} ( \pi )$''.  When $\sigma$ runs 
in the interval $[ \pirrc , 1_n ] \subseteq NC(n)$, 
the relative Kreweras complement $\tau$ runs in the interval 
$[ \Kr_{\pirrc} ( \pi ), \Kr_{1_n} ( \pi )]$, where
$\Kr_{1_n} ( \pi )$ is just $\Kr ( \pi )$.  For a discussion 
of this nice behaviour of the partition 
$\Kr_{\sigma} ( \pi )$ viewed as a function of $\sigma$ (and 
with $\pi$ fixed) see \cite[Lemma 18.9]{NiSp2006}.

Let us also recall, from Remark \ref{def:93}.2, that the inequality 
$\tau \geq \Kr_{\pirrc} ( \pi )$ amounts to requesting that $\tau$
connects $m$ with $n$, where $m = \min ( W_{\mathrm{right}} )$.
Our processing of the left-hand side of Equation (\ref{eqn:95a})
has thus taken us to:
\begin{equation}   \label{eqn:95b}
\sum_{  \begin{array}{c} 
{\scriptstyle \tau \in \ NC(n), \, \tau \leq \Kr ( \pi )}  \\
{\scriptstyle \mathrm{and} \ \tau \ \mathrm{connects} 
              \ m \ \mathrm{with} \ n}
\end{array} } 
\ \prod_{U \in \tau} \alpha_{|U|}.
\end{equation}

Now let us write explicitly 
$\Kr ( \pi ) = \{ W_1, \ldots , W_p \}$,
with the blocks listed such that $W_p = W_{\mathrm{right}}$.  
A standard decomposition argument shows that a partition 
$\tau \in NC(n)$ such that $\tau \leq \Kr ( \pi )$ is 
bijectively identified to the tuple
\[
( \tau_{ { }_{W_1} }, \ldots , \tau_{ { }_{W_p} } )
\in NC( |W_1| ) \times \cdots \times NC( |W_p| )
\]
which records the relabeled-restrictions of $\tau$ to the
blocks $W_1, \ldots , W_p$.  At the level of the tuple 
$( \tau_{ { }_{W_1} }, \ldots , \tau_{ { }_{W_p} } )$,
the requirement that ``$\tau$ connects $m$ with $n$'' 
(where $m$ and $n$ are the minimal and maximal elements of 
the block $W_{\mathrm{right}} = W_p$) is transformed into 
the requirement that $\tau_{ { }_{W_p} } \in NC( |W_p| )$ 
is irreducible.  We leave it as a straightforward exercise
to the reader to check that, upon performing the change of 
variable $\tau \leftrightarrow
( \tau_{ { }_{W_1} }, \ldots , \tau_{ { }_{W_p} } )$
in the summation from (\ref{eqn:95b}), one gets 
precisely the product of $p$ separate summations which is 
indicated on the right-hand side of (\ref{eqn:95a}).
\end{proof}

\vspace{6pt}

\begin{notation-and-remark}   \label{def:96}
The second factorization formula that we want to use is 
presented in Lemma \ref{lemma:98}.  We find it convenient 
to first prove this lemma in a special case, stated separately
as Lemma \ref{lemma:97}.  In these lemmas we use two sequences
of complex numbers, $( \gamma_n )_{n=1}^{\infty}$ and 
$( \widehat{\gamma}_n )_{n=1}^{\infty}$, with 
$\gamma_1 = \widehat{\gamma}_1 \neq 0$, and where the 
$\widehat{\gamma}_n$'s are expressed in terms of $\gamma_n$'s
by summations over interval partitions, as follows:
\begin{equation}   \label{eqn:96a}
\widehat{\gamma}_n = \sum_{\rho \in \Int (n)}
\, \Bigl( \, \prod_{J \in \rho} 
\gamma_{{ }_{|J|}} \, \Bigr), \ \ \forall \, n \geq 1.
\end{equation}
\end{notation-and-remark}

\vspace{6pt}

\begin{lemma}  \label{lemma:97}
Consider the framework of Notation \ref{def:96}, 
and on the other hand consider an $n \geq 1$ and 
an irreducible partition $\pi \in NC(n)$.   Then:
\begin{equation}   \label{eqn:97a}
\sum_{ \begin{array}{c}
{\scriptstyle \sigma \in NC(n),}  \\
{\scriptstyle \sigma \gg \pi}
\end{array} } \ \Bigl(  
\, \prod_{\begin{array}{c}
{\scriptstyle U \in \Kr ( \sigma ),}  \\
{\scriptstyle U \not\ni n}
\end{array}  } \gamma_{{ }_{|U|}} \, \Bigr)
= \prod_{\begin{array}{c}
{\scriptstyle W \in \Kr ( \pi ),}  \\
{\scriptstyle W \not\ni n}
\end{array}  }  \widehat{\gamma}_{{ }_{|W|}}.
\end{equation}
\end{lemma}

\begin{proof} Due to the hypothesis that $\pi$ is irreducible,
the Kreweras complement $\Kr ( \pi )$ has a singleton block $\{ n \}$.
The same is true for any Kreweras complement $\Kr ( \sigma )$ with
$\sigma \geq \pi$ (since $\sigma$ will have to be irreducible as well).
Hence when we multiply the left-hand side of (\ref{eqn:97a}) by 
$\gamma_1$ and the right-hand side of (\ref{eqn:97a}) by 
$\widehat{\gamma}_1$, where $\gamma_1 = \widehat{\gamma}_1 \neq 0$,
we find that (\ref{eqn:97a}) is equivalent to
\begin{equation}   \label{eqn:97b}
\sum_{ \begin{array}{c}
{\scriptstyle \sigma \in NC(n),}  \\
{\scriptstyle \sigma \gg \pi}
\end{array} } \ \Bigl(  
\, \prod_{U \in \Kr ( \sigma )}
\gamma_{{ }_{|U|}} \, \Bigr)
= \prod_{W \in \Kr ( \pi )}  \widehat{\gamma}_{{ }_{|W|}}.
\end{equation}
It is thus all right if we prove (\ref{eqn:97b})
instead of (\ref{eqn:97a}).

We now recall the poset anti-isomorphism (\ref{eqn:17a})
given by Kreweras complementation between $\ll$ and 
$\sqsubseteq$, where $\ll$ is considered on irreducible 
partitions in $NC(n)$ while $\sqsubseteq$ is considered 
on non-crossing partitions which have $\{ n \}$ as a 
1-element block.  Since the set 
$\{ \sigma \in NC(n) \mid \sigma \gg \pi \}$ only 
contains irreducible partitions, we can use 
(\ref{eqn:17a}) as a change of variable in the 
summation on the left-hand side of (\ref{eqn:97b}),
which is thus transformed into
\begin{equation}   \label{eqn:97c}
\sum_{ \begin{array}{c}
{\scriptstyle \tau \in NC(n),}  \\
{\scriptstyle \tau \sqsubseteq \Kr ( \pi )}
\end{array} } \ \Bigl(  
\, \prod_{U \in \tau} \gamma_{{ }_{|U|}} \, \Bigr).
\end{equation}
From here on we proceed with a variation of the argument that
finalized the proof of Lemma \ref{lemma:95}: we list explicitly
the blocks of $\Kr ( \pi )$ as  $W_1, \ldots , W_p$, and we use
the fact that a $\tau \in NC(n)$ with $\tau \leq \Kr ( \pi )$ is 
bijectively identified to the tuple
\[
( \tau_{ { }_{W_1} }, \ldots , \tau_{ { }_{W_p} } )
\in NC( |W_1| ) \times \cdots \times NC( |W_p| ).
\]
At the level of the latter tuple, the requirement 
``$\tau \sqsubseteq \Kr ( \pi )$'' amounts to asking that 
$\tau_{ { }_{W_1} }, \ldots , \tau_{ { }_{W_p} }$ are 
interval partitions.  Performing the change of 
variable $\tau \leftrightarrow
( \tau_{ { }_{W_1} }, \ldots , \tau_{ { }_{W_p} } )$
in the summation from (\ref{eqn:97c}) thus takes us to a 
summation over 
$\Int ( |W_1| ) \times \cdots \times \Int ( |W_p| )$.
We leave it as a straightforward exercise to the reader to 
check that the latter summation factors as the product of 
$p$ separate summations over 
$\Int ( |W_1| ), \ldots , \Int ( |W_p| )$, and that one 
obtains in this way the product 
indicated on the right-hand side of (\ref{eqn:97b}).
\end{proof}

\vspace{6pt}

\begin{lemma}  \label{lemma:98}
The factorization formula stated in Equation (\ref{eqn:97a})
of Lemma \ref{lemma:97} holds even if we do not assume
the partition $\pi$ to be irreducible.
\end{lemma}

\begin{proof} Consider the canonical decomposition
$\pi = \pi_1 \diamond \cdots \diamond \pi_k$ 
with $\pi_1 \in NC(n_1), \ldots, \pi_k \in NC(n_k)$ 
irreducible, as reviewed in Remark \ref{def:13}.3.
The specifics of the partial order 
$\ll$ force that we have
\begin{equation}   \label{eqn:98a}
\{ \sigma \in NC(n) \mid \sigma \gg \pi \} =
\Bigl\{ \sigma_1 \diamond \cdots \diamond \sigma_k 
\begin{array}{lr}
\vline  &
\sigma_1 \gg \pi_1 \mbox { in } NC(n_1), \ldots ,  \\
\vline  &
\sigma_k \gg \pi_k \mbox { in } NC(n_k) 
\end{array}   \Bigr\} .
\end{equation}
For a partition 
$\sigma = \sigma_1 \diamond \cdots \diamond \sigma_k$
as in (\ref{eqn:98a}) we note that 
$\sigma_1, \ldots , \sigma_k$
are all irreducible, and an examination of the relevant
Kreweras complements leads to the formula
\begin{equation}  \label{eqn:98b}
\prod_{ \begin{array}{c}
{\scriptstyle U \in \Kr ( \sigma ), }  \\
{\scriptstyle U \not\ni n}
\end{array} } \gamma_{|U|} \ = 
\ \prod_{j=1}^k \Bigl(
\, \prod_{ \begin{array}{c}
{\scriptstyle U \in \Kr ( \sigma_j ), }  \\
{\scriptstyle U \not\ni n_j}
\end{array} }  \gamma_{|U|} \, \Bigr).
\end{equation}
The observations from (\ref{eqn:98a}), (\ref{eqn:98b})
and a straightforward conversion of sum into product 
then imply that we have:
\begin{equation}    \label{eqn:98c}
\sum_{ \begin{array}{c}
{\scriptstyle \sigma \in NC(n),}  \\
{\scriptstyle \sigma \gg \pi}
\end{array} }  
\, \prod_{\begin{array}{c}
{\scriptstyle U \in \Kr ( \sigma ),}  \\
{\scriptstyle U \not\ni n}
\end{array} }   \gamma_{{ }_{|U|}}
\ = \ \prod_{j=1}^k \Bigl( 
\sum_{ \begin{array}{c}
{\scriptstyle \sigma_j \in NC(n_j),}  \\
{\scriptstyle \sigma_j \gg \pi_j}
\end{array} }  
\, \prod_{\begin{array}{c}
{\scriptstyle U \in \Kr ( \sigma_j ),}  \\
{\scriptstyle U \not\ni n_j}
\end{array}  } \gamma_{{ }_{|U|}}
\, \Bigr).
\end{equation}
But now, Lemma \ref{lemma:97} can be applied to each of
$\pi_1, \ldots , \pi_k$.  When we do this, we find that
Equation (\ref{eqn:98c}) can be continued with
\[
= \ \prod_{j=1}^k \Bigl(
\, \prod_{\begin{array}{c}
{\scriptstyle W \in \Kr ( \pi_j ),}  \\
{\scriptstyle W \not\ni n_j}
\end{array}  }  \widehat{\gamma}_{{ }_{|W|}} \, \Bigr)
\ = \ \prod_{\begin{array}{c}
{\scriptstyle W \in \Kr ( \pi ),}  \\
{\scriptstyle W \not\ni n}
\end{array}  }  \widehat{\gamma}_{{ }_{|W|}},
\]
where at the second equality sign we used the counterpart of
(\ref{eqn:98b}) in connection to the numbers 
$\widehat{\gamma}_i$, and for the decomposition
$\pi = \pi_1 \diamond \cdots \diamond \pi_k$.
\end{proof}

$\ $

\begin{ad-hoc}   \label{proof:99}
{\bf Proof of Theorem \ref{thm:91}.}
We fix a $q \in \bR$ and an $f \in \cG$ for which we will prove
that $u_q^{-1} * f * u_q \in \cG$.  The case when $q=0$ is clear,
since $u_0^{-1} * f * u_0 = f$, so we assume $q \neq 0$.

Let us denote $\lambda_n := f(0_n, 1_n)$, $n \geq 1$, and let 
$\thetans$ be the sequence of complex numbers obtained out of the 
$\lambda_n$'s by using the formula (\ref{eqn:92a}) from Remark 
\ref{rem:92}. As anticipated in that remark, we will obtain the 
desired conclusion about $u_q^{-1} * f * u_q$ by proving that
\begin{equation}   \label{eqn:99a}
u_q^{-1} * f * u_q \, ( \pi , 1_n )
= \prod_{W \in \Kr ( \pi )} \theta_{|W|}, \ \ \forall 
\, n \geq 1 \mbox{ and } \pi \in NC(n).
\end{equation}

From now on and until the end of the proof we fix an 
$n \geq 1$ and a $\pi \in NC(n)$ for which we will verify that
(\ref{eqn:99a}) holds.  We divide the argument into several steps.

\vspace{6pt}

\noindent
{\em Step 1. Write explicitly what is 
$u_q^{-1} * f * u_q \, ( \pi , 1_n )$, as
a double sum ``over $\sigma$ and $\tau$''.}

\noindent
Similarly to the derivation of Equation (\ref{eqn:92b}), we 
start from
\begin{equation}   \label{eqn:99b}
u_q^{-1} * f * u_q \, (\pi , 1_n)
= \sum_{  \begin{array}{c} 
{\scriptstyle \sigma, \tau \in \ NC(n)}  \\
{\scriptstyle \mathrm{ such \ that} \ \pi \leq \sigma \leq \tau}
\end{array} } \ 
u_q^{-1} ( \pi, \sigma) \, f( \sigma , \tau )
\, u_q ( \tau , 1_n ).
\end{equation}
We have
\[
u_q^{-1} ( \pi, \sigma ) = u_{-q} (\pi, \sigma )
= \left\{ \begin{array}{ll} 
(-q)^{ | \pi | - | \sigma | }, & \mbox{ if $\pi \ll \sigma$,} \\
0, & \mbox{ otherwise.}
\end{array} \right.
\]
We plug this into the right-hand side of (\ref{eqn:99b}), and
also replace the values of $f( \sigma , \tau )$ and of 
$u_q ( \tau , 1_n )$ by using
(\ref{eqn:42c}) and (\ref{eqn:81a}), respectively.  In this 
way we arrive to the formula:
\begin{equation}   \label{eqn:99c}
u_q^{-1} * f * u_q \, (\pi , 1_n)
= \sum_{  \begin{array}{c} 
{\scriptstyle \sigma \in \, NC(n),}  \\
{\scriptstyle  \sigma \gg \pi}
\end{array} }  (-q)^{ | \pi | - | \sigma |} \Bigl( 
\sum_{  \begin{array}{c} 
{\scriptstyle \tau \in \ NC(n), \, \tau \geq \sigma}  \\
{\scriptstyle  \mathrm{and} \, \tau \, \mathrm{irreducible}}
\end{array} } 
\, \bigl( \prod_{W \in \Kr_{\tau} ( \sigma )} \lambda_{|V|} \bigr)
\cdot q^{ | \tau | - 1 }  \Bigr).
\end{equation}

In the summation over $\tau$ performed in 
(\ref{eqn:99c}), the conditions 
``$\tau \geq \sigma$'' and ``$\tau$ irreducible'' are 
consolidated in the requirement that 
$\tau \geq \sigmairrc$. Let us also re-arrange the factor 
$q^{| \tau | - 1}$ appearing in that summation: we have 
(cf.~\cite[Exercise 18.23]{NiSp2006})  
$| \sigma | + | \Kr_{\tau} ( \sigma ) | = | \tau | + n$,
which implies that
$q^{| \tau |-1} = q^{| \Kr_{\tau} ( \sigma ) |} \cdot
q^{| \sigma | - (n+1)}$.
With these changes, we thus arrive to 
\begin{equation}   \label{eqn:99d}
u_q^{-1} * f * u_q \, (\pi , 1_n)
= \sum_{  \begin{array}{c} 
{\scriptstyle \sigma \in \ NC(n),}  \\
{\scriptstyle  \sigma \gg \pi}
\end{array} } 
\ \frac{(-q)^{ | \pi | - | \sigma |}}{q^{ (n+1) - | \sigma |}} 
\cdot \Bigl( \, \sum_{  \begin{array}{c} 
{\scriptstyle \tau \in \ NC(n),}  \\
{\scriptstyle \tau \geq \sigmairrc}
\end{array} } 
\ \prod_{W \in \Kr_{\tau} ( \sigma )} 
                     (q \lambda_{|V|}) \, \Bigr).
\end{equation}

\vspace{6pt}

\noindent
{\em Step 2. Use the factorization formula from 
Lemma \ref{lemma:95}.}

\vspace{6pt} 

\noindent
Here we must first clarify what are the input sequences 
``$\alpha_k$, $\widehat{\alpha}_k$, $\widetilde{\alpha}_k$'' 
that we plan to use in Lemma \ref{lemma:95}.  We go as 
follows: start from the sequence $( \lambda_k )_{k=1}^{\infty}$ 
which was fixed from the beginning of the proof and put 
$\alpha_k := q \lambda_k$ for every $k \geq 1$; after that, 
define sequences $( \widehat{\alpha}_k )_{k=1}^{\infty}$ 
and $( \widetilde{\alpha}_k )_{k=1}^{\infty}$ via the 
formulas (\ref{eqn:94a}) given in Notation \ref{def:94}.

In view of what are our $\alpha_k$'s, we re-write
(\ref{eqn:99d}) in the form
\[
u_q^{-1} * f * u_q \, ( \pi , 1_n)
= \sum_{  \begin{array}{c} 
{\scriptstyle \sigma \in \ NC(n),}  \\
{\scriptstyle  \sigma \gg \pi}
\end{array} } 
\ \frac{ (-1)^{ | \pi | - | \sigma | } 
         \, q^{ | \pi | - | \sigma | } }{ q^{(n+1)-| \sigma |} } 
\cdot \Bigl( \, \sum_{  \begin{array}{c} 
{\scriptstyle \tau \in \ NC(n),}  \\
{\scriptstyle \tau \geq \sigmairrc}
\end{array} } 
\ \prod_{W \in \Kr_{\tau} ( \sigma )} 
                     \alpha_{|W|}) \, \Bigr),
\]
and we invoke Lemma \ref{lemma:95} in order to 
continue with
\begin{equation}   \label{eqn:99e}
= \sum_{  \begin{array}{c} 
{\scriptstyle \sigma \in \ NC(n),}  \\
{\scriptstyle  \sigma \gg \pi}
\end{array} } 
(-1)^{ | \pi | } \, 
\ (-1)^{ | \sigma | } \, q^{ | \pi | - (n+1) } 
\cdot \Bigl( \, 
\widetilde{\alpha}_{| W_{\mathrm{right}} | } \cdot
\prod_{ \begin{array}{c}
{\scriptstyle W \in \Kr ( \sigma ),}  \\
{\scriptstyle W \not\ni n}
\end{array} } \
\widehat{\alpha}_{|W|} \, \Bigr).
\end{equation}

In the expression we arrived to, note that we can pull 
to the front of the sum the factors $(-1)^{ | \pi | }$, 
$q^{ | \pi | - (n+1) }$ and 
$\widetilde{\alpha}_{| W_{\mathrm{right}} | }$.  The 
justification for pulling out the latter factor comes from 
Remark \ref{def:93}.3 -- the block $W_{\mathrm{right}}$ is
the same for all the partitions $\sigma$ with $\sigma \gg \pi$.
Thus from (\ref{eqn:99e}) we go on with
\begin{equation}   \label{eqn:99f}
= (-1)^{ | \pi | } \, q^{ | \pi | - (n+1) } 
\, \widetilde{\alpha}_{| W_{\mathrm{right}} | }
\cdot \sum_{  \begin{array}{c} 
{\scriptstyle \sigma \in \ NC(n),}  \\
{\scriptstyle  \sigma \gg \pi}
\end{array} } \ (-1)^{ | \sigma | }
\cdot \Bigl( \, \prod_{ \begin{array}{c}
{\scriptstyle W \in \Kr ( \sigma ),}  \\
{\scriptstyle W \not\ni n}
\end{array} } \ \widehat{\alpha}_{|W|} \, \Bigr) .
\end{equation}

\vspace{6pt}

\noindent
{\em Step 3. Use the factorization formula from 
Lemma \ref{lemma:98}.}

\vspace{6pt} 

\noindent
Here we must clarify what are the input sequences 
``$\gamma_k$ and $\widehat{\gamma}_k$'' 
that we plan to use in Lemma \ref{lemma:98}.  We go as 
follows: put $\gamma_k = - \widehat{\alpha}_k$ for all
$k \geq 1$; after that, define the sequence 
$( \widehat{\gamma}_k )_{k=1}^{\infty}$ via the 
formula (\ref{eqn:96a}) indicated in Notation \ref{def:96}.
Observe that the common value of $\gamma_1$ and 
$\widehat{\gamma}_1$ is equal to $-q$ (this is found by 
backtracking in the definitions: $\gamma_1 = 
- \widehat{\alpha}_1 = - \alpha_1 = -q \lambda_1 = -q$).
Since it is assumed that $q \neq 0$, we are thus in a situation
where the hypotheses of Lemma \ref{lemma:98} are satisfied.

Next observation: in (\ref{eqn:99f}), the factor 
$(-1)^{ | \sigma | }$ can be written as
\[
(-1)^n \cdot (-1)^{ n - | \sigma | }
= (-1)^n \cdot (-1)^{ | \Kr ( \sigma ) | - 1 },
\]
where at the second equality we use the fact that one
always has $| \sigma | + | \Kr ( \sigma ) | = n+1$.
The $(-1)^{ | \Kr ( \sigma ) | -1 }$ can be absorbed
into the product of $\widehat{\alpha}_{ |W| }$'s (which 
has $| \Kr ( \sigma ) | -1$
factors), and therefore (\ref{eqn:99f}) continues with
\[
= (-1)^{ | \pi | } \, q^{ | \pi | - (n+1) } 
\, \widetilde{\alpha}_{| W_{\mathrm{right}} | }
\cdot \sum_{  \begin{array}{c} 
{\scriptstyle \sigma \in \ NC(n),}  \\
{\scriptstyle  \sigma \gg \pi}
\end{array} } \ (-1)^n
\cdot \Bigl( \, \prod_{ \begin{array}{c}
{\scriptstyle W \in \Kr ( \sigma ),}  \\
{\scriptstyle W \not\ni n}
\end{array} } \ (- \widehat{\alpha}_{|W|}) \, \Bigr) 
\]
\begin{equation}   \label{eqn:99g}
= (-1)^{ n - | \pi | } \, q^{ | \pi | - (n+1) } 
\, \widetilde{\alpha}_{| W_{\mathrm{right}} | }
\cdot \sum_{  \begin{array}{c} 
{\scriptstyle \sigma \in \ NC(n),}  \\
{\scriptstyle  \sigma \gg \pi}
\end{array} } 
\Bigl( \, \prod_{ \begin{array}{c}
{\scriptstyle W \in \Kr ( \sigma ),}  \\
{\scriptstyle W \not\ni n}
\end{array} } \ \gamma_{|W|}) \, \Bigr) . 
\end{equation}

The sum over $\sigma \gg \pi$ in (\ref{eqn:99g}) is 
precisely the one to which Lemma \ref{lemma:98} applies,
and in this way we arrive to the conclusion of Step 3,
which is that we have 
\begin{equation}   \label{eqn:99h}
u_q^{-1} * f * u_q \, ( \pi , 1_n )
= (-1)^{ n - | \pi | } \, q^{ | \pi | - (n+1) } 
\, \widetilde{\alpha}_{| W_{\mathrm{right}} | }
\cdot 
\prod_{ \begin{array}{c}
{\scriptstyle U \in \Kr ( \pi ),}  \\
{\scriptstyle U \not\ni n}
\end{array} } \ \widehat{\gamma}_{|U|}. 
\end{equation}

\vspace{6pt}

\noindent
{\em Step 4. Identify the factors in the 
product found in (\ref{eqn:99h}).}

\noindent
It is convenient to re-write the right-hand side of 
(\ref{eqn:99h}) in the form
\begin{equation}   \label{eqn:99i}
\bigl( \frac{1}{q} 
\widetilde{\alpha}_{| W_{\mathrm{right}} | } \bigr)
\cdot 
\prod_{ \begin{array}{c}
{\scriptstyle U \in \Kr ( \pi ),}  \\
{\scriptstyle U \not\ni n}
\end{array} }  \bigl( - \frac{1}{q} \widehat{\gamma}_{|U|} \bigr), 
\end{equation}
with the pre-factor
$(-1)^{ n - | \pi | } \, q^{ | \pi | - (n+1) }$ distributed
among the $(n+1) - | \pi |$ blocks of $\Kr ( \pi )$.

We are then left to chase through the formulas used in 
Steps 2 and 3, and verify that the product over blocks of 
$\Kr ( \pi )$ that appears in (\ref{eqn:99i}) is the same as the 
one on the right-hand side of our target Equation (\ref{eqn:99a})
indicated at the beginning of the proof.  It is visible that 
everything would be in place if we had that:
\begin{equation}   \label{eqn:99j}
\widetilde{\alpha}_k = q \theta_k \ \mbox{ $\ $ and $\ $ }
\ \widehat{\gamma}_k = -q \theta_k, \ \ \forall \, k \geq 1.
\end{equation}
We will argue that the desirable relations stated in (\ref{eqn:99j})
are indeed holding.

The first relation (\ref{eqn:99j}) comes out by direct 
comparison of the formulas defining $\widetilde{\alpha}_k$ 
and $\theta_k$.  Indeed, upon replacing 
$\alpha_{|V|} = q \lambda_{|V|}$ in the formula (\ref{eqn:94a})
which defines $\widetilde{\alpha}_k$, we find that
\[
\widetilde{\alpha}_k 
= \sum_{  \begin{array}{c}
{\scriptstyle \rho \in NC(k),} \\
{\scriptstyle \mathrm{irreducible}}
\end{array} } 
\prod_{V \in \rho} (q \lambda_{|V|})
= \sum_{  \begin{array}{c}
{\scriptstyle \rho \in NC(k),} \\
{\scriptstyle \mathrm{irreducible}}
\end{array} } 
q^{| \rho |} \prod_{V \in \rho} \lambda_{|V|}
= q \, \theta_k,
\]
where at the third equality sign we refer to the formula 
(\ref{eqn:92a}) for $\theta_k$. 

For the second relation (\ref{eqn:99j}) it suffices to check 
that $\widehat{\gamma}_k = - \widetilde{\alpha}_k$.
We have
\begin{align*}
\widehat{\gamma}_k 
& = \sum_{\rho \in \Int (k)} \prod_{J \in \rho} \gamma_{|J|}     
    \mbox{ $\ $ (by the definition of $\widehat{\gamma}_k$, 
          in Equation (\ref{eqn:96a}))}                           \\
& = \sum_{\rho \in \Int (k)} \prod_{J \in \rho} 
    ( - \widehat{\alpha}_{|J|} )                                  
    \mbox{ $\ $ (by the definition of $\gamma_{|J|}$ in Step 3)}   \\
& = \sum_{\rho \in \Int (k)} (-1)^{ | \rho | }
          \prod_{J \in \rho} \widehat{\alpha}_{|J|} = 
          - \widetilde{\alpha}_k,
\end{align*}
where the latter equality follows from Equation (\ref{eqn:94b})
of Remark \ref{def:94}.
\hfill $\square$
\end{ad-hoc}

$\ $

\section{An application: multiplication of free random variables,
in terms of $t$-Boolean cumulants}

\noindent
As explained in Section 1.5 of the Introduction, the 
multiplication of free random variables has a nice 
description in terms of $t$-Boolean cumulants, by a formula
which is actually the same for all values of $t$. 
In the present section we show how this fact can be neatly derived 
by using the 1-parameter subgroup $\{ u_q \mid q \in \bR \}$  
of $\cGtildctoc$.  

\vspace{6pt}

\begin{notation-and-remark}   \label{def:101}
{\em (Framework, and discussion of what we will prove.)}

\noindent
We fix for the whole section a non-commutative probability 
space $( \cA , \varphi )$ and two unital subalgebras 
$\cM , \cN \subseteq \cA$ which are freely independent 
with respect to $\varphi$.  For every $t \in \bR$ we 
consider the family
$\uBetat = ( \betat_n : \cA^n \to \bC )_{n=1}^{\infty}$
of $t$-Boolean cumulants of $( \cA , \varphi )$; we also 
consider the standard enlargement of $\uBetat$ to
$( \betat_{\pi} : \cA^n \to \bC )_{n \geq 1, \, \pi \in NC(n)}$,
as discussed in Notation \ref{rem:52}.2.  It will be convenient 
to aim for a formula slightly more general
than what was announced in Equation (\ref{eqn:intro6}) of the 
Introduction, and which is stated as follows:
 
\begin{equation}   \label{eqn:101a}
\left\{  \begin{array}{c}
\mbox{One has } \betat_n (x_1 y, \ldots, x_n y)
= \sum_{\pi \in NC(n)} \betat_{\pi} (x_1, \ldots , x_n)
\cdot \betat_{\Kr ( \pi)} (y, \ldots , y),                \\
\mbox{ $\ $ holding for every $n \in \bN$, 
$x_1, \ldots , x_n \in \cM$, $y \in \cN$ and 
$t \in \bR$.}
\end{array}   \right.
\end{equation}

Our approach to (\ref{eqn:101a}) is this: we note 
that for fixed $y$ and $t$, the family of equalities 
stated in (\ref{eqn:101a}) is equivalent to one equation 
concerning the action of the group $\cGtild$ on the space 
$\fMcM$ of sequences of multilinear functionals on $\cM$. 
The latter equation can then be treated by using results 
from Sections 9 and 10, particularly Theorem \ref{thm:91}.

In order for the trick of fixing a $y$ to play smoothly
into the setting from Sections 9 and 10, it is good to arrange
that $\varphi (y) = 1$.  We start by pointing out that, 
without loss of generality, we can make this assumption.
\end{notation-and-remark}

\begin{lemma}   \label{lemma:102}
Assume it is true that (\ref{eqn:101a}) holds whenever 
$y \in \cN$ has $\varphi (y) = 1$.  Then (\ref{eqn:101a}) 
is sure to hold with $y \in \cN$ arbitrary.
\end{lemma}

\begin{proof}  We first extend the validity of 
(\ref{eqn:101a}) to the case when $\varphi (y) \neq 0$.  
If $\varphi (y) = \lambda \neq 0$
then for every $t \in \bR$, $n \in \bN$ and 
$x_1, \ldots , x_n \in \cM$ we have
\begin{align*}
\betat_n (x_1 y, \ldots , x_n y)
& = \betat_n \bigl( \, ( \lambda x_1) \cdot (\lambda^{-1}y), 
  \ldots , ( \lambda x_n ) \cdot ( \lambda^{-1} y ) \, \bigr)    \\
& = \sum_{\pi \in NC(n)}
    \betat_{\pi} (\lambda x_1, \ldots , \lambda x_n) \cdot
    \betat_{\Kr ( \pi )} ( \lambda^{-1} y, \ldots , 
                           \lambda^{-1} y )          \\
& \mbox{ $\ $ (by hypothesis, since 
               $\varphi ( \lambda^{-1} y ) = 1$) }   \\
& = \sum_{\pi \in NC(n)}
    \lambda^n \betat_{\pi} (x_1, \ldots , x_n) \cdot
    \lambda^{-n} \betat_{\Kr ( \pi )} (y, \ldots , y) \\
& = \sum_{\pi \in NC(n)}
    \betat_{\pi} (x_1, \ldots , x_n) \cdot
    \betat_{\Kr ( \pi )} (y, \ldots , y), \ \mbox{ as required.}
\end{align*}

Now consider a $y \in \cN$ with $\varphi (y) = 0$.  From the fact proved in the preceding paragraph, it follows that: for 
every $t \in \bR$, $n \in \bN$ and $x_1, \ldots , x_n \in \cM$, 
one has
\begin{equation}   \label{eqn:102a}
\betat_n \bigl( x_1 (y + \delta \oneA), \ldots , 
                x_n (y+ \delta \oneA) \bigr)  
\end{equation}
\[
= \sum_{\pi \in NC(n)}
    \betat_{\pi} (x_1, \ldots , x_n) \cdot
    \betat_{\Kr ( \pi )} (y+ \delta \oneA, \ldots , y+ \delta \oneA),
\ \ \forall \, \delta \neq 0 \mbox{ in } \bC.
\]
It is easy to check that the two sides of (\ref{eqn:102a})
depend continuously (in fact polynomially) on $\delta$.  We 
can thus make $\delta \to 0$ in (\ref{eqn:102a}), to conclude
that (\ref{eqn:101a}) holds for this $y$ as well.
\end{proof}

\vspace{6pt}

\begin{notation}   \label{def:103}
$1^o$ For the remaining part of this section we fix 
an element $y \in \cN$ with $\varphi (y) = 1$, in 
connection to which we will prove that (\ref{eqn:101a}) 
is holding.  

\vspace{6pt}

\noindent
$2^o$ It is convenient that, by using the $y$ which was 
fixed, we introduce some sequences of multilinear 
functionals on $\cM$, as follows: for every 
$t \in \bR$ and $n \in \bN$, let 
$\gamma^{(t, \cM )}_n : \cM^n \to \bC$ be defined by 
\begin{equation}   \label{eqn:103a}
\gamma^{(t, \cM )}_n (x_1, \ldots , x_n) 
= \betat_n (x_1 y, \ldots , x_n y), \ \ \forall
\, x_1, \ldots , x_n \in \cM.
\end{equation}
Clearly, we have that
$\uGamma^{(t, \cM )} := 
( \gamma^{(t, \cM)}_n )_{n=1}^{\infty} \in \fMcM$, where 
$\fMcM$ is defined exactly as in Notation \ref{def:51}, 
but by using $\cM$ instead of $\cA$.

In the same vein, it is convenient that for every $t \in \bR$
and $n \geq 1$ we use the notation 
$\beta^{(t, \cM )}_n : \cM^n \to \bC$
for the restriction of the multilinear functional 
$\betat_n : \cA^n \to \bC$ to the subspace $\cM^n$.  Then 
$\uBeta^{(t, \cM )} := ( \beta^{(t, \cM)}_n )_{n=1}^{\infty}
\in \fMcM$, and (as immediately verified) it is the family of 
$t$-Boolean cumulants of the 
non-commutative probability space $( \cM , \varphi | \cM )$.

\vspace{6pt}

\noindent
$3^o$ Recall from Section 5 that every sequence $\alphans$ of 
complex numbers, with $\alpha_1 = 1$, defines a multiplicative 
function $f \in \cG$ via the requirement that 
$f(0_n, 1_n) = \alpha_n$ for all $n \geq 1$.  For every 
$t \in \bR$ we can therefore consider a multiplicative 
function $f_t \in \cG$ defined via the requirement that
\begin{equation}   \label{eqn:103b}
f_t (0_n, 1_n) = \betat_n (y, \ldots , y), \ \ \forall 
\, n \geq 1,
\end{equation}
where $y$ is the element of $\cN$ fixed in part $1^o$ of
this notation.  Note that when defining $f_t$ we use the
fact that $\varphi (y) = 1$, which ensures that the sequence 
of numbers proposed on the right-hand side of (\ref{eqn:103b})
does indeed start with $\betat_1 (y) = \varphi (y) = 1$.

For every $t \in \bR$ and for general $\pi \leq \sigma$ 
in some $NC(n)$, an explicit formula giving $f_t ( \pi , \sigma )$ 
is then obtained out of (\ref{eqn:103b}), in the way reviewed 
in Remark \ref{rem:42}.  Recall, in particular, that for every 
$n \geq 1$ and $\pi \in NC(n)$ we have
\begin{equation}   \label{eqn:103c}
f_t ( \pi, 1_n) 
= \prod_{W \in \Kr ( \pi )} f_t (0_{|W|}, 1_{|W|} ) 
= \prod_{W \in \Kr ( \pi )} \betat_{|W|} (y, \ldots , y) 
= \betat_{\Kr ( \pi )} (y, \ldots , y).
\end{equation}
\end{notation}

\vspace{6pt}

In terms of the notation just introduced, we can give 
an equivalent form of (\ref{eqn:101a}), which is stated 
as follows.

\vspace{6pt}

\begin{lemma}   \label{lemma:104}
For every $t \in \bR$, one has:
\begin{equation}   \label{eqn:104a}
\left(  \begin{array}{c}
\mbox{Formula (\ref{eqn:101a}) holds for our} \\
\mbox{fixed $y$ and this particular value of $t$}
\end{array}  \right) \ \Leftrightarrow
\ \left(   \begin{array}{c}
\uGamma^{(t, \cM)} = \uBeta^{(t, \cM)} \cdot f_t  \\  
\mbox{ (an equality in $\fMcM$)}
\end{array}   \right)  .
\end{equation}
\end{lemma}

\begin{proof}  The equality stated on the right-hand side 
of the equivalence is spelled out as follows: 
\begin{equation}   \label{eqn:104b}
\left\{  \begin{array}{r}
\gamma^{(t, \cM)}_n (x_1, \ldots , x_n)
= \sum_{\pi \in NC(n)}  f_t ( \pi, 1_n) 
\, \beta^{(t, \cM)}_{\pi} (x_1, \ldots , x_n), \\
                                               \\
\mbox{ holding for every $n \geq 1$ and
$x_1, \ldots , x_n \in \cM$.}
\end{array}   \right.
\end{equation}
We leave it as an immediate exercise to the reader to replace
the various quantities mentioned in (\ref{eqn:104b}) by their 
definition from Notation \ref{def:103}, and to verify that what
comes out is indeed equivalent to the instance of (\ref{eqn:101a}) 
referring to our fixed $y$ and $t$.
\end{proof}

\vspace{6pt}

We next examine how one can connect two instances of the 
equation appearing on the right-hand side of the equivalence
(\ref{eqn:104a}), considered for two different values 
$s,t \in \bR$.  This is done by using the 1-parameter subgroup 
$\{ u_q \mid q \in \bR \}$ from the preceding sections, both in
reference to $\beta^{(t, \cM )}, \gamma^{(t, \cM )}$ (in 
Lemma \ref{lemma:105}) and in reference to $f_t$ (in 
Lemma \ref{lemma:106}). 

\vspace{6pt} 

\begin{lemma}   \label{lemma:105}
For every $s,t \in \bR$ we have
\begin{equation}   \label{eqn:105a}
\uBeta^{(t, \cM)} = \uBeta^{(s, \cM)} \cdot u_{s-t} 
\mbox{ and } 
\uGamma^{(t, \cM)} = \uGamma^{(s, \cM)} \cdot u_{s-t}.
\end{equation}
\end{lemma}

\begin{proof}  The first formula (\ref{eqn:105a}) is 
a direct consequence of Corollary \ref{cor:85}, written
in connection to the non-commutative probability space
$( \cM , \varphi \mid \cM )$.

The second formula (\ref{eqn:105a}) also follows from 
Corollary \ref{cor:85}.  Indeed, for every $n \geq 1$ and 
$x_1, \ldots , x_n \in \cM$ we can write
\[
\gamma^{(t, \cM)}_n (x_1, \ldots , x_n)
\ = \ \betat_n (x_1 y, \ldots , x_n y) 
\ = \ \sum_{  \begin{array}{c}
{\scriptstyle \pi \in NC(n),} \\
{\scriptstyle \mathrm{irreducible}}  \end{array}  } 
\ (s-t)^{| \pi | -1} \cdot 
\beta^{(s)}_{\pi} (x_1 y, \ldots , x_n y),
\]
where at the second equality sign we use 
Equation (\ref{eqn:612c}) from Remark \ref{rem:612}.
An inspection of the definition of the functionals
$\beta^{(s)}_{\pi}$ and $\gamma^{(s)}_{\pi}$ shows that
in the latter expression we can replace
$\beta^{(s)}_{\pi} (x_1 y, \ldots , x_n y)$ with
$\gamma^{(s, \cM)}_{\pi} (x_1, \ldots , x_n)$; hence 
what we got is
\[
\gamma^{(t, \cM)}_n (x_1, \ldots , x_n)
\ = \ \sum_{  \begin{array}{c}
{\scriptstyle \pi \in NC(n),} \\
{\scriptstyle \mathrm{irreducible}}  \end{array}  } 
\ (s-t)^{| \pi | -1} \cdot 
\gamma^{(s)}_{\pi} (x_1, \ldots , x_n),
\]
where the right-hand side is indeed the value at 
$(x_1, \ldots , x_n)$ of the $n$-th functional in the 
family $\uGamma^{(s , \cM)} \cdot u_{s-t} \in \fMcM$.
\end{proof}

\vspace{6pt}

\begin{lemma}   \label{lemma:106}
Let $t,q$ be in $\bR$.  One has that
$u_q^{-1} * f_t * u_q = f_{t-q}$.
\end{lemma}

\begin{proof}
We have that $f_{t-q}$ is multiplicative (by definition,
cf.~Notation \ref{def:103}.3) and $u_q^{-1} * f_t * u_q$
is multiplicative as well (due to Theorem \ref{thm:91}); so
in order to prove their equality, it suffices to verify that
\begin{equation}   \label{eqn:106a}
u_q^{-1} * f_t * u_q \, (0_n, 1_n) 
= f_{t-q} (0_n, 1_n), \ \ \forall \, n \geq 1.
\end{equation}

The right-hand side of (\ref{eqn:106a}) is, by definition,
equal to $\beta^{(t-q)}_n (y, \ldots , y)$.  For the 
left-hand side of the same equation we resort to 
Equation (\ref{eqn:92a}) of Remark \ref{rem:92}, which 
says that
\begin{equation}  \label{eqn:106b}
u_q^{-1} * f_t * u_q \, (0_n, 1_n) =
\sum_{ \begin{array}{c}
{\scriptstyle \pi \in NC(n),}  \\
{\scriptstyle \mathrm{irreducible}}
\end{array}  } \ q^{| \pi | - 1} \prod_{V \in \pi} 
f_t ( 0_{|V|}, 1_{|V|} ).
\end{equation}
Upon replacing $f_t ( 0_{|V|} , 1_{|V|} )$ from its
definition, the right-hand side of (\ref{eqn:106b}) becomes
\[
\sum_{ \substack{\pi \in NC(n),  \\ \mathrm{irreducible}} } 
\ q^{| \pi | - 1}  \prod_{V \in \pi}
\beta^{(t)}_{|V|} (y, \ldots , y),
\]
and this is indeed equal to 
$\beta^{(t-q)}_n (y, \ldots , y)$, thanks to 
Equation (\ref{eqn:612c}) of Remark \ref{rem:612}.
\end{proof}

\vspace{6pt}

\begin{lemma}       \label{lemma:107}
Suppose there exists a value 
$t_o \in \bR$ for which it is true that
$\uGamma^{(t_o, \cM)} = \uBeta^{(t_o, \cM)} \cdot f_{t_o}$.
Then it follows that
$\uGamma^{(t, \cM)} = \uBeta^{(t, \cM)} \cdot f_t$
for all $t \in \bR$.
\end{lemma}

\begin{proof} 
Fix a $t \in \bR$.  We use Lemmas \ref{lemma:105} and 
\ref{lemma:106}, with $q := t_o -t$, to replace 
$\uBeta^{(t, \cM)} = \uBeta^{(t_o, \cM)} \cdot u_{t_o -t}$ 
and $f_t = u_{t_o -t}^{-1} * f_{t_o} * u_{t_o -t}$, and 
thus get:
\[
\uBeta^{(t, \cM)} \cdot f_t
= ( \uBeta^{(t_o, \cM)} \cdot u_{t_o - t} ) 
\cdot ( u_{t_o -t}^{-1} * f_{t_o} * u_{t_o -t})         
= ( \uBeta^{(t_o, \cM)} \cdot f_{t_o} ) \cdot  u_{t_o -t}.
\]
In the latter expression we can replace
$\uBeta^{(t_o, \cM)} \cdot f_{t_o}$ with $\uGamma^{(t_o, \cM)}$
(by hypothesis), then we can invoke Lemma \ref{lemma:105} to 
conclude that 
$\uGamma^{(t_o, \cM)} \cdot  u_{t_o -t} = \uGamma^{(t, \cM)}$.
In this way we obtain that
$\uBeta^{(t, \cM)} \cdot f_t = \uGamma^{(t, \cM)}$, as required.
\end{proof}

\vspace{6pt}

\begin{ad-hoc}
{\bf Proof of the statement (\ref{eqn:101a}).}  
In view of Lemma \ref{lemma:102}, it suffices to prove 
(\ref{eqn:101a}) in connection to the element $y \in \cN$ 
with $\varphi (y) = 1$ which was fixed since Notation 
\ref{def:103}.

The special case $t_o =1$ of (\ref{eqn:101a}) concerns the 
description of multiplication of free elements in terms 
of free cumulants.  This is a basic result in the 
combinatorics of free probability, which is not hard to 
obtain via a suitable grouping of terms in the moment-cumulant
formula for free cumulants, followed by an application of 
M\"obius inversion.  For the details of how this goes, see 
for instance \cite[Theorem 14.4]{NiSp2006}.

We therefore accept the case $t_o =1$ in (\ref{eqn:101a}).
The equivalence noticed in Lemma \ref{lemma:104} then tells 
us that that the equality 
$\uBeta^{(t_o, \cM)} \cdot f_{t_o} = \uGamma^{(t_o, \cM)}$
holds for $t_o = 1$.  This puts us in the position to invoke 
Lemma \ref{lemma:107}, in order to conclude that the equality
$\uBeta^{(t, \cM)} \cdot f_t = \uGamma^{(t, \cM)}$ holds for 
every $t \in \bR$.  Finally, the equivalence noticed in 
Lemma \ref{lemma:104} is used again (this time in the 
direction from right to left) to conclude that (\ref{eqn:101a})
holds for all values $t \in \bR$, as required.
\hfill $\square$
\end{ad-hoc}

\vspace{6pt}

\begin{remark}   \label{rem:10x}
In the formula (\ref{eqn:intro6}) of the Introduction, the roles
played by the elements $x,y \in \cA$ were similar to each other.
This symmetry was broken when we moved to the more general statement in
(\ref{eqn:101a}), where we continue to work with $(y, \ldots , y) \in \cA^n$
but we use a tuple $(x_1, \ldots , x_n)$ instead of just $(x, \ldots , x)$. 
In connection to that, we mention that (\ref{eqn:101a}) can be further
extended to the following statement:
\begin{equation}   \label{eqn:10xa}
\left\{  \begin{array}{c}
\mbox{ Let $\cM, \cN \subseteq \cA$ be as in (\ref{eqn:101a}). 
    One has that }  \\
\beta^{(t)}_n (x_1 y_1, \ldots, x_n y_n)
= \sum_{\pi \in NC(n)} \beta^{(t)}_{\pi} (x_1, \ldots , x_n)
\cdot \beta^{(t)}_{\Kr ( \pi)} (y_1, \ldots , y_n),                    \\
\mbox{ $\ $ holding for every $n \in \bN$, 
$x_1, \ldots , x_n \in \cM$, $y_1, \ldots , y_n \in \cN$ and 
$t \in \bR$.}
\end{array}   \right.
\end{equation}

The proof shown above for (\ref{eqn:101a}) does not cover the more
general statement (\ref{eqn:10xa}), because our handling of the
multiplicative function $f_t$ makes effective use (e.g.~when
considering the functionals $\beta^{(t)}_{|W|}$ in (\ref{eqn:103c}))
of the fact that $y_1 = \cdots = y_n =y$.  
For a reader who is interested to pursue this, we outline
below a possible approach to (\ref{eqn:10xa}), which is however
straying a bit outside the main body of ideas of the paper, and 
requires some work around a certain 
``$t$-Boolean Bercovici-Pata bijection'' that was introduced
in \cite{BeNi2009}.

Let us quickly review some notation from 
\cite[Lectures 16 and 17]{NiSp2006}.  We consider a sheer algebraic 
setting, with a ``space of distributions'' defined as
\[
\Dalg (n) := \{ \mu : \bC \langle X_1, \ldots , X_n \rangle \to \bC
\mid \mu \mbox{ linear, } \mu (1) =1 \}.
\]
Every $\mu \in \Dalg (n)$ has an {\em $R$-transform} $R_{\mu}$
which belongs to the space 
$\bC_o \langle \langle z_1, \ldots , z_n \rangle \rangle$ 
of formal power series without constant coefficient in the
non-commuting indeterminates $z_1, \ldots , z_n$. The series
$R_{\mu}$ is put together by using the free cumulants of $\mu$ 
as coefficients (cf.~\cite[Definition 16.3]{NiSp2006}).  The 
multiplication of freely independent $n$-tuples of elements
in a non-commutative probability space is encoded by a binary
operation $\boxtimes$ on $\Dalg (n)$. Then, upon taking
$R$-transforms, $\boxtimes$ is turned into 
a certain binary operation $\freestar$ on power series:
\begin{equation}   \label{eqn:10xb}
R_{\mu \boxtimes \nu} = R_{\mu} \, \freestar \, R_{\nu}, 
\ \ \forall \, \mu, \nu \in \Dalg (n).
\end{equation}
Moreover: for $f,g \in  
\bC_o \langle \langle z_1, \ldots , z_n \rangle \rangle$,
the coefficients of $f \freestar g$ can be explicitly described 
in terms of the coefficients of $f$ and of $g$ via a formula which
is reminiscent of (\ref{eqn:10xa}) -- 
cf.~\cite[Definition 17.1 and Proposition 17.2]{NiSp2006}.

For our discussion here it is relevant that for every 
$\mu \in \Dalg (n)$ and $t \in \bR$
we can define an {\em $\eta^{(t)}$-transform},
\[
\eta_{\mu}^{(t)} \in
\bC_o \langle \langle z_1, \ldots , z_n \rangle \rangle;
\]
the series $\eta_{\mu}^{(t)}$ is put together
by using the $t$-Boolean cumulants of $\mu$ as coefficients.
The $R$-transform is retrieved at $t=1$, $\eta^{(1)}_{\mu} = R_{\mu}$.
The point of relevance for the proof of (\ref{eqn:10xa}) is
that one can extend Equation (\ref{eqn:10xb}) from the case $t=1$ to
the case of a general $t \in \bR$:
\begin{equation}   \label{eqn:10xe}
\eta^{(t)}_{\mu \boxtimes \nu}
= \eta^{(t)}_{\mu}  \freestar \eta^{(t)}_{\nu}
\ \ \forall \, t \in \bR \mbox{ and }
\mu, \nu \in \Dalg (n).
\end{equation}

Verification that (\ref{eqn:10xa}) follows from (\ref{eqn:10xe}):
consider the setting from (\ref{eqn:10xa}), and let 
$\mu, \nu \in \Dalg (n)$ be the joint distributions of the tuples
$(x_1, \ldots , x_n)$ and $(y_1, \ldots , y_n)$, respectively. 
The definition of $\boxtimes$ ensures that the joint distribution 
of $(x_1 y_1, \ldots , x_n y_n)$ is $\mu \boxtimes \nu$. 
Thus $\beta_n^{(t)} (x_1 y_1, \ldots , x_n y_n)$ is retrieved as the
coefficient of $z_1 \cdots z_n$ in  $\eta^{(t)}_{\mu \boxtimes \nu}$,
and in view of (\ref{eqn:10xe}) we get that
\begin{equation}   \label{eqn:10xg}
\beta_n^{(t)} (x_1 y_1, \ldots , x_n y_n)
= \mbox{ [Coefficient of $z_1 \cdots z_n$ in  
$\eta^{(t)}_{\mu} \freestar \eta^{(t)}_{\nu}$] }. 
\end{equation}
From (\ref{eqn:10xg}), the explicit description of how $\freestar$ 
works takes us precisely to the right-hand side of the formula
indicated in (\ref{eqn:10xa}).

Now, the reason for reducing (\ref{eqn:10xa}) 
to (\ref{eqn:10xe}) is that the latter formula can be studied in 
connection to a family of bijective maps 
$\bigl( \, \bB_t : \Dalg (n) \to \Dalg (n) \, \bigr)_{t \in [ 0, \infty )}$
introduced in \cite{BeNi2009}.  These maps form a semigroup 
($\bB_s \circ \bB_t = \bB_{s+t}$ for all $s, t \in [0, \infty )$), 
and have the property that
\begin{equation}   \label{eqn:10xc}
\bB_t ( \mu \boxtimes \nu ) = \bB_t ( \mu ) \boxtimes \bB_t ( \nu),
\ \ \forall \, t \in [ 0, \infty) \mbox{ and } \mu, \nu \in \Dalg (n).
\end{equation}
When $t=1$, the map $\bB_1$ is known as 
``Boolean Bercovici-Pata bijection'', and has the following description 
(the idea of which can be tracked back all the way to \cite{BePa1999}):
\begin{equation}   \label{eqn:10xd1}
\left\{   \begin{array}{c}
\mbox{ For every $\mu \in \Dalg (n)$, we have that $\bB_1 (\mu )$ is }  \\
\mbox{ the unique distribution $\nu \in \Dalg (n)$ such that
       $R_{\nu} = \eta^{(0)}_{\mu}$.}
\end{array}   \right.
\end{equation}
A reader interested in pursuing this line of thought should be able to
develop the relevant $NC(n)$-combinatorics presented in \cite{BeNi2009}
and obtain the following generalization of (\ref{eqn:10xd1}): 
\begin{equation}   \label{eqn:10xdt}
\left\{   \begin{array}{c}
\mbox{ For every $t \in [0, \infty )$ and $\mu \in \Dalg (n)$, 
                                we have that $\bB_t (\mu )$ is }  \\
\mbox{ the unique distribution $\nu \in \Dalg (n)$ such that
       $\eta^{(t)}_{\nu} = \eta^{(0)}_{\mu}$.}
\end{array}   \right.
\end{equation}

By using (\ref{eqn:10xdt}) and the machinery of the $\bB_t$'s, it is then
rather straightforward to upgrade from (\ref{eqn:10xb}) to the case of 
(\ref{eqn:10xe}) with $t \in [0, 1]$.  Indeed, if 
$\mu, \nu \in \Dalg (n)$ and $t \in [0,1]$ are given, one puts
$\mu ' := \bB_{1-t} ( \mu )$ and $\nu ' := \bB_{1-t} ( \nu )$, and 
processes the equality 
$R_{\mu ' \boxtimes \nu '} = R_{\mu '} \freestar R_{\nu'}$
into becoming 
$\eta^{(t)}_{\mu \boxtimes \nu} 
= \eta^{(t)}_{\mu} \freestar \eta^{(t)}_{\nu}$.
Finally, it is also straightforward to observe that (\ref{eqn:10xe}) can
be expressed in the guise of a family of polynomial identities in $t$; and
if such an identity holds for all $t \in [0,1]$, then it must actually hold
for all $t \in \bR$.
\end{remark}

\vspace{6pt}

\begin{remark}   \label{rem:109}
Upon seeing how things came up in (\ref{eqn:101a}),
one is prompted to ask the analogous question in connection
to the other important brand of cumulants mentioned in 
Section 7, the monotone cumulants.  More precisely: let 
$( \cA , \varphi )$, the freely independent unital subalgebras
$\cM , \cN \subseteq \cA$ and the element $y \in \cN$ be the
same as above, and consider the sequence of monotone 
cumulant functionals 
$\uRho = ( \rho_n : \cA^n \to \bC)_{n=1}^{\infty}$.  Is there
a nice formula which expresses a monotone cumulant 
\[
\rho_n (x_1 y, \ldots , x_n y), \ \mbox{ with $n \geq 1$ and }
x_1, \ldots , x_n \in \cM,
\]
in terms of the monotone cumulants of the $x_i$'s (on the 
one hand) and the monotone cumulants of $y$ (on the other 
hand)?  It may seem intriguing that low order calculations 
show an analogy with (\ref{eqn:101a}): one has
\begin{equation}   \label{eqn:109a}
\left\{   \begin{array}{l}
\rho_n( x_1 y, \ldots , x_n y) = \sum_{\pi \in NC(n)} 
\rho_{\pi} ( x_1, \ldots , x_n ) \cdot 
\rho_{ { }_{\Kr (\pi)} } (y, \ldots , y),  \\
\mbox{ for all $n \leq 4$ and $x_1, \ldots , x_n \in \cM$.}
\end{array}   \right.
\end{equation}

\noindent
This turns out to be an accident which no longer holds
for $n \geq 5$.  In an appendix at the end of the paper 
we show the output of some computer calculations which 
check the difference of the two sides of (\ref{eqn:109a})
for $5 \leq n \leq 8$ and in the case when 
$x_1 = \cdots = x_n =:x \in \cM$.  It is probably another
low-dimensional accident that the ``irregular'' terms
appearing in this difference don't seem to be that numerous.  
In any case, it would be interesting to have a theorem 
establishing an analogue of (\ref{eqn:101a}) for monotone 
cumulants; this theorem would then also explain the 
structure of the irregular terms shown in the appendix.
\end{remark}

$\ $

\section{Identifying $\cGtild$ as character 
group of a Hopf algebra}

\noindent
The machinery of incidence algebras on posets can be 
re-cast in a way which uses Hopf algebra considerations.  
More precisely: the main object studied in the present 
paper, the group $\cGtild$, will now be identified in a 
natural way as the group of characters a Hopf algebra 
$\cT$.  The construction of $\cT$ is quite direct, when 
one pursues the following guidelines:

-- As an algebra: $\cT$ should have a universality property
which makes it that the characters of $\cT$ are parametrized 
by functions from $\cGtild$.

-- As a coalgebra: the comultiplication of $\cT$ has to play
into the formula (\ref{eqn:26a}) which governs the group
operation on $\cGtild$.

The construction of $\cT$ works seamlessly due to certain 
underlying properties of the lattices $NC(n)$.  This falls within
the framework of ``hereditary family of posets'' in the sense
developed by Schmitt \cite{Schm1994}, and consequently $\cT$ 
is an incidence Hopf algebra in the sense of that paper.
A version of $\cT$ has also been recently studied in the paper 
\cite{EFFoKoPa2019}. 

$\ $

\subsection{Description of \boldmath{$\cT$}.}

$\ $

\noindent
This subsection describes the Hopf algebra $\cT$ we are 
interested in, with detailed explicit formulas for the algebra 
and coalgebra operations.  We assume the reader to be familiar 
with basic notions and facts concerning Hopf algebras with 
combinatorial flavour, as presented for instance in
\cite[Section 1B and Chapter 14]{GBVaFi2001} or in 
\cite[Chapters I and II]{Ma2003}.

\vspace{6pt}

\begin{notation-and-remark}  \label{def:111}
% {\em (Description of $\cT$.)}
$1^o$ We let $\cT$ be the commutative algebra of 
polynomials over $\bC$ which uses a countable 
collection of indeterminates indexed by non-crossing
partitions with at least two blocks:
\begin{equation}   \label{eqn:111a}
\cT := \bC \Bigl[ \, X_{\pi} \mid 
\pi \in \sqcup_{n=1}^{\infty} 
\bigl( NC(n) \setminus \{ 1_n \} \bigr) \, \Bigr].
\end{equation}
We also make the convention to denote
\begin{equation}  \label{eqn:111b}
X_{1_n} :=  1_{ { }_{\cT} }, \ \ 
\forall \, n \geq 1,
\end{equation}
and thus get to have elements $X_{\pi} \in \cT$ 
defined for all the partitions in 
$\sqcup_{n=1}^{\infty} NC(n)$.

\vspace{6pt}

\noindent
$2^{o}$ As an immediate consequence of how notation 
is set in $1^{o}$, $\cT$ has 
a universality property described as follows: 
\begin{equation}   \label{eqn:111c}
\left\{   \begin{array}{l}
\mbox{If $\cA$ is a unital commutative algebra 
over $\bC$ and we are given}   \\
\mbox{ $\ $ elements 
$\bigl\{ a_{\pi} \mid \pi \in \sqcup_{n=1}^{\infty} NC(n) 
\bigr\}$ in $\cA$, with 
$a_{ { }_{1_n} } = 1_{ { }_{\cA} }$ for all $n \geq 1$,}  \\
\mbox{then there exists a unital algebra homomomorphism 
$\Phi : \cT \to \cA$, uniquely}                         \\
\mbox{ $\ $ determined, such that 
$\Phi (X_{\pi}) = a_{\pi}$ for all 
$\pi \in \sqcup_{n=1}^{\infty} NC(n)$.}
\end{array}   \right.
\end{equation}

\vspace{6pt}

\noindent
$3^o$ Consider the unital algebra $\cT \otimes \cT$.
The universality property of $\cT$ observed in $2^o$ 
assures us that there exists a unital algebra
homomorphism $\Delta : \cT \to \cT \otimes \cT$, 
uniquely determined, such that for every $n \geq 1$
and $\pi \in NC(n)$ we have 
\begin{equation}   \label{eqn:111d}
\Delta (X_{\pi}) =  \sum_{\sigma \in NC(n), \, \sigma \geq \pi} 
\ \Bigl( \prod_{W \in \sigma} X_{ \pi_{{ }_W} } \Bigr) 
\otimes X_{\sigma} ,
\end{equation}
with the relabeled-restrictions $\pi_{{ }_W} \in NC( |W|)$
considered in the sense of Notation \ref{def:12}.  
We will refer to the homomorphism $\Delta$ by 
calling it ``the comultiplication of $\cT$''.

\vspace{6pt}

\noindent
$4^o$ The universality property observed in $2^o$ 
also assures us that there exists a unital algebra
homomorphism $\eeps : \cT \to \bC$, 
uniquely determined, such that 
\begin{equation}   \label{eqn:111e}
\eeps ( X_{\pi} ) = 0, \ \ \forall \, 
\pi \in \sqcup_{n=1}^{\infty} 
( NC(n) \setminus \{ 1_n \} ).
\end{equation}
We will refer to $\eeps$ by calling it 
``the counit of $\cT$''.

\vspace{6pt}

\noindent
$5^o$ We denote $\cT_0 := 
\{ \lambda \cdot 1_{{ }_{\cT}} \mid \lambda \in \bC \}$,
and for every $m \geq 1$ we denote
\begin{equation}   \label{eqn:111f}
\cT_m = \mbox{span} \left( \;\,\bigcup_{k=1}^m  \Bigl\{ 
X_{\pi_1} \cdots X_{\pi_k} 
\begin{array}{ll}
\vline & \pi_1, \ldots , \pi_k \in 
         \sqcup_{n=1}^{\infty} ( NC(n) \setminus \{1 _n \}) \\
\vline & \mbox{with } | \pi_1 | + \cdots + | \pi_k | = m+k
\end{array}  \Bigr\} \right) .
\end{equation}
In other words, $\cT_m$ is the linear span of all monomials 
``of degree $m$'', where we declare that every indeterminate 
$X_{\pi}$ has degree $| \pi | - 1$.  This gives a direct sum 
decomposition
$\cT = \bigoplus_{m=0}^{\infty} \cT_m$,
which we will refer to as ``the grading of $\cT$''.
\end{notation-and-remark}

\vspace{6pt}

\begin{notation-and-remark}   \label{def:112}
For every $n \geq 1$ and 
$\pi \leq \sigma$ in $NC(n)$ let us denote
\begin{equation}   \label{eqn:112x}
M_{\pi, \sigma} :=
\prod_{W \in \sigma} X_{ \pi_{{ }_W} } \in \cT.
\end{equation}
Note that for $\sigma = 1_n$ the monomial 
$M_{\pi , \sigma}$ consists
of only one factor, so we get
\[
M_{\pi, 1_n} = X_{\pi}, \ \ \forall \, n \geq 1
\mbox{ and } \pi \in NC(n).
\]
At the other extreme, setting $\sigma = \pi$ makes 
$M_{\pi, \sigma}$ consist of factors $X_{1_{|V|}}$ with
$V$ running among the blocks of $\pi$, and we thus get
\[
M_{\pi,\pi} = \oneT, \ \ \forall \, n \geq 1
\mbox{ and } \pi \in NC(n).
\]

In terms of the monomials $M_{\pi, \sigma}$, the formula 
(\ref{eqn:111d}) defining the comultiplication of $\cT$ 
takes the more appealing form
\begin{equation}   \label{eqn:112y}
\Delta (X_{\pi}) =  \sum_{\sigma \in NC(n), \, \sigma \geq \pi} 
\ M_{\pi , \sigma} \otimes X_{\sigma},
\ \mbox{ for all $n \geq 1$ and $\pi \in NC(n)$.}
\end{equation}
It is easy to further extend this, in the way indicated in
the next lemma.
\end{notation-and-remark}

\vspace{6pt}

\begin{lemma}   \label{lemma:113}
Let $n \geq 1$ and let $\pi, \tau \in NC(n)$ be such that
$\pi \leq \tau$. Then
\begin{equation}   \label{eqn:113a}
\Delta ( M_{\pi, \tau} ) \ =
\ \sum_{\sigma \in NC(n), \, \pi \leq \sigma \leq \tau} 
\ M_{\pi, \sigma} \otimes M_{\sigma , \tau}.
\end{equation}
\end{lemma}

\begin{proof}  Let us write explicitly
$\tau = \{ U_1, \ldots , U_k \}$ and let us denote 
$\pi^{(j)} := \pi_{{ }_{U_j}} \in NC( |U_j| )$ for 
$1 \leq j \leq k$. The left-hand 
side of Equation (\ref{eqn:113a}) then becomes
\[
\Delta ( \prod_{j=1}^k X_{{ }_{\pi^{(j)}}} )
= \prod_{j=1}^k \Delta ( X_{{ }_{\pi^{(j)}}} )
= \prod_{j=1}^k \Bigl( \sum_{ 
          \substack{\sigma^{(j)} \in NC(|U_j|),    \\ 
                   \sigma^{(j)} \geq \pi^{(j)} }  }
                   M_{\pi^{(j)}, \sigma^{(j)}} 
\otimes X_{\sigma^{(j)}} \, \Bigr).
\]
Expanding the product over $j$ in the latter expression takes us to
\begin{equation}   \label{eqn:113b}
\sum_{ \substack{
  \sigma^{(1)} \in NC(|U_1|), \ldots, \sigma^{(k)} \in NC(|U_k|),   \\
  \sigma^{(1)} \geq \pi^{(1)}, \ldots, \sigma^{(k)} \geq \pi^{(k)} }  }
\ \bigl( \prod_{j=1}^k
\prod_{V \in \sigma^{(j)}} X_{ \pi^{(j)}_V } \bigr) \otimes
\bigl( \prod_{j=1}^k X_{\sigma^{(j)}} \bigr).
\end{equation}

The index set for the sum in (\ref{eqn:113b}) can be identified
as the bijective image of the set
$\{ \sigma \in NC(n) \mid \pi \leq \sigma \leq \tau \}$, via the map 
\begin{equation}   \label{eqn:113c}
\sigma \mapsto ( \sigma_{{ }_{U_1}}, \ldots \sigma_{{ }_{U_k}} ).
\end{equation}
We leave it as an exercise to the patient reader to check that, 
when the bijection (\ref{eqn:113c}) is used as a change of variable 
in the summation from (\ref{eqn:113b}), what comes out is indeed the 
right-hand side of the formula (\ref{eqn:113a}) claimed by the lemma.
\end{proof}

\vspace{6pt}

\begin{proposition}  \label{prop:112}
When endowed with the structure introduced in 
Notation \ref{def:111}, $\cT$ becomes a graded bialgebra.
\end{proposition}

\begin{proof}
The proof of consists of three 
verifications, pertaining to comultiplication, counit 
and grading, respectively.  

\vspace{6pt}

\noindent
{\em (i) Verification that $\Delta$ is coassociative}.

\noindent
Here we have to check that $(\Id\otimes \Delta)\circ \Delta 
= (\Delta\otimes \Id)\circ \Delta$. 
Since both sides of this equality are unital algebra
homomorphisms from $\cT$ to $\cT \otimes \cT \otimes \cT$, it 
suffices to check that they agree on every generator $X_{\pi}$
of $\cT$.  We thus pick an $n \geq 1$ and a $\pi \neq 1_n$ in 
$NC(n)$, and we will verify that both
$\Id \otimes \Delta \, \bigl(  \Delta ( X_{\pi} ) \bigr)$ and
$\Delta \otimes \Id \, \bigl(  \Delta ( X_{\pi} ) \bigr)$ are
equal to 
\begin{equation}   \label{eqn:112a}
\sum_{ \substack{\sigma, \tau \in NC(n) \\ \tau \geq \sigma \geq \pi} }
M_{\pi , \sigma} \otimes M_{\sigma , \tau} 
\otimes X_{\tau} \mbox{ (element of $\cT \otimes \cT \otimes \cT$).}
\end{equation}
Indeed, if in the double sum of (\ref{eqn:112a}) 
we first sum over $\tau$, then we get
\[
\sum_{ \substack{\sigma \in NC(n) \\ \sigma \geq \pi} }
M_{\pi , \sigma} \otimes \Bigl(
\sum_{ \substack{\tau \in NC(n) \\ \tau \geq \sigma} }
M_{\sigma , \tau} \otimes X_{\tau} \Bigr)
= \sum_{ \substack{\sigma \in NC(n) \\ \sigma \geq \pi} }
M_{\pi , \sigma} \otimes \Delta ( X_{\sigma} )
= \Id \otimes \Delta \, \bigl(  \Delta ( X_{\pi} ) \bigr).
\]
While if in (\ref{eqn:112a}) we first sum over $\sigma$,
then we get
\[
\sum_{ \substack{\tau \in NC(n), \\ \tau \geq \pi} }
\Bigl( \, \sum_{ \substack{\sigma \in NC(n), \\ \pi \leq \sigma \leq \tau} }
\ M_{\pi, \sigma} \otimes M_{\sigma, \tau}
\, \Bigr) \otimes X_{\tau}
= \sum_{ \substack{\tau \in NC(n), \\ \tau \geq \pi} }
\Delta ( M_{\pi, \tau} ) \otimes X_{\tau}
\ \mbox{ (by Lemma \ref{lemma:113})},
\]
and the latter quantity is precisely equal to
$\Delta \otimes \Id \, \bigl(  \Delta ( X_{\pi} ) \bigr)$.

\vspace{6pt}

\noindent
{\em (ii) Verification that $\epsilon$ satisfies 
the counit property},
i.e.~that $(\Id\otimes \epsilon)\circ \Delta = \Id 
= (\epsilon\otimes \Id)\circ \Delta.$

\noindent
Here again it suffices to focus on a generator $X_{\pi}$.
Upon chasing through the definitions, we see that what needs 
to be verified is this: given $n \geq 2$ and $\pi \neq 1_n$ in 
$NC(n)$, check that
\begin{equation}   \label{eqn:112b}
\sum_{\sigma\geq \pi} \epsilon (X_{\sigma}) \cdot
\prod_{W \in \sigma} X_{\pi_{{ }_W} }
= X_{\pi} =
\sum_{\sigma\geq \pi}   \prod_{W \in \sigma}
\epsilon \bigl( X_{ \pi_{{ }_W} } \bigr) \cdot  X_\sigma.
\end{equation}
And indeed: the first of the two equalities (\ref{eqn:112b}) 
holds because the only non-zero 
contribution to the sum occurs for $\sigma = 1_n$, when
$\prod_{W \in 1_n} X_{ \pi_{{ }_W} } = X_{\pi}$. 
The second equality (\ref{eqn:112b}) also holds, with the only 
non-zero contribution now coming from the term indexed 
by $\pi$:
\[
\Bigl( 0\neq \prod_{W \in \sigma} 
\epsilon( X_{ \pi_{{ }_W} } ) \Bigr) 
\ \Leftrightarrow \ \Bigl( \pi_{{ }_W} = 1_{|W|}, 
                  \ \forall \, W \in \sigma \Bigr) 
\ \Leftrightarrow \ ( \sigma = \pi ). 
\]

\vspace{6pt}

\noindent
{\em (iii) Verifications related to the grading.}

\noindent
We leave it to the reader to go over the list of conditions
that have to be verified here, and confirm that the only 
non-obvious item on the list is this: given $\pi \in NC(n)$ with 
$| \pi | = m$ (hence with $X_{\pi} \in \cT_{m-1}$) for some 
$2 \leq m \leq n$, one has that 
$\Delta (X_{\pi}) \in \bigoplus_{i=0}^{m-1} \cT_i \otimes \cT_{m-1-i}.$
In order to verify this fact one checks that, in the sum defining 
$\Delta (X_{\pi})$ in (\ref{eqn:111d}) one has:
\begin{equation}   \label{eqn:112c}
\Bigl( \prod_{W \in \sigma} X_{\pi_{{ }_W}} \Bigr) 
\otimes X_{\sigma} \in \bigoplus_{i=0}^{m-1} \cT_i \otimes \cT_{m-1-i}, 
\ \ \forall \, \sigma \in NC(n) \mbox{ such that } 
\sigma \geq \pi.
\end{equation}
Indeed, in the tensor indicated in (\ref{eqn:112c}), the product
of generators that appears to the 
left of the tensor sign has degree equal to
\[
\sum_{W \in \sigma} ( | \pi_{ { }_W } | - 1 ) 
= \sum_{W \in \sigma} | \pi_{ { }_W } | - \sum_{W \in \sigma} 1
= | \pi | - | \sigma | = m - | \sigma |.
\]
Since $X_{\sigma} \in \cT_{| \sigma | - 1}$, the tensor indicated
in (\ref{eqn:112c}) thus belongs to 
$\cT_{m- | \sigma |} \otimes \cT_{ | \sigma | - 1}$,
with $(m - | \sigma | ) + ( | \sigma | -1 ) = m-1$, as required.
\end{proof}

\vspace{6pt}

\begin{remark}   \label{rem:114}
Recall that the space $\cT_0 \subseteq \cT$ of homogeneous 
elements of degree $0$ consists precisely of the scalar 
multiples of the unit $\oneT$.  One refers to this property 
of $\cT$ by saying that it is {\em connected}.  As a 
consequence of being a graded connected bialgebra, $\cT$ is 
sure to be a {\em Hopf algebra}; that is, we are guaranteed 
(cf.~\cite[Section II.3]{Ma2003}) to have a unital algebra 
homomorphism $S: \cT \to \cT$, called {\em antipode of $\cT$},
which is in a certain sense the convolution inverse to the 
identity map $\mathrm{Id} : \cT \to \cT$.  A discussion of the 
antipode of $\cT$ is made in the next section of the paper.  
Right now we only record the fact that, due to these general 
Hopf algebra considerations, Proposition \ref{prop:112} can be 
restated in the following stronger form.
\end{remark}

\vspace{6pt}

\begin{theorem}   \label{thm:115}
When endowed with the structure introduced in Notation 
\ref{def:111}, $\cT$ becomes a graded connected Hopf algebra.
\hfill $\square$
\end{theorem}

\vspace{6pt}
 
\begin{remark}     \label{rem:116}
As mentioned at the beginning of the subsection, the Hopf 
algebra $\cT$ can be treated as an {\em incidence Hopf algebra} 
in the sense of Schmitt \cite{Schm1994}.  The present remark 
gives a brief outline of how this happens.

For every $n \geq 1$ and $\pi \leq \sigma$ in $NC(n)$ 
let us denote
\begin{equation}    \label{eqn:116a}
[ \pi , \sigma ] = \{ \rho \in NC(n) \mid 
\pi \leq \rho \leq \sigma \}
\ \mbox{ (a sub-poset of $( NC(n), \leq )$)}
\end{equation}
and let $\cP$ denote the collection of all the posets 
$[ \pi , \sigma]$ considered in (\ref{eqn:116a}).  On 
$\cP$ we have a natural operation of multiplication 
defined by
\begin{equation}   \label{eqn:116b}
[ \pi_1 , \sigma_1 ] \times [ \pi_2 , \sigma_2 ] =:
[ \pi_1 \diamond \pi_2, \sigma_1 \diamond \sigma_2 ],
\end{equation}
where ``$\diamond$'' denotes concatenation (as in 
Notation \ref{def:13}).  It turns out that on $\cP$ one 
can introduce an equivalence relation ``$\sim$'' which 
is compatible with the multiplication (\ref{eqn:116b})
and produces a commutative quotient monoid $\cP / \sim$ 
generated by
\[
\Bigl\{ \widehat{ [ \pi, 1_n] } \mid n \geq 1 \mbox{ and }
\pi \in NC(n) \setminus \{ 1_n \}  \Bigr\} ,
\]
where we use the notation ``$\widehat{ [ \pi , \sigma ] }$''
for the image of $[ \pi , \sigma ]$ under the quotient
map $\cP \rightarrow \cP / \sim$.  Moreover, the 
equivalence relation $\sim$ is set in such a way that one gets
factorizations
\begin{equation}    \label{eqn:116c}
\widehat{[ \pi , \sigma ]} = \prod_{W \in \sigma} 
\widehat{ [ \pi_{{ }_W}, 1_{{ }_{|W|}} ] }, \ \mbox{ for every }
[ \pi , \sigma ] \in \cP.
\end{equation}
When plugged into the general machinery described in 
\cite[Sections 2-4]{Schm1994}, the monoid algebra 
$\bC [ \cP / \sim ]$ becomes a Hopf algebra, which turns out to
be naturally isomorphic (as Hopf algebras!) to $\cT$ from 
Theorem \ref{thm:115}, via the unital algebra homomorphism 
defined by requiring that
\[
\cT \ni X_{\pi} \mapsto \widehat{ [ \pi , 1_n ] } \in
\bC [ \cP / \sim ], \ \ \forall \, n \geq 1 \mbox{ and } 
\pi \in NC(n) \setminus \{ 1_n \}.
\]
\end{remark}

$\ $

\subsection{The isomorphism 
\boldmath{$\cGtild \approx \bX ( \cT )$}.}

\begin{remark}   \label{rem:117}
{\em (Review of the group $( \bX ( \cT ) , * )$).}
A unital algebra homomorphism from $\cT$ to $\bC$ is also known under
the name of {\em character} of $\cT$, and it is customary to denote
\begin{equation}   \label{eqn:117a}
\bX ( \cT ) := \{ \chi : \cT \to \bC \mid \chi 
\mbox{ is a character} \} .
\end{equation}
The definition of $\bX ( \cT )$ only uses the algebra
structure on $\cT$.  But the coalgebra structure is 
important too, because it allows us to define a 
convolution operation for characters, via the formula
$\chi_1 * \chi_2 = ( \chi_1 \otimes \chi_2 ) \circ \Delta$.
That is: given $\chi_1 , \chi_2 \in \bX ( \cT )$ and 
$P \in \cT$, one considers some concrete writing 
$\Delta (P) = \sum_{i=1}^n P_i' \otimes P_i''$, and defines 
\begin{equation}   \label{eqn:117b}
\chi_1 * \chi_2  \, (P) := 
\sum_{i=1}^n \chi_1 (P_i') \chi_2 (P_i'').
\end{equation}
It is easily checked that the definition of $\chi_1 * \chi_2$
makes sense, and that in this way one gets an associative 
operation ``$*$'', called convolution, on $\bX ( \cT )$. 

It is clear that the counit $\eeps$ introduced in 
Notation \ref{def:111}.4 belongs to $\bX ( \cT )$. Then the 
counit verification from (ii) in the proof of Proposition 
\ref{prop:112} shows precisely that $\eeps$ is the (necessarily
unique) unit element of $( \bX ( \cT ), * )$.
Finally, for every $\chi \in \bX ( \cT )$ one can consider 
the new character $\chi \circ S \in \bX ( \cT )$, where 
$S : \cT \to \cT$ is the antipode map, and one can verify 
(see e.g.~\cite[Proposition II.4.1]{Ma2003})
that $\chi \circ S$ is inverse to $\chi$ with respect to
convolution.  Hence the overall conclusion is that 
$( \bX ( \cT ), * )$ is a group.
\end{remark}

\vspace{6pt}

\begin{theorem}   \label{thm:118}
$1^o$ For every $g \in \cGtild$ there exists a character
$\chi_{ { }_g } \in \bX ( \cT )$, uniquely determined, 
such that
\begin{equation}   \label{eqn:118a}
\chi_{ { }_g } ( X_{\pi} ) = g( \pi, 1_n ), \ \ 
\mbox{ for all } n \geq 1 \mbox{ and } \pi \in NC(n).
\end{equation}

$2^o$ The map 
$\cGtild \ni g \mapsto \chi_{ { }_g } \in \bX ( \cT )$
is a group isomorphism, i.e.~it is bijective and has
\begin{equation}  \label{eqn:118b}
\chi_{{ }_{g_1 * g_2}} = \chi_{{ }_{g_1}} * \chi_{{ }_{g_2}} ,
\ \ \forall \, g_1, g_2 \in \cGtild .
\end{equation}
\end{theorem}

\begin{proof}
The universality property noted in (\ref{eqn:111c}) 
implies that the characters of $\cT$ are in bijective
correspondence with families of complex numbers of the form  
$\bigl\{ z( \pi ) \mid \pi \in \sqcup_{n=1}^{\infty} 
( NC(n) \setminus \{ 1_n \} \bigr\}$, 
where the family of numbers corresponding to 
$\chi \in \bX ( \cT )$ is simply obtained by putting 
$z( \pi ) = \chi ( X_{\pi} )$ for all $n \geq 1$ and 
$\pi \in NC(n) \setminus \{ 1_n \}$.  When considered in 
conjunction with the Proposition \ref{prop:27} about 
functions in $\cGtild$, this immediately implies the 
statement $1^o$ of the theorem, and 
also the fact that the map
$\cGtild \ni g \mapsto \chi_{ { }_g } \in \bX ( \cT )$ 
is a bijection.

We are left to check that (\ref{eqn:118b}) holds.  In order 
to establish the equality of the characters 
$\chi_{g_1} * \chi_{g_2}$ and $\chi_{g_1 * g_2}$ it suffices
to verify that they agree on every generator $X_{\pi}$ of $\cT$.
We thus fix an $n \geq 1$ and a 
$\pi \in NC(n) \setminus \{ 1_n \}$, and we compute:
\begin{align*}
\chi_{g_1} * \chi_{g_2} \, (X_\pi) 
& = \sum_{\sigma \geq \pi \ \mathrm{in} \ NC(n)} 
\ \chi_{g_1} \bigl( \prod_{W \in \sigma} X_{\pi_{{ }_W}} \bigr)
  \cdot \chi_{g_2} ( X_{\sigma} )                           \\
& \mbox{ $\ $ (by (\ref{eqn:117b}), where we use the explicit 
         formula for $\Delta ( X_{\pi} )$)  }            \\
& = \sum_{\sigma \geq \pi \ \mathrm{in} \ NC(n)} 
\ \bigl( \prod_{W \in \sigma} g_1 ( \pi_{W}, 1_{{ }_{|W|}} ) \bigr)
\cdot g_2 ( \sigma, 1_n )                           \\
& \mbox{ $\ $ (by formulas defining $\chi_{g_1}, \chi_{g_2}$
          in terms of $g_1, g_2$)  }                          \\ 
& = \sum_{\sigma \geq \pi \ \mathrm{in} \ NC(n)} 
\  g_1 ( \pi , \sigma ) \cdot g_2 ( \sigma, 1_n )      
\ \mbox{ (by Eqn.(\ref{eqn:26a}) in Definition \ref{def:26})}   \\
& = g_1 * g_2 ( \pi, 1_n )
\ \mbox{ (by the definition of convolution in $\cGtild$)}     \\
& = \chi_{g_1 * g_2} ( X_{\pi} )
\mbox{ $\ $ (by the formula defining $\chi_{g_1 * g_2}$).}  
\end{align*}
\end{proof}        
        
$\ $

\subsection{A discussion of the primitive elements of \boldmath{$\cT$}.}

\begin{remark}   \label{rem:119x}
We now consider the set of {\em primitive} elements of $\cT$,
\[
\mbox{Prim} ( \cT) := 
\{ P \in \cT \mid \Delta (P) = P \otimes 1_\cT + 1_\cT \otimes P \}.
\]
The study of primitive elements is of great importance for 
co-commutative Hopf algebras, due to a fundamental theorem of 
Milnor-Moore which holds in that framework 
(see e.g.~\cite[Section 14.3]{GBVaFi2001}).
The Hopf algebra $\cT$ studied here is not co-commutative 
(corresponding to the fact that the group $\cGtild$ is not 
commutative), thus the role played by $\Prim ( \cT )$ in the study
of $\cT$ is less significant.  But for the sake of completeness, 
we give below a precise description of how $\Prim ( \cT )$ looks like.

We start by observing that: if $\pi \in NC(n)$ has $| \pi | = 2$, 
then $\{ \sigma \in NC(n) \mid \sigma \geq \pi \} = \{ \pi, 1_n \}$,
hence the sum which defined the comultiplication 
$\Delta (X_{\pi})$ in Equation (\ref{eqn:111d}) only has two 
terms.  It is moreover immediate that the terms indexed by $\pi$ 
and by $1_n$ in the said sum are $\oneT \otimes X_{\pi}$ and 
respectively $X_{\pi} \otimes \oneT$.  It thus follows that 
$X_{\pi} \in \Prim ( \cT )$ for every 
$\pi \in \sqcup_{n=1}^{\infty} NC(n)$ with $| \pi | = 2$.
Since the space $\cT_1 \subseteq \cT$ of homogeneous elements 
of degree 1 (as defined in Notation \ref{def:111}.5) is just
\[
\cT_1 = \mbox{span} \bigl\{ X_{\pi} \mid 
\pi \in \sqcup_{n=1}^{\infty} NC(n)
\mbox{ with } | \pi | = 2 \bigr\},
\]
we conclude that $\cT_1 \subseteq \Prim ( \cT )$.  The goal of 
the present subsection is to point out that the opposite inclusion
holds as well, and we therefore have:
\begin{equation}   \label{eqn:119b}
\mathrm{Prim} ( \cT ) = \cT_1.
\end{equation}
We will prove this equality in Proposition \ref{prop:1113} below.
Towards that goal, we first introduce some notation 
and prove a couple of lemmas.
\end{remark}

\vspace{6pt}

\begin{notation}    \label{def:1110}
$1^o$ The algebra $\cT$ has a linear basis $\cM$ consisting of 
monomials.  It consists of elements of the form
\begin{equation}   \label{eqn:1110b}
M := X_{\pi_1}^{q_1} \cdots X_{\pi_k}^{q_k},
\end{equation}
where $\{ \pi_1, \ldots , \pi_k \}$ is a finite subset of
$\sqcup_{n=1}^{\infty} \bigl( NC(n) \setminus \{ 1_n \} \bigr)$
and $(q_1, \ldots , q_k) \in \bN^k$ is a tuple of multiplicities.
We will use the notation $\# (M)$ for the total number of $X_{\pi}$'s
that are multiplied together to give an $M \in \cM$;
thus the monomial shown in (\ref{eqn:1110b}) has 
$\# (M) = q_1 + \cdots + q_k$.  We make the convention that if 
the set $\{ \pi_1, \ldots , \pi_k \}$ considered in
(\ref{eqn:1110b}) is empty, i.e.~if $k=0$, then the 
corresponding monomial is $M := \oneT$ with $\# (M) = 0$.

\vspace{6pt}

\noindent
$2^o$ For every $M \in \cM$ we let
$\xi_M : \cT \to \bC$ be the linear functional which acts
on $\cM$ by the prescription that
$\xi_M (M) = 1 \mbox{ and } \xi_M (N) = 0 
\mbox{ for any } N \in \cM \setminus \{ M \}$.

\vspace{6pt}

\noindent
$3^o$ Moving to $\cT \otimes \cT$: we have a linear basis
$\cM \otimes \cM := \{ M_1 \otimes M_2 \mid M_1, M_2 \in \cM \}$.
For every $M_1, M_2 \in \cM$ we let 
$\xi^{(2)}_{M_1,M_2} : \cT \otimes \cT \to \bC$ be the linear
functional which acts on $\cM \otimes \cM$ by:
\[
\xi^{(2)}_{M_1,M_2} (N_1 \otimes N_2) 
= \xi_{M_1} (N_1) \cdot \xi_{M_2} (N_2)
= \left\{   \begin{array}{ll}
1, &  \mbox{if $N_1 = M_1$ and $N_2 = M_2$,}  \\
0, &  \mbox{otherwise}
\end{array} \right\}, \mbox{for $N_1, N_2 \in \cM$.}
\]
\end{notation}

\vspace{6pt}

\begin{lemma}   \label{lemma:1111}
$1^o$ Let $M \in \cM$ be such that $\# (M) \geq 2$.
There exist $M_1, M_2 \in \cM$ with
$\# (M_i) \geq 1$ for $i=1,2$ and a $q \in \bN$ such that
\begin{equation}   \label{eqn:1111a}
\xi^{(2)}_{M_1, M_2} \circ \Delta = q \, \xi_M.
\end{equation}

\vspace{6pt}

\noindent
$2^o$ Let $n \in \bN$ and let $\pi$ be a partition in $NC(n)$
such that $| \pi | \geq 3$. There exist $M_1, M_2, M_3 \in \cM$
with $\# (M_1), \# (M_2) \geq 1$ and $\# (M_3) \geq 2$ such that
\begin{equation}   \label{eqn:1111b}
\xi^{(2)}_{M_1, M_2} \circ \Delta = \xi_{X_{\pi}} + \xi_{M_3}.
\end{equation}
\end{lemma}

\begin{proof} $1^o$ The monomial $M$ has an explicit writing
$M = X_{\pi_1}^{q_1} \cdots X_{\pi_k}^{q_k}$
with $\pi_1 \in NC(n_1), \ldots$, 
$\pi_k \in NC(n_k)$ and $q_1, \ldots , q_k \in \bN$, and where 
$\pi_1, \ldots , \pi_k$ are arranged such that 
$n_1 \geq n_2 \geq \cdots \geq n_k$.  For the role of the required 
$M_1, M_2$ we pick $M_2 := X_{\pi_k}$ and we put
\[
M_1 := X_{\pi_1}^{q_1} \cdots X_{\pi_{k-1}}^{q_{k-1}} 
\cdot X_{\pi_k}^{q_k - 1};
\]
that is, $M_1$ is chosen in such a way that $M_1 X_{\pi_k} = M$. 
Note that this is always possible 
due to the hypothesis that $\# (M) \geq 2$.  We leave it as an 
exercise to the reader to verify that
\begin{equation}   \label{eqn:1111c}
\xi^{(2)}_{M_1, M_2} \bigl( \, \Delta (M) \, \bigr) = q_k
\ \mbox{ and that }
\ \xi^{(2)}_{M_1, M_2} \bigl( \, \Delta (N) \, \bigr) = 0
\ \mbox{ for every $N \neq M$ in $\cM$.}
\end{equation}
As a hint towards the verification of the latter equality, we mention 
that it can be obtained by writing $N$ as a product 
$X_{\rho_1} \cdots X_{\rho_{\ell}}$ and then by applying the formula
(\ref{eqn:111d}) to every $\Delta ( X_{\rho_j} )$ in the factorization
$\Delta (N) = \Delta ( X_{\rho_1} ) \cdots \Delta ( X_{\rho_{\ell}} )$.
 
As a consequence of (\ref{eqn:1111c}), it is immediate that the required
formula (\ref{eqn:1111a}) is holding in connection to the $M_1, M_2$
indicated above and where we take $q := q_k$.
 
\vspace{6pt}
 
\noindent
$2^o$ The requirements of this part of the lemma can be 
fulfilled by putting 
\[
M_1 = X_{\sigma_1}, \ M_2 = X_{\sigma_2} 
\ \mbox{ and } M_3 = X_{\sigma_1} X_{\sigma_2}
\]
for suitably chosen non-crossing partitions $\sigma_1, \sigma_2$.
A concrete recipe for finding such $\sigma_1, \sigma_2$ is as 
follows: we let $\sigma_2$ be of the form 
$\sigma_2 = \{ U,W \} \in NC(n)$ where $U$ is a special block 
of $\pi$ (to be picked below) and $W = \{ 1, \ldots , n \} \setminus U$;
then we put $\sigma_1 := \pi_{{ }_W} \in NC( |W| )$, the
relabeled-restriction of $\pi$ to $W$.  

For an $M_1, M_2$ defined by a recipe as above, we leave it as
an exercise to the reader to examine what are the conditions on 
a monomial $N \in \cM$ which would allow 
$\xi^{(2)}_{M_1, M_2} \bigl( \, \Delta (N) \, \bigr)$ to be 
non-zero.  The result of the examination is that one has
\[
\xi^{(2)}_{M_1, M_2} \bigl( \, \Delta (X_{\pi}) \, \bigr)
= \xi^{(2)}_{M_1, M_2} \bigl( \, \Delta (M_3) \, \bigr) = 1
\]
and that $\xi^{(2)}_{M_1, M_2} \bigl( \, \Delta (N) \, \bigr) = 0$
for all other $N$, with the exception of a stray $N$ that can only
exist when $|U| = |W|$. The conclusion we draw is this: if in the 
construction of $\sigma_1, \sigma_2$ we can also arrange to have
$|U| \neq |W|$, then the desired Equation (\ref{eqn:1111b}) will hold.

It remains to make certain that we can always pick a block 
$U \in \pi$ such that, with $W := \{ 1, \ldots , n \} \setminus U$,
we have that $|W| \neq |U|$ and that $\sigma_2 := \{ U, W \}$
is non-crossing.  The non-crossing property of $\{ U,W \}$ is sure 
to hold when $U$ is an interval block of $\pi$; thus, if $\pi$ has 
an interval block $J$ with $|J| \neq n/2$, then we take $U =J$ and 
we are done.  But what if $\pi$ has a unique interval block $J$, 
with $|J| = n/2$?  In that case, a quick examination of the nesting
structure for the blocks of $\pi$ will show that $\pi$ also has a
unique outer block $H$, and that picking $U = H$ will give all the 
properties that $\sigma_2$ needs to have.
\end{proof}

\vspace{6pt}

\begin{lemma}   \label{lemma:1112}
Let $A$ be an element in $\Prim ( \cT )$.

\vspace{6pt}

\noindent
$1^o$ One has $\xi_M (A) = 0$ for every $M \in \cM$ with
$\# (M) \geq 2$.

\vspace{6pt}

\noindent
$2^o$ One has $\xi_{X_{\pi}} (A) = 0$ for every 
$\pi \in \sqcup_{n=1}^{\infty} NC(n)$ with $| \pi | \geq 3$.
\end{lemma}

\begin{proof} $1^o$ Pick an $M \in \cM$ with $\# (M) \geq 2$,
and let $M_1, M_2 \in \cM$ and $q \in \bN$ be 
such that (\ref{eqn:1111a}) holds.
Since $\Delta (A) - A \otimes \oneT - \oneT \otimes A = 0 \in \cT$,
we can write:
\begin{align*}
0 
& = \xi^{(2)}_{M_1,M_2} \bigl( \, \Delta (A) 
- A \otimes \oneT - \oneT \otimes A \, \bigr)  \\
& = \bigl( \xi^{(2)}_{M_1,M_2} \circ \Delta \bigr) (A)
    - \xi_{M_1} (A) \cdot \xi_{M_2} ( \oneT ) 
    - \xi_{M_1} ( \oneT ) \cdot \xi_{M_2} (A)  \\
& = q \, \xi_M (A) - 0 - 0,
\end{align*}
where at the latter equality sign we took into account 
(\ref{eqn:1111a}) and the fact that
$\xi_{M_1} ( \oneT ) = \xi_{M_2} ( \oneT ) = 0$.  We thus found that
$q \, \xi_M (A) = 0$, and the conclusion follows.

\vspace{6pt}

$2^o$ Pick a $\pi \in \sqcup_{n=1}^{\infty} NC(n)$ with $| \pi | \geq 3$,
and let $M_1, M_2, M_3 \in \cM$ be such that (\ref{eqn:1111b}) holds.
By using the same trick as in the proof of part $1^o$ we find that
\[
0 = \xi^{(2)}_{M_1,M_2} \bigl( \, \Delta (A) 
- A \otimes \oneT - \oneT \otimes A \, \bigr) 
= \bigl( \xi^{(2)}_{M_1,M_2} \circ \Delta \bigr) (A) \, \bigr) - 0 - 0.
\]
Since this time our handle on 
$\xi^{(2)}_{M_1,M_2} \circ \Delta$ comes from (\ref{eqn:1111b}), we now get:
\[
0 = \bigl( \, \xi_{X_{\pi}} + \xi_{M_3} \, \bigr) (A) 
= \xi_{X_{\pi}} (A) + \xi_{M_3} (A).
\]
But $\xi_{M_3} (A) = 0$, as proved in $1^o$ above.  We thus conclude that
$\xi_{X_{\pi}} (A) = 0$, as required.
\end{proof}

\vspace{6pt}

\begin{proposition}    \label{prop:1113}
$\mbox{Prim} ( \cT ) = \cT_1$.
\end{proposition}

\begin{proof}  
In view of Remark \ref{rem:119x}, we only have to verify the 
inclusion ``$\subseteq$''.  We thus fix for the whole proof 
an $A \in \Prim ( \cT )$, for which we want to prove that 
$A \in \cT_1$.  

We know that $A$ (same as any other element of $\cT$) can be 
decomposed as a sum of homogeneous elements.  That is: for
$m \in \bN$ large enough we can write 
\begin{equation}   \label{eqn:1113a}
A = A_0 + A_1 + \cdots + A_m 
\mbox{ with 
$A_0 \in \cT_0, \ldots , A_m \in \cT_m$,}
\end{equation}
where the homogeneous spaces $\cT_0, \ldots , \cT_m \subseteq \cT$ are
as defined in Equation (\ref{eqn:111f}) above.

By using Lemma \ref{lemma:1112} it is however easy to see that we must 
have $A_j = 0$ for every $2 \leq j \leq m$.  Indeed, it is immediate that 
if we had $A_j \neq 0$ then we would be able to find a monomial 
$M \in \cT_j$ such that $\xi_M (A_j) \neq 0$. The fact that $M \in \cT_j$ 
implies that $\cT_i \subseteq \mathrm{Ker} ( \xi_M )$ for all $i \neq j$
in $\bN \cup \{ 0 \}$, which implies in turn that
\[
\xi_M (A) = \xi_M (A_0) + \cdots + \xi_M (A_m) = \xi_M (A_j) \neq 0.
\]
But on the other hand, the fact that $M \in \cT_j$ with $j \geq 2$ also 
implies that either $\# (M) \geq 2$ or that $M$ is of the form $X_{\pi}$
for a partition $\pi$ with $| \pi | \geq 3$; thus Lemma \ref{lemma:1112}
asserts that $\xi_M (A) = 0$ -- contradiction!

Hence the decomposition (\ref{eqn:1113a}) has $A_j = 0$ for every 
$2 \leq j \leq m$, and since $A_0 \in \cT_0 = \bC \, \oneT,$
we thus get an equality of the form
\begin{equation}   \label{eqn:1113b}
A = \lambda \, \oneT + A_1,
\mbox{ with $\lambda \in \bC$ and $A_1 \in \cT_1$.}
\end{equation}
The comultiplication of the right-hand side of (\ref{eqn:1113b}) is
computed to be
\[
\Delta \bigl( \, \lambda \, \oneT + A_1 \, \bigr)
= \lambda \, \oneT \otimes \oneT + \Bigl( \, A_1 \otimes \oneT 
+ \oneT \otimes A_1 \Bigr)
\]
(where we took into account that
$A_1 \in \cT_1 \subseteq \Prim ( \cT )$).  By comparing this against
\[
\Delta (A) = A \otimes \oneT + \oneT \otimes A
= 2 \lambda \, \oneT \otimes \oneT + \Bigl( \, A_1 \otimes \oneT 
+ \oneT \otimes A_1 \Bigr),
\]
we see that $\lambda = 0$.  Hence $A = A_1 \in \cT_1$, as
we had to prove.
\end{proof}

$\ $

\section{A discussion of the antipode of $\cT$}

\noindent
The antipode of the Hopf algebra $\cT$ deserves special 
attention due to its potential use as a tool for inversion 
in formulas that relate moments to cumulants, or relate 
different brands of cumulants living in the $NC(n)$ framework.
The issue of performing such inversions is constantly present
in the literature on cumulants.  Indeed, it is typical that 
cumulants (of one brand or another) are introduced via some 
simple formulas which are deemed to express moments in terms 
of the desired cumulants; these simple formulas then need to be 
inverted, if one wants to see explicit formulas describing 
cumulants in terms of moments.  In such a situation, the tool 
that is typically used for inversion is the M\"obius function 
of some underlying poset which luckily turns out to be related 
to the cumulants in question.

The considerations on the Hopf algebra $\cT$ suggest an alternate
method which can provide a unified way of treating the inversions of 
various cumulant-to-moment formulas, and also for doing inversions 
in cumulant-to-cumulant formulas.  For a concrete illustration:
consider the framework of Example \ref{example:78}, and the question
of computing the inverse in $\cGtild$ for the semi-multiplicative 
function $\Fmcbc$ which encodes the transition from monotone cumulants
to Boolean cumulants.  The antipode strategy for this job amounts to 
looking at the character $\chi_{\mathrm{mc-bc}} \in \bX ( \cT )$ 
which corresponds to $\Fmcbc$, and then by performing the required 
inversion via the formula (cf.~\cite[Proposition II.4.1]{Ma2003})
\[
\chi_{\mathrm{mc-bc}}^{-1} = 
\chi_{\mathrm{mc-bc}} \circ S,
\ \mbox{ where $S : \cT \to \cT$ is the antipode map.}
\]

In this section we make a start towards the study of the antipode 
of $\cT$, with the hope that applications of the kind described 
above will be obtained in future work.

\vspace{6pt}

\begin{remark}  \label{rem:121}
{\em (Review of antipode basics.)}
Consider the graded bialgebra $\cT$, as discussed in Section 12.1.
The space of linear operators
$L ( \cT ) = \{ F : \cT \to \cT \mid F \mbox{ is linear} \}$
carries an associative operation of convolution defined as follows:
one puts
\begin{equation}   \label{eqn:121a}
F' * F'' = m \circ ( F' \otimes F'' ) \circ \Delta, 
\ \mbox{ for $F', F'' \in L( \cT )$,}
\end{equation}
where the map ``$m$'' indicated on the right-hand side is the 
multiplication, $m : \cT \otimes \cT \to \cT$ acting by 
$m( P \otimes Q ) = PQ$ for $P, Q \in \cT$.  What (\ref{eqn:121a})
says is that in order to evaluate $F' * F''$ on an element 
$P \in \cT$ we should pick a writing
$\Delta (P) = \sum_{i=1}^n P_i ' \otimes P_i ''$ for the 
comultiplication of $P$, which yields that
\begin{equation}   \label{eqn:121b}
F' * F'' \, (P) = 
\sum_{i=1}^n F' (P_i ') \, F'' (P_i '') \in \cT.
\end{equation}
It is easy to see that the convolution operation ``$*$'' on 
$L( \cT )$ is well-defined and is associative.  Moreover, if we 
consider the map $\widehat{\epsilon} \in L( \cT )$ defined by
\begin{equation}   \label{eqn:121c}
\widehat{\epsilon} \, (P) = \epsilon (P) \, \oneT, 
\ \ \forall \, P \in \cT,
\ \mbox{ where $\epsilon : \cT \to \bC$ is the counit of $\cT$,}
\end{equation}
then it is easily verified that $\widehat{\epsilon}$ is the 
(necessarily unique) unit for the semigroup $( L( \cT ) , * )$.

Now comes the point anticipated in Remark \ref{rem:114} of the 
preceding section, that the identity map $\Id \in L( \cT )$ 
is sure to be an invertible element of $( L( \cT ), * )$.  This
follows from general considerations on graded connected bialgebras 
-- see for instance \cite[Corollary II.3.2]{Ma2003}.  The inverse 
of $\Id$ in $( L( \cT ), * )$ is called the {\em antipode} of 
$\cT$, denoted by $S$, and the existence of $S$ makes $\cT$ be 
a Hopf algebra, as anticipated in Theorem \ref{thm:115}.

General bialgebra considerations, which also take into account
that $\cT$ is commutative, yield  the fact that $S : \cT \to \cT$ 
is a unital algebra homomorphism (cf.~\cite[Proposition I.7.1]{Ma2003}). 
This implies in particular that $S$ is completely determined
by how it acts on the generators $X_{\pi}$ of $\cT$.  In 
Proposition \ref{prop:123} below we state some formulas which
allow recursive calculations of values $S( X_{\pi} )$, and go under
the name of {\em Bogoliubov formulas}.
\end{remark}

\begin{notation}   \label{def:122}
For $\pi, \sigma \in NC(n)$, we will write 
``$\pi < \sigma$'' to mean that $\pi \leq \sigma$ (in the 
sense or reverse refinement) and that $\pi \neq \sigma$.
\end{notation}

\begin{proposition}   \label{prop:123}
{\em (Bogoliubov formulas.)}  For
$n \geq 1$ and $\pi \in NC(n) \setminus \{ 1_n \}$ one has:
\begin{equation}    \label{eqn:123a}
S( X_{\pi} ) = - X_{\pi} - 
\sum_{ \substack{ \sigma \in NC(n), \\ \pi < \sigma < 1_n} }
M_{\pi , \sigma} \, S( X_{\sigma}),
\end{equation}
and also that

\begin{equation}    \label{eqn:123b}
S( X_{\pi} ) = - X_{\pi} - 
\sum_{ \substack{ \sigma \in NC(n), \\ \pi < \sigma < 1_n} }
S( M_{\pi , \sigma} ) \, X_{\sigma},
\end{equation}
where the monomials $M_{\pi , \sigma}$ are as introduced 
in Notation \ref{def:112}.
\end{proposition}

\begin{proof} The relation $\Id * S = \widehat{\epsilon}$
implies in particular that
$\Id * S  \, (X_{\pi}) = \epsilon (X_{\pi}) \oneT = 0$.
But on the other hand, the explicit description (\ref{eqn:121b})
used for $\Id * S$ says that:
\[
\Id * S \, (X_{\pi}) = \sum_{\sigma \geq \pi}
M_{\pi, \sigma} \, S( X_{\sigma} ) 
= M_{\pi, \pi} \, S( X_{\pi} ) 
 + M_{\pi , 1_n} \, S( X_{1_n} )
 + \sum_{\pi < \sigma < 1_n}  M_{\pi, \sigma} \, S( X_{\sigma} ).
\]
Upon recalling (cf.~Remark \ref{def:112}) that
$M_{\pi, \pi} = \oneT$
and $M_{\pi, 1_n} = X_{\pi}$, we thus find that

\begin{equation}   \label{eqn:123c}
0 = S( X_{\pi} ) +  X_{\pi} 
+ \sum_{\pi < \sigma < 1_n}  M_{\pi, \sigma} \, S( X_{\sigma} ),
\end{equation}
where separating the term $S( X_{\pi} )$ on the right-hand side
leads to the formula (\ref{eqn:123a}).

The derivation of (\ref{eqn:123b}) is analogous, where we now
start from the fact that $S * \Id = \widehat{\epsilon}$.
\end{proof}

\begin{remark}  \label{rem:124}
$1^o$ In the statement of Proposition \ref{prop:123} we excluded 
the case when $\pi = 1_n$.  In that case we 
have $X_{\pi} = \oneT$ and taking the antipode just gives 
$S( X_{1_n} ) = S( \oneT ) = \oneT$.

Note also that, in the case when $| \pi | = 2$, the sum over 
$\{ \sigma \in NC(n) \mid \pi < \sigma < 1_n \}$ is an 
empty sum.  In that case, either (\ref{eqn:123a}) or 
(\ref{eqn:123b}) gives that $S( X_{\pi} ) = - X_{\pi}$; this is in 
agreement with the fact, observed in Section 12.3, that 
$X_{\pi}$ is a primitive element of $\cT$. 

\vspace{6pt}

$2^o$ Both (\ref{eqn:123a}) and (\ref{eqn:123b}) can be used
for a recursive computation of values $S( X_{\pi} )$,
but the setting of the recursion is different in the two situations.
Formula (\ref{eqn:123a}) works when we fix an $n \in \bN$, taken 
in isolation, and compute $S( X_{\pi} )$ for $\pi \in NC(n)$, by 
induction on $| \pi |$.
Formula (\ref{eqn:123b}) works when we already know
how $S$ works on some partitions from $NC(m)$'s with $m<n$ -- 
for instance, this works neatly when we fix an $\ell \geq 1$ and we 
are interested in $S( X_{\pi} )$ for all 
$\pi \in \sqcup_{n=1}^{\infty} NC(n)$ such that every block 
$V$ of $\pi$ has $|V| \leq \ell$.

The next example (continued in the subsequent Examples 
\ref{example:128} and \ref{example:1213})
illustrates how these two recursive methods work towards 
computing $S( X_{0_n} )$ for some small values of $n$.
\end{remark}

\begin{example}  \label{example:125}
Recall that $0_n \in NC(n)$ is the partition with $n$ 
singleton blocks.  From Remark \ref{rem:124}.1 we infer 
that $S( X_{0_1} ) = \oneT$ (because $X_{0_1} = \oneT$)
and that $S( X_{0_2} ) = - X_{0_2}$ (because $X_{0_2}$
is primitive).

For the computation of $S( X_{0_3} )$, let us record that
the set of intermediate partitions
$\{ \sigma \in NC(3) \mid 0_3 < \sigma < 1_3 \}$ consists 
of $\sigma_1, \sigma_2, \sigma_3$, where:
\[
\sigma_1 = \bigl\{ \, \{ 1 \}, \, \{ 2,3 \} \, \bigr\},
\ \sigma_2 = \bigl\{ \, \{ 1,3 \}, \, \{ 2 \} \, \bigr\},
\ \sigma_3 = \bigl\{ \, \{ 1,2 \}, \, \{ 3 \} \, \bigr\}.
\]
For every $1 \leq i \leq 3$ we have that 
$S( X_{\sigma_i} ) = - X_{\sigma_i}$, because $| \sigma_i | = 2$,
and (directly from the definition of the monomials $M_{\pi , \sigma}$)
we see that $M_{0_3, \sigma_i} = X_{0_2}$.  We leave to the reader  
the immediate verification that, based on this information, either 
of the two Bogoliubov formulas shown in Proposition \ref{prop:123}
leads to:
\begin{equation}   \label{eqn:125a}
S( X_{0_3} ) = - X_{0_3} 
+ X_{0_2} \bigl( X_{\sigma_1} + X_{\sigma_2} + X_{\sigma_3} \bigr).
\end{equation}

The sum of 4 terms that appeared on the right-hand side of Equation 
(\ref{eqn:125a}) can be viewed as a sum indexed by all possible chains
that go from $0_3$ to $1_3$ in the poset $NC(3)$.  This is clarified
in Proposition \ref{prop:127} below, which is a special case of a 
result of Schmitt \cite{Schm1987} holding in the 
general framework of an incidence Hopf algebra.  For the proof of 
Proposition \ref{prop:127} (which is, essentially, an induction on 
$| \pi |$ based on the recursion formula (\ref{eqn:123a})) we refer 
to \cite[Theorem 6.1]{Schm1987} or \cite[Theorem 4.1]{Schm1994}.
\end{example}

\vspace{6pt}

\begin{definition}   \label{def:126}
Let $n$ be a positive integer and let $\pi, \sigma \in NC(n)$ 
be such that $\pi < \sigma$.  A {\em chain} from $\pi$ to $\sigma$ 
is a tuple
\begin{equation}   \label{eqn:126a}
c = ( \pi_0, \pi_1, \ldots , \pi_k ), \mbox{ where } 
\pi = \pi_0 < \pi_1 < \cdots < \pi_k = \sigma.
\end{equation}
The number $k$ appearing in (\ref{eqn:126a}) is called 
the {\em length} of $c$. 

For a chain $c$ as in (\ref{eqn:126a}) it will be convenient to 
use the shorthand notation
\begin{equation}   \label{eqn:126b}
M_c := M_{\pi_0, \pi_1} M_{\pi_1, \pi_2} 
\cdots M_{\pi_{k-1}, \pi_k} \in \cT.
\end{equation}
\end{definition}

\vspace{6pt}

\begin{proposition}  \label{prop:127}
For $n \geq 1$ and $\pi \in NC(n) \setminus \{ 1_n \}$ one has:
\begin{equation}    \label{eqn:127a}
\mbox{$\ $} \hspace{2cm}
S ( X_{\pi} ) = \sum_{ \substack{
         c = ( \pi_0, \pi_1, \ldots , \pi_k ),  \\
         \mathrm{chain \ from} \ \pi \ \mathrm{to} \ 1_n} }
(-1)^k M_c.
% \ (-1)^k M_{\pi_0, \pi_1} M_{\pi_1, \pi_2} \cdots M_{\pi_{k-1}, \pi_k}.
\hspace{5.4cm}  \square
\end{equation}
\end{proposition}

\vspace{6pt}

\begin{example}     \label{example:128}
In continuation of Example \ref{example:125}, let us now compute
what is $S( X_{0_4} )$.  Proposition \ref{prop:127} gives an 
explicit formula for this antipode, as a sum indexed by chains in 
$NC(4)$ which go from $0_4$ to $1_4$. There are 29 such chains:

-- 1 chain of length $1$, the chain $c = (0_4, 1_4)$;

-- 12 chains of length 2, of the form 
$c = ( 0_4, \sigma , 1_4)$ with
$\sigma \in NC(4) \setminus \{ 0_4, 1_4 \}$;

-- 16 chains of length $3$, of the form 
$c = ( 0_4, \sigma, \sigma ' , 1_4)$ where 
$\sigma , \sigma ' \in NC(4)$ are 

\hspace{0.3cm} such that 
$| \sigma | = 3$, $| \sigma ' | = 2$ and $\sigma < \sigma '$.

\noindent
Hence Proposition \ref{prop:127} gives $S( X_{0_4})$ written 
as a sum of 29 terms. 

Now, recall that Proposition \ref{prop:127} is based on the 
Bogoliubov formula (\ref{eqn:123a}), which ``has $S$-factors on 
the right''.  The computation of $S( X_{0_4} )$ can also be done 
by using the formula (\ref{eqn:123b}), which has $S$-factors on 
the left:
\begin{equation}   \label{eqn:128a}
S( X_{0_4} ) = - X_{0_4} - 
        \sum_{ \substack{\sigma \in NC(4), \\ 0_4 < \sigma < 1_4} }
S( M_{0_4, \sigma} ) \, X_{\sigma}.
\end{equation}
It is immediate that, for every $\sigma \in NC(4)$ with 
$0_4 < \sigma < 1_4$, the monomial $M_{0_4, \sigma}$ is a product of
factors $X_{0_2}$ and $X_{0_3}$; so, consequently,
$S( M_{0_4, \sigma} )$ can be computed explicitly by using the formulas
for $S( X_{0_2} ), S( X_{0_3} )$ found in Example \ref{example:125}.
In this way, the right-hand side of (\ref{eqn:128a}) is turned into 
an explicit formula for $S( X_{0_4} )$.  The reader who has the patience
to really write the latter formula will discover the interesting detail
that it only has 25 terms (instead of 29, as we got from applying 
Proposition \ref{prop:127}).
This happens because the formula (\ref{eqn:127a})
isn't generally {\em cancellation-free}.  In the case at hand, of 
$\pi = 0_4$, we can pin down precisely where it is that the 
cancellations in (\ref{eqn:127a}) take place.  There are two terms 
disappearing because the chains of length 3
\begin{equation}   \label{eqn:128b}
\left\{   \begin{array}{l}
c' = \bigl( 0_4 , \, \{\{1,2\},\{3\},\{ 4\}\} , 
\,  \{\{1,2\},\{3,4\}\} , \, 1_4 \bigr)
\mbox{ and }                                      \\
c'' = \bigl( 0_4 , \, \{\{1\},\{2\},\{3,4\}\} , 
\, \{\{1,2\},\{3,4\}\} , \, 1_4 \bigr)
\end{array}   \right.
\end{equation}
have the same contribution (but with opposite sign) 
as the shorter chain $( 0_4, \{\{1,2\},\{3,4\}\}$, 
$1_4 )$.  Then there are two other 
terms that disappear, in a similar way, in connection 
to the chain $( 0_4, \{\{1,4\},\{2,3\}\} , 1_4 )$.

The method based on (\ref{eqn:123b}) can be shown to give 
a cancellation-free formula for $S( X_{0_n} )$, for every 
$n \geq 1$. The number $t_n$ of terms which appears in the
cancellation-free formula satisfies a recursion presented in 
Example \ref{example:1213} below.  According to the calculations 
we showed so far, the sequence $(t_n)_{n=1}^{\infty}$ starts with 
$1,1,4,25$; but this promising start turns out to not continue 
towards some known integer sequence.
\end{example}

\vspace{6pt}

\begin{remark}   \label{rem:129}
We will next show how one can re-structure the summation 
over chains from (\ref{eqn:127a}) in order to obtain a 
cancellation-free summation formula.  This will be done by 
pruning the index set used in (\ref{eqn:127a}) to a smaller 
collection of chains in $NC(n)$, which we call ``efficient 
chains'' -- cf.~Definition \ref{def:129}, Theorem \ref{thm:1210}.

We mention that our identifying of the notion of efficient chain 
retrieves a special case of a notion identified in the thesis 
\cite{Ein2010}, in the general framework of incidence Hopf 
algebras, where the terms of the cancellation-free summations 
arrive to be described by objects called ``forests of lattices''
(cf.~\cite[Chapter 5]{Ein2010}).
While it would be possible to review the fairly substantial
background and terminology developed in \cite{Ein2010} and then
invoke the result from there, we find it easier to write down a
direct inductive argument which covers the special case needed 
in Theorem \ref{thm:1210}.

We would also like to signal that another path towards obtaining
a cancellation-free summation formula for the antipode of $\cT$
is offered by the work in \cite{MePa2018}.  This would exploit the
fact that $\cT$ is an example of so-called
\textit{left-handed polynomial Hopf algebra}, a term which refers to
the fact that the formula (\ref{eqn:111d}) defining comultiplication 
merely has an ``$X_{\sigma}$'' (rather than a product of 
$X_{\sigma}$'s) on the right side of the tensor product.
\end{remark} 

\vspace{6pt}

\begin{definition}   \label{def:129}
Let $n$ be a positive integer and let $\pi, \sigma \in NC(n)$ 
be such that $\pi < \sigma$.

\vspace{6pt}

\noindent
$1^o$ To every chain $c = ( \pi_0, \pi_1, \ldots , \pi_k )$ 
from $\pi$ to $\sigma$ we associate two collections of 
subsets of $\{ 1, \ldots , n \}$, as follows:
\[
\Blocks (c) := \{ V \subseteq \{ 1, \ldots , n \} 
\mid \exists \, 0 \leq j \leq k 
\mbox{ such that $V$ is a block of } \pi_j \},
\mbox{ and}
\]

\[
\Blocksplus (c) := \{ V \in \Blocks (c) \mid
V \mbox{ is not a block of } \pi_0 \}.
\]

\vspace{6pt}

\noindent
$2^o$ A chain $c = ( \pi_0, \pi_1, \ldots , \pi_k )$ from
$\pi$ to $\sigma$ will be said to be {\em efficient} when 
it satisfies:
\[
\left\{  \begin{array}{l}
\mbox{For every set $V \in \Blocksplus (c)$ there exists }  \\
\mbox{a {\em unique} $j \in \{ 1, \ldots , k \}$ such that
      $V$ is a block of $\pi_j$.}
\end{array}  \right.
\]
      
\vspace{6pt}

\noindent
$3^o$  We denote by $\EC (\pi,\sigma)$ the set of all 
efficient chains from $\pi$ to $\sigma$.
\end{definition}

\vspace{6pt}

\begin{remark}  \label{rem:1211}
$1^o$  In order to explain the term ``efficient'' used in the
preceding definition, let $\pi < \sigma$ be as above and let
$c = ( \pi_0, \pi_1, \ldots , \pi_k)$ be a chain from $\pi$ 
to $\sigma$.  Pick an $m \in \{ 1, \ldots , n \}$ and 
for every $0 \leq j \leq k$ let us denote by $V^{(j)}$ the block 
of $\pi_j$ which contains the number $m$.  Then we have 
\begin{equation}   \label{eqn:1211a}
V^{(0)} \subseteq V^{(1)} \subseteq \cdots 
\subseteq V^{(k)} 
\ \mbox{ $\ $ (subsets of $\{ 1, \ldots , n \}$),}
\end{equation}
where some of the inclusions in (\ref{eqn:1211a}) may 
actually be equalities.  The property of $c$ described in
Definition \ref{def:129}.2 amounts to the fact that 
once we run into an inclusion $V^{(i-1)} \subseteq V^{(i)}$
which is strict, all the subsequent inclusions
$V^{(j-1)} \subseteq V^{(j)}$ with $j \geq i$ have to be 
strict as well -- in a certain sense, one 
``moves efficiently'' towards the last set $V^{(k)}$ 
indicated in that list.

\vspace{6pt}

\noindent
$2^o$ Given $\pi < \sigma$ in $NC(n)$ and a chain 
$c = ( \pi_0, \pi_1, \ldots , \pi_k ) \in \EC ( \pi, \sigma )$
we will be interested in the quantity
\begin{equation}   \label{eqn:1211b}
(-1)^{ | \Blocksplus (c) | } M_c 
= (-1)^{ | \Blocksplus (c) | }  
M_{\pi_0, \pi_1} M_{\pi_1, \pi_2} \cdots M_{\pi_{k-1}, \pi_k},
\end{equation}
which will be featured in Theorem \ref{thm:1210} below.
For illustration, let us look at how this quantity plays out 
in connection to the cancellations we spotted in 
Example \ref{example:128}.  The chains $c', c''$ of length 3 
shown in (\ref{eqn:128b}) are not efficient: for instance for
the first of them we find that the set 
$V = \{ 1,2 \} \in \Blocksplus (c')$ belongs
to both partitions $\pi_1$ and $\pi_2$ of $c'$, where
$\pi_1 = \{ \{1,2\},\{3\},\{ 4\} \}$ and 
$\pi_2 = \{ \{1,2\}, \{3,4\} \}$.  On the other hand:
\[
c := ( 0_4, \, \{\{1,2\},\{3,4\}\}, \, 1_4 )
\mbox{ is efficient, with }
\Blocksplus (c) = \bigl\{ 
\, \{ 1,2 \}, \, \{ 3,4 \}, \, \{ 1,2,3,4 \} \, \bigr\}.
\]

The issue observed in Example \ref{example:128} was
this: when plugged into the summation on the 
right-hand side of (\ref{eqn:127a}), both $c'$ and $c''$ 
have contributions of $- X_{0_2} X_{\sigma}$, for 
$\sigma = \{ \{ 1,2 \}, \, \{ 3,4 \} \}$, while $c$ has a 
contribution of $+ X_{0_2} X_{\sigma}$. (Cancellation!)
In the formula featured in Theorem \ref{thm:1210}, 
the chains $c'$ and $c''$ will no longer appear, while 
$c$ will appear with a contribution 
of $- X_{0_2} X_{\sigma}$; we note the different sign in 
the contribution of $c$ (coming from the fact that
$| \Blocksplus (c) |$ is of different parity than the 
length of $c$), and accounting for the cancellations
``$(-1) + (-1) + 1 = -1$'' that we had encountered before.

\vspace{6pt}

$3^o$ Let $\pi < \sigma$ be in $NC(n)$ and consider a chain 
$c = ( \pi_0, \pi_1, \ldots , \pi_k ) \in \EC ( \pi, \sigma )$.
Upon tallying what indeterminates ``$X_{\rho}$'' are taken 
into the monomials 
$M_{\pi_0, \pi_1}, M_{\pi_1, \pi_2}, \ldots, M_{\pi_{k-1}, \pi_k}$
multiplied in (\ref{eqn:1211b}), one finds the following 
interpretation for the cardinality of the set $\Blocksplus (c)$: 
it counts the total number of factors $X_{\rho}$ when the product
$M_{\pi_0, \pi_1} M_{\pi_1, \pi_2} \cdots$

\noindent
$\cdots M_{\pi_{k-1}, \pi_k}$
is simply treated as a product of $X_{\rho}$'s, and we eliminate the 
units ``$X_{1_m}$'' which may have appeared in the description
of the monomials $M_{\pi_{j-1}, \pi_j}$.

The observation made in the preceding paragraph ensures that the
summation formula stated in Theorem \ref{thm:1210} is
cancellation-free!  Indeed, if two chains appearing on the 
right-hand side of that summation formula turn out to have the 
same ``$M_c$'' contribution, then they will also have the same 
sign in the ``$(-1)^{ | \Blocksplus (c) | }$'' part
of the formula; hence the terms indexed by the two chains
in question will not cancel, but will rather add up.

For a concrete example, suppose we make $n=7$ and we consider the 
chains
\[
c' := \Bigl( 0_7, \, \{\{1,2\}, \, \{3\}, \, \{4,5,6\}, \, \{7\}\} , 
                 \, \{\{1,2,3\},\{4,5,6,7\}\} , \, 1_7 \Bigr)
\mbox{ and }
\]
\[
\begin{array}{lcr}
c'' & := &  \Bigl( 0_7, \{ \{1\},\{2\},\{3\},\{4,5\},\{6\},\{7\} \},
            \{ \{1\},\{2\},\{3\},\{4,5,6\},\{7\} \}, \ \mbox{$\ $}    \\ 
    &    &   \{ \{1,2,3\},\{4,5,6,7\} \}, \, 1_7 \Bigr),
\end{array}
\]
which are efficient chains going from $0_7$ to $1_7$.  In the 
summation formula (\ref{eqn:1210a}) of Theorem \ref{thm:1210}, 
$c'$ and $c''$ have identical contributions, of 
$(-1)^5 X_{\rho_1} \cdots X_{\rho_5}$ where 
\[
\rho_1 = \bigl\{ \{1,2,3\},\{4,5,6,7\} \bigr\}, 
\, \rho_2 = \bigl\{ \{1,2,3 \},\{4\} \bigr\}, 
\, \rho_3 = \bigl\{ \{1,2\},\{3\} \bigr\}, 
\, \rho_4 = 0_3, \, \rho_5 = 0_2.
\]
The point to note is that the contributions of $c'$ and $c''$ 
to the right-hand side of (\ref{eqn:1210a}) do not cancel each other, 
but rather get to be added together, as mentioned above.

The proof of Theorem \ref{thm:1210} is based on the following lemma.
\end{remark}

\begin{lemma}
\label{lem.bijection.antipode}
Let $\pi, \sigma$ be in $NC(n)$ for some 
$n \geq 1$, such that $\pi < \sigma < 1_n$, and where we 
write explicitly $\sigma = \{ V_1, \ldots , V_r \}$.
Consider the set of 
\footnote{Note that in the chain $c$ indicated in (\ref{eqn:12x})
the partition $\pi$ appears as $\pi^{(k)}$.  We chose this way 
of denoting $c$ because it simplifies the write-up of the proof 
of the lemma.}
efficient chains
\begin{equation}   \label{eqn:12x}
\widetilde{\EC} := 
\{ c \in \EC ( \pi, 1_n ) \mid
c=(\pi^{(k)},\pi^{(k-1)},\ldots,\pi^{(1)},1_n)
\mbox{ with $k \geq 2$ and } \pi^{(1)} = \sigma \} .
\end{equation}
One has a bijection
\begin{equation}   \label{eqn:12y}
\widetilde{\EC} \ni c \mapsto  (c_1, \ldots , c_r) 
\in \EC ( \pi_{ { }_{V_1} }, 1_{|V_1|} )
\times \cdots \times
\EC ( \pi_{ { }_{V_r} }, 1_{|V_r|} )
\end{equation}
where, for 
$c = (\pi^{(k)},\pi^{(k-1)},\ldots,\pi^{(1)},1_n)
\in \widetilde{\EC}$ and $1 \leq s \leq r$, we put
$c_s:= (\pi^{(k)}_{ { }_{V_s} },\ldots,\pi^{(1)}_{ { }_{V_s} })$. 

\noindent
(If it happens that we have
$\pi_{ { }_{V_2} }=\pi^{(k)}_{ { }_{V_s} }=
\pi^{(k-1)}_{ { }_{V_s}}=\dots=\pi^{(j)}_{ { }_{V_s}}$ for some 
$1\leq j\leq k-1$, then
$(\pi^{(k)}_{ { }_V },\ldots,\pi^{(2)}_{ { }_V },\pi^{(1)}_{ { }_V })$ 
is not properly a chain, so we rather take 
$c_s=(\pi_{ { }_V },\pi^{(j-1)}_{ { }_V },\ldots,\pi^{(1)}_{ { }_V })$.)

\vspace{6pt}

\noindent
Furthermore, for $c \mapsto (c_1, \ldots , c_r)$ as in 
(\ref{eqn:12y}), one has
\begin{equation}   \label{eq.equality.M}
(-1)^{ | \Blocksplus (c) | } M_c=  
-X_{\pi^{(1)}} \prod_{s=1}^r 
\left( (-1)^{ | \Blocksplus (c_s) | } M_{c_s} \right).
\end{equation}
\end{lemma}

\vspace{6pt}

\begin{proof}
We first prove that $c_1, \ldots , c_r$ from (\ref{eqn:12y}) 
are efficient chains.  Pick an $s \in \{ 1, \ldots , r \}$ and,
for the sake of contradiction, assume that $c_s$ is not efficient. 
This implies that there exist a block $W\in \Blocksplus(c_s)$
and indices $1\leq i<j\leq k$ such that
$W\in \pi^{(i)}_{ { }_V }$ and $W\in \pi^{(j)}_{ { }_V }$. Since 
$\pi^{(j)}< \pi^{(i)} \leq \pi^{(1)}$, this implies that 
$W\in \Blocksplus(c)$ with $W\in \pi^{(i)}$ and $W\in \pi^{(j)}$,
contradicting the fact that $c$ is efficient. Therefore,
$c'\in \EC(\pi_{ { }_{V} },1_{|V|})$ for all $c\in\EC(\pi,1_n)$ and
$V\in \pi^{(1)}$.

In order to prove that the map indicated in (\ref{eqn:12y}) is bijective,
we will describe how its inverse works.
For this, suppose we have an $r$-tuple of chains, $c_s=
(\pi_s^{(j_s)}, \dots, \pi_s^{(1)})\in \EC(\pi_{ { }_{V_s} },1_{|V_s|})$. 
To reconstruct the chain $c \in \widetilde{\EC}$ which corresponds to
$(c_1, \dots, c_r)$,  we first consider the size of the largest chain 
$j:=\max_{1\leq s\leq r} j_s$. Then, we enlarge the other chains so 
that all have the largest size, by denoting $\pi_s^{(i)}:=\pi_{ { }_{V_s}}$ 
for every $s=1,\dots,r$ and $j_s<i\leq j$. Then, for every
$i=1,\dots,j$, we construct the partition $\pi^{(i)}\in NC(n)$ uniquely 
determined by the fact that $\pi^{(i)}\leq \sigma$ and
$\pi^{(i)}_{ { }_{V_s}}=\pi_s^{(i)}$ for $s=1,\dots,r$. Finally, we define
$c := (\pi^{(k)},\pi^{(k-1)},\ldots,\pi^{(1)},1_n)$.
It is not hard to show (left as exercise to the reader) that this
$c$ is in $\widetilde{\EC}$, is mapped by (\ref{eqn:12y}) into the 
$(c_1, \ldots , c_r)$ we started from, and is uniquely determined by 
this property.

Finally, Equation \eqref{eq.equality.M} follows easily from the bijection
(\ref{eqn:12y}).
Indeed, for the equality of signs we break $c$ by taking apart 
$(\pi^{(1)},1_n)$, and then regroup the remaining chain in terms of the 
blocks of $\pi ^{(1)}$. Since $\Blocksplus(\pi^{(1)},1_n)=1$ we get that

\[
\Blocksplus(c) 
= \Blocksplus(\pi^{(k)},\pi^{(k-1)},\ldots,\pi^{(1)})+1 
= 1+ \sum_{s=1}^r \Blocksplus(c_r).
\]

For the term $M_c$, the idea is the same, although the computation is 
a bit more involved:
\begin{align*}
M_c 
& = M_{\pi^{(1)}, 1_n} \prod_{i=1}^{k-1} M_{\pi^{(i+1)}, \pi^{(i)}}
  = X_{\pi^{(1)}} \prod_{i=1}^{k-1} 
    \prod_{s=1}^r M_{\pi^{(i+1)}_{ { }_{V_s}}, \pi^{(i)}_{ { }_{V_s}}}
  = X_{\pi^{(1)}} \prod_{s=1}^r \prod_{i=1}^{k-1} 
      M_{\pi^{(i+1)}_{ { }_{V_s}}, \pi^{(i)}_{ { }_{V_s}}}         \\
& = X_{\pi^{(1)}}\prod_{s=1}^r \prod_{i=1}^{j_s} 
      M_{\pi^{(i+1)}_s, \pi^{(i)}_s}
  = X_{\pi^{(1)}} \prod_{s=1}^r M_{c_s}.
\end{align*}

Putting together the sign and the computation for $M_c$ yields 
Equation \eqref{eq.equality.M}.
\end{proof}

\vspace{6pt}

\begin{theorem}  \label{thm:1210}
For every 
$\pi \in \sqcup_{n=1}^{\infty} ( NC(n) \setminus \{ 1_n \} )$
one has:
\begin{equation}   \label{eqn:1210a}
S(X_\pi)= \sum_{c \in \EC (\pi,1_n)}
\, (-1)^{ | \Blocksplus (c) | }  M_c.
\end{equation}
\end{theorem}

\begin{proof}
The proof is by induction on $| \pi |$. 
For the base case: consider a $\pi$ in some $NC(n)$, such that 
$|\pi|=2$.  In this case we know that $S(X_\pi) = -X_\pi$. On the 
other hand, the set $\EC ( \pi , 1_n )$ consists of only
one chain, namely $c=(\pi, 1_n)$, which has 
$| \Blocksplus (c) | = 1$ and $M_c = M_{\pi, 1_n} = X_{\pi}$;
hence the right-hand side of Equation (\ref{eqn:1210a}) also comes 
out as $-X_{\pi}$, as required.

For the inductive step we fix a $j \geq 3$, we assume that the 
formula (\ref{eqn:1210a}) holds for every 
$\sigma \in \sqcup_{n=1}^{\infty} ( NC(n) \setminus \{ 1_n \} )$ 
with $|\sigma|< j$, and we prove that the same formula also holds 
for a $\pi$ with $|\pi|=j$. 

By the Bogoliubov recursion indicated in Equation \eqref{eqn:123b}, 
we have
\begin{align}
S(X_\pi)
& =-X_{\pi}-\sum_{\substack{\sigma \geq \pi \\ \pi\neq \sigma \neq 1_n}}
     \left( \prod_{V\in \sigma}S( X_{\pi|V})\right) X_\sigma \\
 &=-X_\pi - \sum_{\substack{\sigma=\{V_1,\dots,V_r\} \\ 1_n > \sigma > \pi}} 
     X_\sigma  \prod_{s=1}^r
     \left( \sum_{c_s\in \mathcal{EC}(\pi_{ {}_{V_s}},1_{|V_s|})} 
     (-1)^{| \Blocksplus (c_s) |} M_{c_s} \right), \label{ecu1}
\end{align}
where for the latter equality we used the induction hypothesis on
$S(X_{\pi_{ {}_{V_s}}})$, for each $V_s\in \sigma$. Finally, from the 
bijection in Lemma \ref{lem.bijection.antipode}, equation \eqref{ecu1} can 
be concisely written as 
\begin{equation}
 \label{eqn:1321}
-X_\pi + \sum_{\substack{\sigma=\{V_1,\dots,V_r\} \\ 1_n > \sigma > \pi}} \sum_{\substack{c\in \widetilde{\EC}_\sigma }}X_{\pi^{(1)}} (-1)^{| \Blocksplus (c) |} M_{c},
\end{equation}
where the notation $\widetilde{\EC}_\sigma$ is just to acknowledge that the set
$\widetilde{\EC}$ from Lemma \ref{lem.bijection.antipode} depends on the partition $\sigma$.

The conclusion follows from observing that the sum in \eqref{eqn:1321}
is a sum over all chains in $\EC(\pi,1_n)$ and thus coincides with the 
right hand side of \eqref{eqn:1210a}. Indeed, given a chain $c\in \EC(\pi,1_n)$, 
we either have $c=(\pi,1_n)$, in which case we get the term $-X_\pi$, or else
we have $c=(\pi^{(k)},\pi^{(k-1)},\ldots,\pi^{(1)},1_n)$ with $k \geq 2$ 
and $\pi^{(1)} = \sigma$ for some $1_n > \sigma > \pi$, implying that 
$c\in \widetilde{\EC}_\sigma$.
\end{proof}

\vspace{6pt}

\begin{example}   \label{example:1213}
In continuation of the last paragraph of Example \ref{example:128}, let 
$t_n$ denote the number of terms in the cancellation-free summation giving
$S( X_{0_n} )$, $n \geq 1$.  In view of Theorem \ref{thm:1210}, $t_n$ can 
also be viewed as $| \EC (0_n, 1_n ) |$, the number of efficient chains 
from $0_n$ to $1_n$ in $NC(n)$.  

We will derive a recursion satisfied by the numbers $t_n$.  It is possible
(left as exercise to the reader) to do so by examining the method of proof 
used for Theorem \ref{thm:1210}, and by extracting out of it a recursion 
among the cardinalities of the sets $\EC (0_n, 1_n)$.  Here we will take 
the alternative path of getting the desired recurrence for the $t_n$'s via 
a direct analysis of the Bogoliubov formula (\ref{eqn:123b}), which says that
\[ 
S( X_{0_n} ) = - X_{0_n} - 
        \sum_{ \substack{\sigma \in NC(n), \\ 0_n < \sigma < 1_n} }
S( M_{0_n, \sigma} ) \, X_{\sigma}.
\]
Every monomial $M_{0_n, \sigma}$ is equal, by definition, to
$\prod_{W \in \sigma} X_{0_{{ }_{|W|}}}$.  Since $S$ is multiplicative,
we thus find that
\begin{equation}   \label{eqn:1213a}
S( X_{0_n} ) = - X_{0_n} - 
        \sum_{ \substack{\sigma \in NC(n), \\ 0_n < \sigma < 1_n} }
\Bigl( \prod_{W \in \sigma} S( X_{0_{{ }_{|W|}}} ) \Bigr) \, X_{\sigma}.
\end{equation}

Suppose that on the right-hand side of (\ref{eqn:1213a}) we write
every $S( X_{0_{{ }_{|W|}}})$ as a cancellation-free sum of 
$t_{ { }_{|W|} }$ terms, then cross-multiply these sums.  For every 
$\sigma \in NC(n) \setminus \{ 0_n, 1_n \}$ we thus get
a sum of $\prod_{W \in \sigma} t_{ { }_{|W|} }$ terms, which
(very importantly) get to be also multiplied by an additional factor 
of $X_{\sigma}$.  Now, the latter factor of $X_{\sigma}$ appears only 
once in the whole expression on the right-hand side of (\ref{eqn:1213a}).
Multiplying with it will therefore prevent any cancellations with terms
that appear from the analogous discussion related to some other 
$\sigma ' \in NC(n) \setminus \{ 0_n, 1_n \}$.

Altogether, the discussion in the preceding paragraph shows how on 
the right-hand side of (\ref{eqn:1213a}) we arrive to a 
cancellation-free summation, where we can count the terms, in order 
to arrive to the conclusion that
\begin{equation}   \label{eqn:1213b}
t_{n}= 1 + 
\sum_{ \substack{\sigma \in NC(n), \\ 0_n < \pi < 1_n } } 
\ \prod_{W \in \sigma} t_{ { }_{|W|} }, \ \ n \geq 1
\end{equation}
(the empty sums appearing for $n=1$ and $n=2$ 
correspond to the fact that $t_1 = t_2 = 1$).
Equation (\ref{eqn:1213b}) is the recursion we wanted for the 
numbers $t_n$.  If we read the separate term of $1$ on the 
right-hand side as $t_1^n$, and we add on both side a term of 
$t_n$, we arrive to the nicer form 
\begin{equation}   \label{eqn:1213c}
2 t_{n}=  \sum_{\sigma \in NC(n)} 
\, \prod_{W \in \sigma} t_{ { }_{|W|} }, \ \ n \geq 2.
\end{equation}

Finally, Equation (\ref{eqn:1213c}) strongly suggests using
the functional equation of the $R$-transform from free probability 
(very closely related to free cumulants -- 
cf.~\cite[Lecture 16]{NiSp2006}), in order to find an equation 
satisfied by the generating function 
\begin{equation}   \label{eqn:1213d}
T(z) := \sum_{n=1}^\infty t_n z^n 
= z + z^2 + 4 z^3 + 25 z^4 + \cdots
\end{equation}
More precisely: let $\mu : \bC [X] \to \bC$ be the linear 
functional which has $\mu (1) = 1$ and has its sequence of free 
cumulants equal to $(t_n)_{n=1}^{\infty}$, hence has $R$-transform
$R_{\mu}(z)$ equal to the above $T(z)$.  From (\ref{eqn:1213c}) it follows
that the moment series $M_{\mu} (z) = \sum_{n=1}^{\infty} \mu (X^n) z^n$
is then equal to $2T(z) - z$.  The functional equation of the 
$R$-transform says that 
\[
R_{\mu} \bigl( z( 1 + M_{\mu} (z) \bigr) = M_{\mu} (z)
\ \mbox{ $\ $ (cf.~\cite[Remark 16.18]{NiSp2006}),}
\]
which becomes
\begin{equation}    \label{eqn:1213e}
T \bigl( z( 2T(z) - z + 1 ) \bigr) = 2T(z) -z.
\end{equation}
It is nicer to record this equation in terms of the series
\begin{equation}    \label{eqn:1213f}
U(z) := 2T(z) - z +1  = 1 + z + 2z^2 + 8z^3 + 50 z^4 + \cdots,
\end{equation}
which satisfies:
\begin{equation}    \label{eqn:1213g}
U \bigl( zU(z) \bigr) = (2-z) U(z) - 1.
\end{equation}

The first few $t_n$'s come out as
$1, 1, 4, 25, 206, 2060, 23920, 314065, 4582300, \ldots$
% 73393490, 1278859176, \ldots
This, unfortunately, doesn't seem to match the beginning of 
some sequence previously recorded in the research literature.
\end{example}

$\ $

\section{The Hopf algebra view on the inclusion 
$\cG \subseteq \cGtild$}

\noindent
There exists a result parallel to the above Theorem \ref{thm:115}, 
concerning the smaller group $\cG$ of multiplicative functions which 
was reviewed in Section 5.  This result was established 
in \cite{MaNi2010} and recognizes $\cG$ as the group of characters
$\bX ( \mbox{Sym} )$ of the Hopf algebra Sym of symmetric functions. 
In the present section we put into evidence a natural Hopf algebra 
homomorphism $\Psi : \cT \to \mbox{Sym}$, with $\cT$ as in 
Section 12, such that the dual group homomorphism 

\noindent
$\Psi^{*} : \bX ( \mbox{Sym} ) \to \bX ( \cT )$
corresponds in a canonical way to the inclusion of $\cG$
into $\cGtild$.

$\ $

\subsection{Review of the group isomorphism
\boldmath{$\cG \approx \bX ( \mathrm{Sym} )$}.}

\begin{notation-and-remark}    \label{def:131}
We use the incarnation of the Hopf algebra Sym as
\begin{equation}   \label{eqn:131a}
\mbox{Sym} = 
\bC \bigl[ Y_2, Y_3, \ldots , Y_n, \ldots \, \bigr]
\ \ \mbox{ (commutative algebra of polynomials)}
\end{equation}
where $Y_2, Y_3, \ldots , Y_n, \ldots$ are the so-called
{\em parking-function symmetric functions}.  
In the same spirit as for the considerations on the 
Hopf algebra $\cT$, we will also denote 
\begin{equation}    \label{eqn:131b}
Y_1 := \oneSym  \ \ \mbox{ (the unit of Sym).}
\end{equation}

A description of how the $Y_n$'s relate to other (more commonly
used) families of generators of Sym can e.g.~be found in 
\cite[Proposition 2.2]{St1997b}.  But here the only thing we need 
to know about the $Y_n$'s is how the comultiplication
$\Delta : \Sym \to \Sym \otimes \Sym$ 
operates on them.  The original motivation for featuring the 
$Y_n$'s in \cite{MaNi2010} was that the formula
giving $\Delta (Y_n)$ follows the same pattern as we saw 
in Section 11 in connection to the multiplication of 
free elements: one has
\begin{equation}  \label{eqn:131c}
\Delta (Y_n) = \sum_{\pi \in NC(n)} 
\Bigl( \prod_{V \in \pi} Y_{|V|} \Bigr) \otimes 
\Bigl( \prod_{W \in \Kr ( \pi )} Y_{|W|} \Bigr) ,
\ \ \forall \, n \geq 1.
\end{equation}

\noindent
[For instance $\Delta (Y_3)
= Y_3 \otimes Y_1^3 + 3 Y_1 Y_2 \otimes Y_1 Y_2 
+ Y_1^3 \otimes Y_3 
= Y_3 \otimes \oneSym + 3 Y_2 \otimes Y_2 
+ \oneSym \otimes Y_3$,
a sum of 5 terms, corresponding to the 
5 partitions in $NC(3)$.]

\vspace{6pt}

\noindent
We also mention that, when described in terms of the 
$Y_n$'s:

\noindent
-- The counit of Sym is the character $\ee : \Sym \to \bC$ 
uniquely determined by the requirement that 
$\ee (Y_n) = 0$ for all $n \geq 2$.

\noindent
-- The grading of Sym is determined by the fact that $Y_n$ 
has degree $n-1$, for every $n \geq 1$, with the usual 
follow-up defining the degree of a monomial 
$Y_{n_1} \cdots Y_{n_k}$ to be $n_1 + \cdots + n_k -k$. 
\end{notation-and-remark}

\vspace{6pt}

\begin{notation-and-remark}   \label{def:132}
In the framework of Notation \ref{def:131}, it is convenient
that for every $\pi \in \sqcup_{n=1}^{\infty} NC(n)$ we 
denote
\begin{equation}    \label{eqn:132a}
Y_{\pi} := \prod_{V \in \pi} Y_{|V|} 
\ \mbox{ (a monomial in the algebra Sym).}
\end{equation}
[For example, 
$\pi = \bigl\{ \{ 1,2,6 \}, \{ 3,4 \}, \{ 5 \}, \{ 7,8 \} \, \Bigr\}
\in NC(8)$ has $Y_{\pi} = Y_3 Y_2 Y_1 Y_2 = Y_3 Y_2^2$.]

The formula (\ref{eqn:131c}) describing comultiplication 
can then be written concisely as
\begin{equation}  \label{eqn:132b}
\Delta (Y_n) = \sum_{\pi \in NC(n)} 
Y_{\pi} \otimes Y_{\Kr ( \pi )}, \ \ n \geq 1.
\end{equation} 
By using the fact that $\Delta$ is an algebra 
homomorphism, it is easy (see
\cite[Lemma 3.3]{MaNi2010}) to extend (\ref{eqn:132b}) to 
\begin{equation} \label{eqn:132c}
\Delta (Y_{\sigma})=
\sum_{\substack{\pi \in NC(n) \\ \pi \leq \sigma}} 
Y_{\pi} \otimes Y_{\Kr_{\sigma} ( \pi )}, 
\ \ \forall \, n \geq 1 \mbox{ and } \sigma \in NC(n),
\end{equation}
where $\Kr_{\sigma} ( \pi )$ stands, as usual, for the 
relative Kreweras complement of $\pi$ in $\sigma$.  
\end{notation-and-remark}

\vspace{6pt}

\begin{remark}    \label{rem:133}
Now let us look at the group $\cG$ of multiplicative functions
and at the group $\bX ( \mathrm{Sym} )$ of characters 
of Sym.  For every $f \in \cG$ one can consider a character
$\widehat{\chi}_f \in \bX ( \mathrm{Sym} )$, 
defined by requiring that
\begin{equation}   \label{eqn:133a}
\widehat{\chi}_f (Y_n) = f( 0_n, 1_n ),
\ \ \forall \, n \geq 1.
\end{equation}
It is clear that the map
$\cG \ni f \mapsto \widehat{\chi}_f \in \bX ( \mathrm{Sym} )$
is bijective, and it is easy to check
that
\[
\widehat{\chi}_{f_1 * f_2} =
\widehat{\chi}_{f_1} * \widehat{\chi}_{f_2}, 
\ \ \forall \, f_1, f_2 \in \cG,
\]
where on the left-hand side we invoke the convolution operation 
on $\cG$, while on the right-hand side we use the convolution 
of characters of Sym.  Thus
$f \mapsto \widehat{\chi}_f$ gives a group isomorphism
$\cG \approx \bX ( \mathrm{Sym} )$, analogous to the isomorphism 
$\cGtild \approx \bX ( \cT )$ from Theorem \ref{thm:118}.
\end{remark}

$\ $

\subsection{The surjective homomorphism
\boldmath{$\Psi : \cT \to \mathrm{Sym}$}.}

$\ $

\noindent
Consider now the Hopf algebra $\cT$ from Section 12 and
recall that $\cT$ enjoys a universality property, stated 
in (\ref{eqn:111c}), which makes it very easy to define 
unital algebra homomorphisms having $\cT$ as domain.  
We use that to make the following definition.

\vspace{6pt}

\begin{definition}    \label{def:134} 
We let $\Psi : \cT \to \Sym$ be the unital algebra 
homomorphism defined by using the universality property 
(\ref{eqn:111c}) and the requirement that 
\begin{equation}  \label{eqn:134a}
\Psi (X_{\pi}) = Y_{\Kr ( \pi )}
= \prod_{W \in \Kr ( \pi )} Y_{|W|}, 
\ \ \forall \, \pi \in \sqcup_{n=1}^{\infty} NC(n).
\end{equation}
Note: in order for the universality property of $\cT$
to apply, the right-hand side of (\ref{eqn:134a}) must
be equal to $\oneSym$ whenever $\pi = 1_n$ for some
$n \geq 1$.  This is indeed the case, since 
$\Kr (1_n) = 0_n$ and $Y_{0_n} = Y_1^n = \oneSym$.
\end{definition}

\vspace{6pt}

\begin{remark}   \label{rem:134b}
For every $n \geq 1$, the definition of $\Psi$ gives 
$\Psi ( X_{0_n} ) = Y_{1_n} = Y_n$.
This immediately implies that the homomorphism $\Psi$ is 
surjective.

The point about $\Psi$ is that it also 
respects the coalgebra structure, as we show next.
\end{remark}

\vspace{6pt}

\begin{theorem}    \label{thm:135}
The map $\Psi$ introduced in Definition \ref{def:134} is a 
homomorphism of graded connected bialgebras.
\end{theorem}

\begin{proof}  We have to check that $\Psi$ respects:
(i) comultiplications; (ii) counits; (iii) gradings
on the Hopf algebras $\cT$ and Sym.  

\vspace{6pt}

\noindent
{\em For (i):} we have to verify the equality
\begin{equation}   \label{eqn:135c}
\DeltaSym \circ \Psi = ( \Psi \otimes \Psi ) \circ \DeltaT,
\end{equation}
where $\DeltaSym$ and $\DeltaT$ are the comultiplications 
of Sym and of $\cT$, respectively.  Since both sides of 
(\ref{eqn:135c}) are unital algebra homomorphisms from $\cT$ 
to $\mathrm{Sym} \otimes \mathrm{Sym}$), it suffices to 
check their agreement on a generator $X_{\pi}$ of $\cT$, with
$\pi \in NC(n) \setminus \{ 1_n \}$ for some $n \geq 1$.

Let us then pick an $X_{\pi}$ as mentioned above, plug it 
into the left-hand side of (\ref{eqn:135c}), and compute:
\begin{align*}
\bigl( \DeltaSym \circ \Psi \bigr) (X_\pi)
& = \DeltaSym ( Y_{\Kr(\pi)} )                       \\
& = \sum_{\rho \leq \Kr(\pi)} 
    Y_{\rho} \otimes Y_{\Kr_{ \Kr(\pi)} (\rho) } 
    \mbox{ $\ $ (by Equation (\ref{eqn:132c})). }
\end{align*}

We next do the same on
the right-hand side of (\ref{eqn:135c}):
\begin{align*}
\bigl( ( \Psi \otimes \Psi ) \circ \DeltaT \bigr) (X_\pi)
& = ( \Psi \otimes \Psi ) \bigl( \DeltaT ( X_{\pi} ) \bigr)   \\
& = ( \Psi \otimes \Psi ) \Bigl(  \sum_{\sigma \geq \pi} 
       \bigl( \prod_{W \in \sigma} X_{\pi_W} \bigr) 
       \otimes X_{\sigma} \Bigr)                             \\
&= \sum_{\sigma \geq \pi} 
         \bigl( \prod_{W \in \sigma} \Psi(X_{\pi_W}) \bigr) 
         \otimes \Psi (X_{\sigma})                            \\
&= \sum_{\sigma \geq \pi} 
         \bigl( \prod_{W \in \sigma} Y_{\Kr(\pi_W)} \bigr) 
         \otimes Y_{\Kr( \sigma )},
\end{align*}
with the relabeled-restrictions $\pi_{{ }_W}$ as in 
Notation \ref{def:12}.  In the latter summation over $\sigma$:
when we put together the Kreweras complements of all the
partitions $\pi_{{ }_W}$ with $W$ running in $\sigma$, what
comes out is the relative Kreweras complement of $\pi$ in 
$\sigma$.  Thus the conclusion for the right-hand side of
(\ref{eqn:135c}) reads:
\begin{equation}   \label{eqn:135d}
\bigl( ( \Psi \otimes \Psi ) \circ \DeltaT \bigr) (X_\pi)
= \sum_{\sigma \geq \pi}         
Y_{\Kr_{\sigma} (\pi)} \otimes Y_{\Kr(\sigma)}.
\end{equation}

In order to reconcile the results of our calculations on the 
two sides of (\ref{eqn:135c}), we perform the change of 
variable $\rho := \Kr_{\sigma} ( \pi )$ on the right-hand 
side of (\ref{eqn:135d}).  It fits perfectly to invoke here 
the considerations on relative Kreweras complements from 
\cite[Lecture 18]{NiSp2006}, and specifically Lemma 18.9 from
that lecture, which tells us that:
\[
\left\{
\begin{array}{l}
\mbox{if $\sigma$ runs in the interval $[ \pi , 1_n ]$ of $NC(n)$,} \\
\mbox{then $\rho = \Kr_{\sigma} ( \pi )$ runs (bijectively) in the 
      interval $[0_n, \Kr ( \pi ) ]$ of $NC(n)$,  }                 \\
\mbox{ and one has the relation 
       $\Kr ( \sigma ) = \Kr_{\Kr ( \pi )} ( \rho )$.}
\end{array}   \right.
\]
The change of variable from $\sigma$ to $\rho$ thus transforms
(\ref{eqn:135d}) into
\[
\bigl( ( \Psi \otimes \Psi ) \circ \DeltaT \bigr) (X_\pi)
= \sum_{\rho \leq \Kr ( \pi )}         
Y_{\rho} \otimes Y_{\Kr_{\Kr ( \pi )} (\rho)};
\]
this brings us to precisely the same expression as we found 
when we processed the left-hand side of (\ref{eqn:135c}).

\vspace{6pt}

\noindent
{\em For (ii):} we must check that
$\epsilonSym \circ \Psi = \epsilonT$, where 
$\epsilonSym$ and $\epsilonT$ are the counits of Sym
and of $\cT$, respectively.  Both 
$\epsilonSym \circ \Psi$ and $\epsilonT$ are unital
algebra homomorphisms from $\cT$ to $\bC$, hence it 
suffices to check that they agree on every $X_{\pi}$ with 
$\pi \in \sqcup_{n=1}^{\infty} NC(n) \setminus \{ 1_n \}$.
But for any such $\pi$ we have that
\begin{equation}   \label{eqn:135a}
( \epsilonSym \circ \Psi ) (X_{\pi}) = 0 
= \epsilonT (X_{\pi}).
\end{equation}
Indeed, the second equality (\ref{eqn:135a}) holds by 
the definition of $\epsilonT$; while for the first equality
(\ref{eqn:135a}) we write, 
for $\pi \in NC(n) \setminus \{ 1_n \}$:
\begin{align*}
\pi \neq 1_n 
& \Rightarrow \ \Kr ( \pi ) \neq 0_n 
  \ \Rightarrow \ \exists \, W_o \in \Kr ( \pi )
  \mbox{ with $|W_o| \geq 2$ and hence with } 
  \epsilonSym ( Y_{|W_o|} ) = 0  \\
& \Rightarrow \ ( \epsilonSym \circ \Psi ) (X_{\pi})
  = \epsilonSym ( Y_{Kr(\pi)} )
  = \prod_{W \in Kr(\pi)} \epsilonSym (Y_{|W|})=0.
\end{align*}

\vspace{6pt}

\noindent
{\em For (iii):} since $\Psi$ is a unital
algebra homomorphism, it will suffice to check that
\begin{equation}   \label{eqn:135b}
\mbox{deg}_{ \mathrm{Sym} } (\Psi ( X_{\pi} ))= 
\mbox{deg}_{\cT} ( X_{\pi} ), \ \ \forall 
\, n \geq 1 \mbox{ and } \pi \neq 1_n \mbox{ in } NC(n),
\end{equation}
where $\mbox{deg}_{ \mathrm{Sym} }$ and $\mbox{deg}_\cT$ 
denote the degree functions for Sym and $\cT$, respectively. 
And indeed, direct computation yields that both sides of 
(\ref{eqn:135b}) are equal to $| \pi | - 1$, where on the 
left-hand side we first write that 
$\mbox{deg}_{ \mathrm{Sym} } ( Y_{\Kr ( \pi )} ))= 
n- |Kr(\pi)|$, and then we invoke the known fact that
$|Kr(\pi)| = n+1 - | \pi |$.
\end{proof}

\vspace{6pt}

\begin{corollary}  \label{cor:136}
Let $\Psi : \cT \to \mathrm{Sym}$ be as above, 
and consider the groups of characters 
$\bX ( \Sym )$ and $\bX ( \cT )$ of the Hopf algebras 
Sym and $\cT$.

\vspace{6pt}

$1^o$ One has an injective group homomorphism 
$\Psi^{*} : \bX ( Sym ) \to \bX ( \cT )$
defined by
\begin{equation}   \label{eqn:136a}
\Psi^{*} ( \chi ) := 
\chi \circ \Psi, \ \ \chi \in \bX ( \Sym ).
\end{equation}

\vspace{6pt}

$2^o$ Consider the identifications 
$\bX ( \mathrm{Sym} ) = 
\{ \widehat{\chi}_f \mid f \in \cG \}$ from 
Remark \ref{rem:133} and
$\bX ( \cT ) = \{ \chi_g \mid g \in \cGtild \}$
from Theorem \ref{thm:118}.  In terms of these 
identifications, the group homomorphism $\Psi^{*}$ 
is just the inclusion of $\cG$ into $\cGtild$; that
is, one has
\begin{equation}   \label{eqn:136b}
\Psi^{*} ( \, \widehat{\chi}_f \, ) = \chi_f,
\ \ \forall \, f \in \cG .
\end{equation}
\end{corollary}

\begin{proof}  
The property of $\Psi^{*}$ of being a group homomorphism 
is a general Hopf algebra fact and the injectivity of 
$\Psi^{*}$ is implied, in particular, by (\ref{eqn:136b}).  
We are thus left to fix an $f \in \cG$ and to verify that 
the two characters
$\chi_f, \widehat{\chi}_f \circ \Psi \in \bX ( \cT )$
are equal to each other.  To that end, it suffices to also 
fix an $n \geq 1$ and a $\pi \in NC(n) \setminus \{ 1_n \}$, 
and to check that the two characters in question agree on 
the generator $X_{\pi}$ of $\cT$.  We know that
\[
\chi_f ( X_{\pi} ) = f( \pi , 1_n ) 
= \prod_{W \in \Kr ( \pi )} f( 0_{|W|}, 1_{|W|} ),
\]
where the second equality sign uses the fact that $f$ is 
multiplicative.  On the other hand, we have
\[
\bigl( \widehat{\chi}_f \circ \Psi \bigr) (X_{\pi})
= \widehat{\chi}_f ( Y_{\Kr ( \pi )} )
= \prod_{W \in \Kr ( \pi )} 
\widehat{\chi}_f ( Y_{|W|} )
= \prod_{W \in \Kr ( \pi )} f( 0_{|W|}, 1_{|W|} ),
\]
and this completes the required verification.
\end{proof}

\vspace{6pt}

\begin{remark}   \label{rem:137}
There was another subgroup of $\cGtild$ which played a 
significant role throughout this paper, namely the 
group $\cGtildctoc$ of semi-multiplicative functions of
cumulant-to-cumulant type.  The group $\cGtildctoc$ can
also be identified, in a natural way, as character group 
of a Hopf algebra $\cZ$, where the latter Hopf algebra
is some kind of ``truncation of $\cT$ to irreducible 
non-crossing partitions''.  Without going into details, 
we give here some highlights on what is $\cZ$ and of how 
it comes that $\bX ( \cZ ) \approx \cGtildctoc$.

As an algebra, $\cZ$ is just a commutative algebra of 
polynomials:
\begin{equation}    \label{eqn:137a}
\cZ := \bC \bigl[ \, Z_{\pi} \mid \pi \in \sqcup_{n=1}^{\infty} 
( \NCirr (n) \setminus \{ 1_n \}) \, \bigr],
\end{equation}
where for every $n \in \bN$ we use the shorthand notation
\[
\NCirr (n) := \{ \pi \in NC(n) \mid \pi \ \mbox{is irreducible} \}.
\]
Analogously to how we went when we defined
$\cT$ in Section 12.1, we also put 
\begin{equation}    \label{eqn:137b}
Z_{1_n} := \oneZ, \ \ \forall \, n \geq 1,
\end{equation}
and we record the universality property enjoyed by $\cZ$, 
which says:
\begin{equation}   \label{eqn:137c}
\left\{   \begin{array}{l}
\mbox{If $\cA$ is a unital commutative algebra 
over $\bC$ and we are given}   \\
\mbox{ $\ $ elements 
$\bigl\{ a_{\pi} \mid \pi \in \sqcup_{n=1}^{\infty} \NCirr (n) 
\bigr\}$ in $\cA$, with 
$a_{ { }_{1_n} } = 1_{ { }_{\cA} }$ for all $n \geq 1$,}  \\
\mbox{then there exists a unital algebra homomomorphism 
$\Phi : \cZ \to \cA$, uniquely}                         \\
\mbox{ $\ $ determined, such that 
$\Phi (Z_{\pi}) = a_{\pi}$ for all 
$\pi \in \sqcup_{n=1}^{\infty} \NCirr (n)$.}
\end{array}   \right.
\end{equation}

The universality property (\ref{eqn:137c}) yields in particular 
a recipe for how to construct characters of $\cZ$ 
(i.e.~unital algebra homomomorphisms from $\cZ$ to $\bC$).  
For every function $g \in \cGtildctoc$ let 
$\chicheck_g : \cZ \to \bC$ be the character defined via 
universality and the requirement that
\[
\chicheck_g (Z_{\pi}) = g( \pi, 1_n)
\mbox{ for every $n \geq 1$ and $\pi \in \NCirr (n)$.}
\]
It is easy to verify that in this way we get a bijective map
\begin{equation}   \label{eqn:137d}
\cGtildctoc \ni g \mapsto \chicheck_g \in \bX ( \cZ ),
\end{equation}
where $\bX ( \cZ )$ is the set of all characters of $\cZ$.

Now, in a nutshell, one has that: 
\begin{equation}    \label{eqn:137e}
\left\{ \begin{array}{c}
\mbox{ $\cZ$ also carries a coalgebra structure,} \\
\mbox{ which makes $\bX ( \cZ )$ become a group under convolution,}  \\
\mbox{ and makes the bijection (\ref{eqn:137d}) become a group isomorphism.}
\end{array}  \right.
\end{equation}
The statements made in (\ref{eqn:137e}) require a bunch of verifications 
which are pretty much a repeat of the arguments shown in connection 
to $\cT$ in Sections 12.1 and 12.2 of the paper.  We will leave these
(not difficult) verifications as an exercise to the reader, and only
provide here the definitions for the comultiplication, counit and 
grading on $\cZ$.

\vspace{6pt}

{\em Comultiplication:} this is the unital algebra homomorphism
$\Delta : \cZ \to \cZ \otimes \cZ$ defined via the universality 
property (\ref{eqn:137c}) and the requirement that for every 
$n \geq 1$ and $\pi \in \NCirr (n)$ we have
\begin{equation}   \label{eqn:137f}
\Delta (Z_{\pi}) =  
\sum_{ \substack{ \sigma \in NC(n), \\ \sigma\gg \pi} } 
\, \Bigl( \prod_{W \in \sigma} Z_{\pi_{{ }_W}} \Bigr) 
\otimes Z_\sigma,
\end{equation}
where ``$\gg$'' is in reference to the partial order from 
Notation \ref{def:14}.1.  We take a moment here to emphasize the 
importance of having $\sigma \gg \pi$ (rather than a plain 
``$\sigma \geq \pi$'') on the right-hand side of Equation 
(\ref{eqn:137f}): the condition $\sigma \gg \pi$ amounts 
precisely to the fact that $\pi_{{ }_W} \in \NCirr ( |W| )$ 
for every block $W \in \sigma$, which is crucial in order to 
be able to talk about the element $Z_{ \pi_{{ }_W} } \in \cZ$.

\vspace{6pt}

{\em Counit:} this is the character 
$\epsiloncheck := \chicheck_e \in \bX ( \cZ )$, where $e$ is the 
unit of $\cGtildctoc$.  Spelled explicitly, $\epsiloncheck$ is 
the character defined via the requirement that it has
$\epsiloncheck (Z_{\pi}) = e( \pi, 1_n) = 0$ for every $n \geq 1$
and $\pi \in \NCirr (n) \setminus \{ 1_n \}$.

\vspace{6pt}

{\em Grading:} this is obtained by postulating that every 
$Z_\pi$ has degree $|\pi|-1$, and hence that every monomial 
$Z_{\pi_1} \cdots Z_{\pi_k}$ has degree 
$|\pi_1| + \cdots + |\pi_k| -k$. 

\vspace{6pt}

We conclude the discussion around $\cZ$ by pointing out that 
one has a result analogous to the above Corollary \ref{cor:136},
concerning the inclusion of groups 
$\cGtildctoc \subseteq \cGtild$.  More precisely, let 
$\Phi : \cT \to \cZ$ be the unital algebra homomorphism obtained
by invoking the universality property (\ref{eqn:111c}) of $\cT$ 
in connection to the requirement that for every 
$\pi \in \sqcup_{n=1}^{\infty} NC(n)$ we have: 
\[
\Phi (X_\pi)= 
\begin{cases} Z_{\pi}, & \text{if $\pi$ is irreducible,} \\ 
0, & \text{if $\pi$ is reducible.} 
\end{cases} 
\]
It turns out $\Phi$ also respects the coalgebras structures on
$\cT$ and $\cZ$; this statement is analogous to the statement
of the above Theorem \ref{thm:135}, but has a much simpler (nearly
immediate, in fact) proof.  As a consequence of $\Phi$ being 
a Hopf algebra homomorphism, we get a group homomorphism 
$\Phi^{*} : \bX ( \cZ ) \to \bX ( \cT )$, defined by putting
$\Phi^{*} ( \chicheck ) := \chicheck \circ \Phi$, 
for $\chicheck \in \bX ( \cZ )$.  Directly from definitions 
it follows that for every $g \in \cGtildctoc$ one has
\[
\chicheck_g \circ \Phi = \chi_g,
\ \mbox{ or in other words that
$\Phi^{*} ( \chicheck_g ) = \chi_g$,}
\]
where $\chicheck_g \in \bX ( \cZ )$ is as above, while 
$\chi_g \in \bX ( \cT )$ is picked from Theorem \ref{thm:118}.  
Thus, when the canonical identifications 
$\cGtild \approx \bX ( \cT )$ and $\cGtildctoc \approx \bX ( \cZ )$ 
are considered, $\Phi^{*}$ is just the inclusion of $\cGtildctoc$
into $\cGtild$.
\end{remark}

$\ $

\begin{center}
{\bf Appendix.}
\end{center}

\noindent
Consider the framework and notation used in 
(\ref{eqn:109a}) of Remark \ref{rem:109}, where 
we specialize $x_1 = \cdots = x_n =: x$, with 
$x \in \cM$ picked such that $\varphi (x) = 1$. 
This appendix shows the output of some computer 
calculations which check the difference of the 
quantities on the two sides of (\ref{eqn:109a}),
\[
\rho_n( x y, \ldots , x y) -  \sum_{\pi \in NC(n)} 
\rho_{\pi} ( x, \ldots , x ) \cdot 
\rho_{ { }_{\Kr (\pi)} } (y, \ldots , y) = ?
\]
for $5 \leq n \leq 8$.  As mentioned in 
Remark \ref{rem:109}, the above difference is equal 
to $0$ for $n \leq 4$.

For every $n \geq 1$, we use the shorthand notation
$\rho_n (x) := \rho_n (x, \ldots , x)$ and
$\rho_n (y) := \rho_n (y, \ldots , y)$.

Since $\rho_1 (x) = \varphi (x) =1$ and 
$\rho_1 (y) = \varphi (y) = 1$,
in the formulas listed below we have omitted the 
occurrence of the powers of $\rho_1 (x)$ and $\rho_1 (y)$.
Putting in these powers would make the various products 
appearing there become homogeneous (for instance for $n=5$, 
the product $\rho_2(x) \, \rho_2(y)$ would become
$\rho_1(x)^3\rho_2(x) \, \rho_1(y)^3\rho_2(y)$, homogeneous of
degree $5$ with respect to each of $x$ and $y$).

$\ $

\noindent
{\bf \boldmath{$n=5$.}}
$\rho_5(xy, \ldots , xy) - 
\sum_{\pi \in NC(5)} \rho_\pi (x, \dots ,x) 
                     \rho_{\text{Kr}(\pi)}(y, \ldots , y)
= -\frac{1}{12} \rho_2(x) \, \rho_2(y)$.

$\ $

\noindent
{\bf \boldmath{$n=6$.}}
$\rho_6 (xy, \ldots , xy) - 
\sum_{\pi \in NC(6)} \rho_\pi (x, \dots ,x) 
                     \rho_{\text{Kr}(\pi)}(y, \ldots , y)$
\[                    
=  -\frac{1}{4} \rho_2(x)^2 \, \rho_2(y) 
  - \frac{1}{4} \rho_2(x)  \, \rho_2(y)^2  
  - \frac{1}{3} \rho_2(x) \, \rho_3(y)
  - \frac{1}{3} \rho_3(x)  \, \rho_2(y).
\]

$\ $

\noindent
{\bf \boldmath{$n=7$.}}
$\rho_7 (xy, \ldots , xy) - 
\sum_{\pi \in NC(7)} \rho_\pi (x, \dots ,x) 
                     \rho_{\text{Kr}(\pi)}(y, \ldots , y)$
\begin{align*}
= 
&  - \rho_3(x) \, \rho_3(y) 
   - \frac{4}{3} \rho_2(x)^2 \, \rho_2(y)^2 
   + \frac{7}{180}\rho_2(x) \, \rho_2(y)        \\
& - \frac{19}{12} \rho_2(x)^2 \, \rho_3(y) 
  - \frac{19}{12} \rho_3(x) \, \rho_2(y)^2 
  - \frac{3}{4} \rho_2(x) \, \rho_4(y)  
  - \frac{3}{4} \rho_4(x) \, \rho_2(y)      \\
& - \frac{17}{12} \rho_2(x) \rho_3(x) \, \rho_2(y) 
  - \frac{17}{12} \rho_2(x) \, \rho_2(y)\rho_3(y) 
  - \frac{1}{6} \rho_2(x)^3 \, \rho_2(y) 
  - \frac{1}{6} \rho_2(x) \, \rho_2(y)^3.
\end{align*}

$\ $

\noindent
{\bf \boldmath{$n=8$.}}
$\rho_8 (xy, \ldots , xy) - 
\sum_{\pi \in NC(8)} \rho_\pi (x, \dots ,x) 
                     \rho_{\text{Kr}(\pi)}(y, \ldots , y)$
\begin{align*}
= 
& + \frac{7}{30} \rho_2(x) \, \rho_3(y) 
  + \frac{7}{30} \rho_3(x) \, \rho_2(y) 
  + \frac{43}{180} \rho_2(x)\, \rho_2(y)^2      
  + \frac{43}{180} \rho_2(x)^2 \, \rho_2(y) \\
& -\frac{4}{3} \rho_2(x) \, \rho_5(y) 
  -\frac{4}{3} \rho_5(x) \, \rho_2(y)
  -\frac{4}{3} \rho_2(x) \, \rho_2(y)^2\rho_3(y) 
  -\frac{4}{3} \rho_2(x)^2\rho_3(x) \, \rho_2(y) \\
& -\frac{4}{3} \rho_2(x) \, \rho_3(y)^2 
  -\frac{4}{3} \rho_3(x)^2 \, \rho_2(y)
  -\frac{32}{3} \rho_2(x)^2 \, \rho_2(y)\rho_3(y) 
  -\frac{32}{3} \rho_2(x)\rho_3(x) \, \rho_2(y)^2  \\
& -\frac{8}{3} \rho_3(x) \, \rho_2(y)^3
  -\frac{8}{3} \rho_2(x)^3 \, \rho_3(y) 
  -\frac{9}{2} \rho_2(x)^2 \, \rho_4(y)
  -\frac{9}{2} \rho_4(x) \, \rho_2(y)^2      \\
& -\frac{20}{3} \rho_3(x) \, \rho_2(y)\rho_3(y) 
  -\frac{20}{3} \rho_2(x)\rho_3(x) \, \rho_3(y)
  -2 \rho_2(x)^2 \, \rho_2(y)^3
  -2 \rho_2(x)^3 \, \rho_2(y)^2       \\
& -3 \rho_2(x) \, \rho_2(y)\rho_4(y)
  -3 \rho_2(x)\rho_4(x) \, \rho_2(y)
  -2 \rho_3(x) \, \rho_4(y) 
  -2 \rho_4(x) \, \rho_3(y).
\end{align*}

$\ $

$\ $

\noindent
{\bf Acknowledgement.}  We are grateful to the anonymous 
referee for the careful reading of this long paper and for
several suggestions which led to an improved exposition, 
particularly in Sections 11 and 12.

$\ $

\end{document}